\documentclass[reqno]{amsart}
\usepackage{amsmath,amsthm,amssymb,amsfonts,color}
\usepackage{hyperref,comment}
\usepackage{todonotes}
\usepackage{esint}
\usepackage[latin1]{inputenc}

\usepackage[shortlabels]{enumitem}

\newcommand\bH{\mathbb{H}}

\newcommand\bR{\mathbb{R}}

\newcommand\bS{\mathbb{S}}

\newcommand\jb{\langle v \rangle}

\newcommand\cO{\mathcal{O}}

\newcommand\cL{\mathcal{L}}

\newcommand\cK{\mathcal{K}}

\newcommand\cQ{\mathcal{Q}}

\newcommand\sfA{\mathsf{A}}

\newcommand\sfK{\mathsf{K}}

\newcommand\sfL{\mathsf{L}}

\newcommand\sfG{\mathsf{\Gamma}}

 \theoremstyle{definition}

\newtheorem{theorem}{Theorem}[section]
\newtheorem{lemma}[theorem]{Lemma}
\newtheorem{corollary}[theorem]{Corollary}

\newtheorem{proposition}[theorem]{Proposition}

\newtheorem{definition}{Definition}[section]

\theoremstyle{remark}
\newtheorem{remark}[theorem]{Remark}

\newtheorem{assumption}[theorem]{Assumption}

\numberwithin{equation}{section}

\newcommand\sg{\mathsf{g}}

\begin{document}

\title[KFP  Equations with Specular Reflection Boundary Condition]{kinetic Fokker-Planck  and Landau Equations with Specular Reflection Boundary Condition}


\begin{abstract}
We establish existence of finite energy weak solutions to the
 kinetic Fokker-Planck equation and the linearized Landau equation near Maxwellian, in the presence of
specular reflection boundary condition for general domains. Moreover,  by using a method of reflection and the $S_p$ estimate of \cite{DY_21a},
we prove regularity in the kinetic Sobolev spaces $S_p$ and anisotropic H\"older spaces
for such weak solutions. Such $S_p$ regularity leads to the uniqueness of weak solutions.
\end{abstract}

\author[Hongjie Dong]{Hongjie Dong}
\address[H. Dong]{Division of Applied Mathematics, Brown University, 182 George Street, Providence, RI 02912, USA}
\email{Hongjie\_Dong@brown.edu }
\thanks{H. Dong was partially supported by the Simons Foundation, grant no. 709545, a Simons fellowship, grant no. 007638, and the NSF under agreement DMS-2055244.}

\author[Yan Guo]{Yan Guo}
\address[Y. Guo]{Division of Applied Mathematics, Brown University, 182 George Street, Providence, RI 02912, USA}
\email{Yan\_Guo@brown.edu }
\thanks{Y. Guo's research is supported in part by NSF DMS-grant 2106650.}

\author[Timur Yastrzhembskiy]{Timur Yastrzhembskiy}
\address[T. Yastrzhembskiy]{Division of Applied Mathematics, Brown University, 182 George Street, Providence, RI 02912, USA}
\email{Timur\_Yastrzhembskiy@brown.edu}

\maketitle

\tableofcontents

\section{Introduction and Main results}
                            \label{section 1}
\subsection{Introduction}
We consider the following generalized kinetic Fokker-Planck equation in a bounded domain $\Omega$ with the specular
boundary condition:
\begin{equation}
			\label{eq1.0}
\begin{aligned}
\partial_{t}f+v\cdot\nabla_{x}f-\partial_{v_i} (a^{ij}(t,x,v)\partial_{v_j} f) &+ b (t, x, v) \cdot \nabla_v f
= \sg \\
&\quad f(t,x,v)    = f(t,x,R_{x}v),  (x, v) \in \gamma_{-}.
\end{aligned}
\end{equation}
Here
$$
    R_{x}v = v- 2 (n_{x}\cdot v)n_{x}
$$
is the specular reflected velocity,
 and $n_{x}$ is the outward unit normal vector at $x\in\partial\Omega$,
and
$$\gamma_{\pm} =  \{(x, v): x \in \partial \Omega,  \pm n_x \cdot v > 0 \}$$
 is the outgoing/incoming set.
In this paper, we will study two important cases of the generalized Fokker-Planck equations: the celebrated Kolmogorov-Fokker-Planck equation, which is given by Eq. \eqref{eq1.0} with $a^{i j}=\delta_{ij}$,
and the linearized Landau equation, which we introduce below.

A fundamental  model in plasma physics  is the nonlinear kinetic Landau equation with the specular reflection boundary condition  given by
\begin{equation}
\begin{aligned}
			\label{eq1.0.0}
    &	\partial_t F + v \cdot \nabla_x F  = \mathcal{Q} \,  [F, F] \, \, \text{in} \, \, (0, T) \times \Omega \times \bR^3, \\
     &F (0, x, v) = F_0 (x, v), (x, v) \in \Omega \times \bR^3, \, \,  F (t, x, v) = F (t, x, R_x v), (x, v) \in \gamma_{-}\notag,
\end{aligned}
\end{equation}
where
 $\mathcal{Q}$ is the Landau (Fokker-Planck) collision operator with Coulomb interaction defined as
\begin{align}
						\label{eq11.10}
   &	\mathcal{Q} \,  [F_1, F_2] (v) =  \nabla_v \cdot \int_{\bR^3} \Phi (v-v') [F_1 (v') \nabla_v F_2 (v) - F_2 (v) (\nabla_v F_1) (v')] \, dv',\\
&
	\Phi (v) = \bigg(I_3 - \frac{v}{|v|} \otimes  \frac{v}{|v|}\bigg) |v|^{-1}\notag.
\end{align}
Let $\mu (v) : = \pi^{-3/2} e^{-|v|^2}$ be the Maxwellian, which is a steady state of Eq. \eqref{eq1.0.0}. We rewrite the equation near the Maxwellian by replacing $\mathcal{Q} \,  [F, F]$ with $\mathcal{Q} \,  [\mu + \mu^{1/2} f, \mu + \mu^{1/2} f]$:
\begin{equation}
\begin{aligned}
					\label{11.1.0}
	&\partial_t f + v \cdot \nabla_x f + \sfL f -  \sfG [f, f] =  0 \, \, \text{in} \, \, (0, T) \times \Omega \times \bR^3, \\
	 &
	f (0, x, v) = f_0 (x, v), (x, v) \in \Omega \times \bR^3,  \, \,
		 f (t, x, v) = f (t, x, R_x v), \, (x, v) \in \gamma_{-}\notag,
\end{aligned}
\end{equation}
where
\begin{equation}
			\label{eq11.1.2}
\begin{aligned}
&\mathsf{L} = - \mathsf{A} - \mathsf{K}, \quad \sfA f = \mu^{-1/2} \cQ \,  [\mu, \mu^{1/2} f], \quad \sfK f =    \mu^{-1/2} \cQ \,  [\mu^{1/2} f, \mu], \\
& \mathsf{\Gamma} [g, f] =  \mu^{-1/2} \cQ \, [\mu^{1/2} g, \mu^{1/2} f].
\end{aligned}
\end{equation}
We consider the linear version of Eq. \eqref{11.1.0} given by
\begin{equation}
\begin{aligned}
					\label{11.1}
	&\partial_t f + v \cdot \nabla_x f + \sfL f -  \mathsf{\Gamma} [g, f] =  0 \, \, \text{in} \, \, (0, T) \times \Omega \times \bR^3, \\
	 &
	f (0, x, v) = f_0 (x, v), (x, v) \in \Omega \times \bR^3,  \, \,
		 f (t, x, v) = f (t, x, R_x v), \, (x, v) \in \gamma_{-}.
\end{aligned}
\end{equation}
Such an equation is useful for proving the global in time well-posedness of \eqref{11.1} for  sufficiently small initial data (see, for example, \cite{DGO}, \cite{GHJO_20}, \cite{KGH}).
Furthermore, we can rewrite   \eqref{11.1} in the divergence form
\begin{equation}
\begin{aligned}
			\label{eq11.1}
	&\partial_t f + v \cdot \nabla_x f = \nabla_{v} \cdot (\sigma_G \nabla_{v} f) + a_g \cdot \nabla_v f +\overline{K}_g f  \, \, \text{in} \, \, (0, T) \times \Omega \times \bR^3,\\
	 &
	f (0, x, v) = f_0 (x, v), (x, v) \in \Omega \times \bR^3,  \, \,
		 f (t, x, v) = f (t, x, R_x v), (x, v) \in  \gamma_{-},
\end{aligned}
\end{equation}
where
\begin{align}
		\label{eq11.5}
	 & \sigma = \Phi \ast \mu, \quad	\sigma_G =  \sigma + \Phi \ast (\mu^{1/2} g),
\quad
	a_g^i = - \Phi^{i j} \ast (v_j \mu^{1/2} g + \mu^{1/2} \partial_{v_j} g),\\
& \label{eq11.2}
	\overline{K}_g  = \sfK  + J_g,\\
&\label{eq11.2.1'}
	J_g f = \partial_{v_i} (\sigma^{i j} v_j) f - \sigma^{i j} v_i v_j f \\
&\qquad - \partial_{v_i} \big(\Phi^{i j} \ast (\mu^{1/2} \partial_{v_j} g)\big) f +  \big(\Phi^{i j} \ast (v_i \mu^{1/2} \partial_{v_j} g)\big) f\notag.
\end{align}
See the details in \cite{DGO}, \cite{G_02}, \cite{KGH}.
To establish the existence of the finite energy strong solution to the problem \eqref{eq11.1} (see Definition \ref{definition 11.0}), we work with the equation
\begin{equation}
\begin{aligned}
			\label{eq1.18}
	&\partial f + v \cdot \nabla_x f  = \nabla_{v} \cdot (\sigma_G \nabla_{v} f) + \nu \Delta_v f + a_g \cdot \nabla_v f - \lambda f + h \, \, \text{in} \, \, (0, T) \times \Omega \times \bR^3,\\
&	f (0, x, v) = f_0 (x, v), \,   (x, v) \in \Omega \times \bR^3,  	 f_{-} (t, x, v) = f_{+} (t, x, R_x v),  (x, v) \in \gamma_{-},
\end{aligned}
\end{equation}
where $\nu > 0, \lambda \ge 0$. We call this equation \textit{simplified viscous linearized Landau equation.}

While such boundary problems play an important role in many applications, there have been very limited study due to possible singularity forming from the grazing set (see  \cite{G_95}).
The paper \cite{HJV_14} initiated the study of the regularity of the boundary-value problems for generalized kinetic Fokker-Planck equations. For the related works, we refer the reader to \cite{AM},  \cite{HJJ_18}, \cite{HJV_19}, \cite{LN_21}.
The boundary-value problem for the Landau equation is considered in the articles  \cite{DLSS}, \cite{DGO}, \cite{GHJO_20}.
Our paper serves as a foundation of the linear theory for the nonlinear Landau equation used in the works \cite{DGO}, \cite{GHJO_20}.
The article  is organized as follows.  Section \ref{subsection 1.2} contains the necessary notation. In Section \ref{section 1.3}, we present the main results of this paper. We explain the key ideas of the proof  in Section \ref{subsection 1.4}.
The proof of the main results is divided into 3 sections: Sections \ref{section 2}, \ref{section 3}, and \ref{section 4}.

\subsection{Notation}
To state precisely our results, we now introduce necessary notation as follows.
 In  this subsection, $G \subset \bR^7$ is an open set, $\alpha \in (0, 1]$, $p \in (1, \infty]$, $\theta \in [0, \infty)$, and $T \in (0, \infty)$.
				\label{subsection 1.2}
\begin{itemize}
    \item
 Geometric notation:
\begin{align}
      &  z = (t, x, v), t \in \bR, x, v \in \bR^3,
      \quad B_r  (x_0) = \{\xi \in \bR^3: |\xi - x_0| < r\},\notag\\
&
   \Omega \subset \bR^3 \, \, \text{is a bounded domain}, \, \,  \Omega_{r} (x)  = \Omega \cap B_{r} (x), \, \, \bR^3_{\pm} = \{x \in \bR^3: \pm x_3 > 0\},\notag\\
&
\label{eq1.2.0}
  \mathbb{H}_{\pm} = \{(x, v) \in \bR^3_{\pm} \times \bR^3\},
\, \,  \mathbb{H}_{\pm}^T = \{z \in (0, T) \times  \mathbb{H}_{\pm}\}, \\
&
    \bR^7_T = \{z \in (0, T) \times \bR^6\},  \quad \Sigma^T = (0, T)\times \Omega \times \bR^3,  \notag\\
&
    \gamma_{\pm} = \{(x, v): x \in \partial \Omega, \pm n_x \cdot v > 0 \},
    \, \, \gamma_0 = \{(x, v): n_x \cdot v = 0\},
   \, \,  \Sigma^T_{\pm} = (0,T) \times \gamma_{\pm}\notag.
\end{align}

\item \textit{Matrices.} By $I_3$, we denote the $3 \times 3$ identity matrix and we set $\text{Sym} (\delta), \delta \in (0, 1)$ to be the space of $3 \times 3$ symmetric matrices $a$ such that
$$
        \delta   |\xi|^2 \le a^{i j} \xi_i \xi_j \le \delta^{-1} |\xi|^2, \quad \forall \xi \in \bR^3.
$$

\item Traces: $f_{\pm} = f|_{\gamma_{\pm}}$.

\item
\begin{equation*}
\begin{aligned}
 &   Y f = \partial_t f + v \cdot \nabla_x f.
\end{aligned}
\end{equation*}

\item Function spaces.

\begin{itemize}
\item[--] $C^1_0 (\overline{G})$ - the set of all continuously differentiable functions in $\overline{G}$ that vanish for large $z$.

\item[--] \textit{Anisotropic H\"older space.}
For  an open set $D \subset \bR^6$,
by $C^{\alpha/3, \alpha}_{x, v} (\overline{D})$,
we denote the set of bounded functions
$f = f (x, v)$ such
that
\begin{align*}
        &[f]_{C^{\alpha/3, \alpha}_{x, v} (\overline{D})  } \\
        &:= \sup_{ (x_i, v_i) \in \overline{D}:  (x_1, v_1) \neq (x_2, v_2) } \frac{|f (x_1, v_1)- f (x_2, v_2)|}{(|x_1-x_2|^{1/3}+|v_1-v_2|)^{\alpha}} < \infty.
\end{align*}
Furthermore,
\begin{equation}
			\label{1.2.0}
    \|f\|_{C^{\alpha/3, \alpha}_{ x, v } (\overline{D})  } : = \|f\|_{  L_{\infty} (D) }  +  [f]_{C^{\alpha/3, \alpha}_{x, v} (\overline{D})  }.
\end{equation}

\item[--]\textit{Weighted Lebesgue space.}
For
a  locally integrable function $w(x, v)$ such that $w > 0$ a.e.,
by $L_{p} (G, w)$
we denote the set of all Lebesgue measurable functions $u$
such that
$$
    \|u\|_{ L_p (G, w) }:= \|w^{1/p} u\|_{L_p (G)} < \infty.
$$
In particular, for
$$
    \jb  :=  (1+|v|^2)^{1/2},
$$
 we set
$$
    L_{p, \theta} (G)  = L_p (G, \jb^{\theta}).
$$

\item \textit{Sobolev spaces.}
Let $\mathbb{H}^1_p (G): = \{f \in L_p (G): \nabla_v f \in L_p (G)\}$ be the Banach space equipped with the norm
\begin{equation}
			\label{1.2.2}
	\|f\|_{  \mathbb{H}^1_p (G)  } : = \||f| +|\nabla_v f|\|_{ L_p (G) }.
\end{equation}
Furthermore, $\mathbb{H}^{-1}_p (G)$ is the set of all functions $f$ on $G$ such that there exists   $f_i \in L_p (G), i = 0, 1, 2,  3$, such that
$$
    f = f_0 + \partial_{v_i} f_i.
$$
The $\mathbb{H}^{-1}_p (G)$-norm is given by
\begin{equation}
			\label{1.2.3}
   \|f\|_{\mathbb{H}^{-1}_p (G) }  = \inf \sum_{i = 0}^3 \|f_i\|_{ L_p (G)},
\end{equation}
where the infimum is taken over all such $f_i, i  = 0, 1, 2, 3$.

\item \textit{Kinetic (ultraparabolic) Sobolev spaces.}

By $S_{p, \theta} (G) = \{f \in L_{p, \theta} (G): Y f, \nabla_v f, D^2 f \in L_{p, \theta} (G)\}$
we mean  the Banach space  with the norm
\begin{equation}
			\label{1.2.4}
    \|f\|_{S_{p, \theta} (G)} = \||f| + |\nabla_v f| + |D^2_v f| + |Y f|\|_{ L_{p, \theta} (G) }.
\end{equation}
 We also denote $S_p (G): = S_{p, 0} (G)$.

Furthermore, we set $\mathbb{S}_p (G)$
to be the space of functions $f$
such that $f, \nabla_v f \in L_p (G)$, and
$Y f \in \mathbb{H}^{-1}_p  (G)$. The norm is defined as follows:
\begin{equation}
		\label{1.2.6}
    \|f\|_{\mathbb{S}_p (G)} = \|f\|_{ \bH^1_p (G)}
    + \|Y f\|_{ \mathbb{H}^{-1}_p (G)}.
\end{equation}
\end{itemize}

\item \textit{Space of initial values.}  For $p \in (1, \infty]$ and $\theta \ge 0$, by $\cO_{\theta, p}$ we denote the set of all functions $u$ on $\Omega \times \bR^3$ such that
  $$
    u, v \cdot \nabla_x u, \nabla_v u, D^2_v u \in L_{p, \theta} (\Omega \times \bR^3), \, \, u_{\pm} \in L_{\infty} (  \gamma_{\pm}, |v \cdot n_x|),
  $$
and
  $$
    u_{-} (x, v) = u_{+} (x, R_x v),  \, \, \text{a.e.}  \, \,  (x, v) \in \gamma_{-}.
  $$
  The norm is given by
\begin{equation}
			\label{1.2.7}
\begin{aligned}
	& \|u\|_{\cO_{ p, \theta} }  =  |u|_{\cO_{ p, \theta} } + \|u_{\pm}\|_{   L_{\infty} (  \gamma_{\pm}, |v \cdot n_x|)  }, \\
	&|u|_{\cO_{ p, \theta} }  := \||u| + |v \cdot \nabla_x u| + |\nabla_v u| + |D^2_v u|\|_{ L_{p, \theta} ( \Omega \times \bR^3 ) }.
\end{aligned}
\end{equation}
For $\theta = 0$, we set $\cO_{p} = \cO_{p, 0}$.

\item \textit{Constants.} By $N = N (\cdots)$  we mean a constant depending only on the parameters inside the
parentheses. The constants $N$ might change from line to line. Sometimes, when it is clear what parameters $N$  depend on, we omit  them.

\end{itemize}

\subsection{Main results}
															\label{section 1.3}
\subsubsection{Kinetic Fokker-Planck equation}
									\label{subsection 1.3.1}

Let $a$ be a measurable function taking values in the set of symmetric $d \times d$ matrices.
\begin{definition}
                    \label{definition 7.0}
 Let $ \theta \ge 0, T > 0$ be  numbers.
 We say that  $f$
 is a finite energy weak solution  to
\begin{equation}
    \begin{aligned}
                                    \label{7.0}
    &	Y f  - \nabla_v \cdot (a \nabla_v  f) + b \cdot \nabla_v f + \lambda f =   \sg, \quad z \in \Sigma^T,\\
& f_{-} (t, x, v) =  f_{+} (t, x, R_x v),\, \quad z \in   \Sigma^T_{-},\\
	& f (0, x, v) = f_0 (x, v), \quad (x, v)  \in  \Omega \times \bR^3,
\end{aligned}
\end{equation}
 if
 \begin{enumerate}[i)]
     \item \label{i}
 $f, \nabla_v f \in L_{2, \theta} (\Sigma^T)$,
 $\sg \in L_{2, \theta} (\Sigma^T)$,
 $ f_{\pm}
 \in
 L_{\infty} ( \Sigma^T_{\pm}, |v \cdot n_x|)$,
 $f (T, \cdot), f (0, \cdot) \in L_{2, \theta} (\Omega \times \bR^3)$,   \\

 \item \label{ii} $f_{-} (t, x, v) =  f_{+} (t, x, R_x v)$ a.e. on $\Sigma^T_{-}$,\\

 \item   \label{iii}
 for any $\phi \in C^1_0 (\overline{\Sigma^T})$,
 \begin{align}
                \label{7.0.0}
	&	-\int_{\Sigma^T}
    (Y \phi)    f \, dz
		 + \int_{\Omega \times \bR^3 }     \big(f (T, x, v) \phi (T, x, v)  - f (0, x, v) \phi (0, x, v)\big)\, dx dv \notag\\
 	 &
		+  \int_{\Sigma^T_{+}}     f_{+} \phi \, |v \cdot n_x| d\sigma dt
	-   \int_{\Sigma^T_{-}}   f_{-} \phi \, |v \cdot n_x|  d\sigma dt\\
	  &
	  + \lambda \int_{\Sigma^T}
	  f \phi \, dz
	  +\int_{\Sigma^T}  (a\nabla_{v} f) \cdot \nabla_{v} \phi  \,dz
	+\int_{\Sigma^T}  b \cdot  \nabla_{v} f  \phi  \,dz
	  = \int_{\Sigma^T}   \sg \phi   \, dz\notag.
	  \end{align}
\end{enumerate}	

Furthermore, we say that $f$ is a finite energy strong  solution to Eq. \eqref{7.0}
if $f$ is a finite energy weak solution, and, additionally, $f \in S_2 (\Sigma^T)$,
where $S_2$ is defined in \eqref{1.2.4}.
\end{definition}

We impose the following assumptions on the coefficients and the domain $\Omega$.

\begin{assumption}
                \label{assumption 2.1}
There exists $\delta \in (0, 1)$ such that
$\forall z$, $a (z) \in \text{Sym} (\delta)$,  i.e.,
\begin{equation}
			\label{eq2.1.0}
	\delta |\xi|^2  \le a^{i j} (z) \xi_i \xi_j \le \delta^{-1} |\xi|^2.
\end{equation}
 \end{assumption}

\begin{assumption}
                            \label{assumption 7.3}
There exists a constant $K > 0$  such that
\begin{equation}
			\label{eq2.3.0}
	\|b\|_{ L_{\infty} (\Sigma^T)}\le K.
\end{equation}
\end{assumption}

\begin{assumption}
                            \label{assumption 2.2}
There exists a constant $K > 0$  such that
\begin{equation}
			\label{eq2.2.0}
	\|\nabla_v a \|_{ L_{\infty} (\Sigma^T)}   \le K , \quad \nabla_v b \in L_{\infty} (\Sigma^T).
\end{equation}
\end{assumption}

We present three results on the existence, uniqueness, and the $S_p$-regularity of the finite energy weak solutions.

\begin{theorem}[Existence of finite energy weak solutions]
			\label{theorem 3.3}
Let $\Omega$ be a bounded $C^2$ domain.
Under Assumptions  \ref{assumption 2.1} -  \ref{assumption 2.2},
for any $T > 0$, $\theta \ge 0$,
there exists $\lambda_0 = \lambda_0 (\theta, \delta,  K) > 0$
such that for any $\lambda \ge \lambda_0$
and $\sg \in L_{2, \theta} (\Sigma^T) \cap L_{\infty} (\Sigma^T)$,
$f_0 \in L_{2, \theta} (\Omega \times \bR^3) \cap L_{\infty} (\Omega \times \bR^3) $,
Eq. \eqref{7.0} has a finite energy weak solution $f$, and, in addition,
\begin{equation}
                \label{7.2.1}
\begin{aligned}
         & \|f (T, \cdot)\|_{ L_{2, \theta} (\Omega \times \bR^3) }
+    \lambda^{1/2} \|f\|_{ L_{2, \theta} (\Sigma^T) }
     + \|\nabla_v f\|_{ L_{2, \theta} (\Sigma^T) }\\
  &   \le
       N    (\lambda^{-1/2} \|\sg\|_{ L_{2, \theta} (\Sigma^T) } + \|f_0\|_{  L_{2, \theta} (\Omega \times \bR^3)  }),\\
    &\max \{ \|f (T, \cdot)\|_{ L_{\infty} (\Omega \times \bR^3) }, \|f\|_{ L_{\infty} (\Sigma^T)}, \| f_{\pm}\|_{ L_{\infty} (\Sigma^T_{\pm}, |v \cdot n_x|)  }\} \\
    &
    \le  \lambda^{-1} \|\sg\|_{ L_{\infty} (\Sigma^T) } +  \|f_0\|_{  L_{\infty} (\Omega \times \bR^3)  },
    \end{aligned}
    \end{equation}
where $N = N (\delta, K, \theta) > 0$.
\end{theorem}

The following corollary is derived from Theorem \ref{theorem 3.3}  by using an exponential weight.
\begin{corollary}
One can take $\lambda_0 = 0$ in Theorem \ref{theorem 3.3}.
However, in that case, the constant $N$ also depends on $T$, i.e., $N = N (\delta, K, \theta, T)$.
\end{corollary}

We remark that the uniqueness of such weak solutions remains an open problem in general.
By using the method of reflection (see Section \ref{subsection 1.4}), we prove the uniqueness and higher regularity for Eq. \eqref{7.0} in the case when $a = I_3$.


\begin{theorem}[$S_2$ regularity of  a finite energy weak solution]
                \label{theorem 7.3}
Let $\Omega$ be a bounded $C^3$ domain, $T > 0, \theta  \ge 2$ be numbers,
and $\sg \in L_{2, \theta -2} (\Sigma^T), f_0 \in \cO_{2, \theta - 2}$ (see \eqref{1.2.7}), and $a=I_3$. Under Assumptions \ref{assumption 2.1} - \ref{assumption 7.3},
there exists $\lambda_0 = \lambda_0 (K, \theta, \Omega) > 0$ such that
for any $\lambda \ge \lambda_0$,
if $f$ is a finite  energy weak solution to Eq. \eqref{7.0} with parameter $\theta$,
 then, $f \in S_{2, \theta - 2} (\Sigma^T)$ and $f, \nabla_v f \in L_{7/3, \theta - 2} (\Sigma^T)$.
 Furthermore, we have
 \begin{equation}
                \label{7.3.0}
 \begin{aligned}
    \|f\|_{ S_{2, \theta-2} (\Sigma^T) }& + \||f| +|\nabla_v f|\|_{L_{7/3, \theta - 2} (\Sigma^T)} \\
&\le N (\|\sg\|_{ L_{2, \theta-2} (\Sigma^T) } + |f_0|_{  \cO_{2, \theta - 2} }\\
    &+   \|\nabla_v  f\|_{ L_{2, \theta} (\Sigma^T) } +  \|f\|_{ L_{2, \theta-1} (\Sigma^T) }),
 \end{aligned}
 \end{equation}
 where $N = N (K,  \Omega, \theta)$.
\end{theorem}

\begin{corollary}[Uniqueness of finite energy solutions]
                \label{corollary 7.3.1}
Under assumptions of Theorem \ref{theorem 7.3}, any two finite energy weak solutions to Eq. \eqref{7.0} must coincide.
\end{corollary}

\begin{theorem}[Higher regularity of the finite energy strong solution]
			\label{theorem 1.9}
Let $p > 14$ and
invoke the assumptions of  Theorem  \ref{theorem 7.3} and assume, additionally, $\theta \ge 16$,
and $\sg \in  L_{p, \theta - 4} (\Sigma^T)$, $f_0 \in  \cO_{p, \theta - 4} (\Sigma^T)$ (see \eqref{1.2.7}).
Then, for any $\lambda \ge 0$, one has $f \in S_{p, \theta - 16} (\Sigma^T)$, and
\begin{equation}
			\label{eq1.9.0}
\begin{aligned}
     \|f\|_{ S_{p, \theta - 16} (\Sigma^T) }
    &+ \|f\|_{  L_{\infty} ((0, T), C^{\alpha/3, \alpha}_{x, v} (\overline{\Omega} \times \bR^3) )  }  + \|\nabla_v f|\|_{  L_{\infty} ((0, T), C^{\alpha/3, \alpha}_{x, v} (\overline{\Omega} \times \bR^3) )  }  \\
   & \le  N  ( \|\sg\|_{L_{2, \theta - 2} (\Sigma^T) } + \|\sg\|_{L_{p, \theta - 4} (\Sigma^T) }\\
	&\quad +  |f_0|_{ \cO_{2, \theta-2} } + |f_0|_{\cO_{p, \theta-4}}\\
    &\quad  + \|\nabla_v f\|_{ L_{2, \theta} (\Sigma^T) }  +  \|f\|_{ L_{2, \theta - 1} (\Sigma^T) }),
\end{aligned}
\end{equation}
 where
\begin{itemize}
\item[--] $\alpha = 1 - 14/p$,
\item[--] $C^{\alpha/3, \alpha}_{x, v} (\overline{\Omega} \times \bR^3)$ is the  anisotropic H\"older space defined in \eqref{1.2.0},
\item[--] $N = N (p,  \theta, K, \Omega)$.
\end{itemize}
\end{theorem}

\subsubsection{Linearized Landau equation}

\begin{definition}
			\label{definition 1}
Let $H_{\sigma, \theta}  (\Sigma^T)$ be the Hilbert space of all Lebesgue measurable functions such that the norm
\begin{equation}
			\label{eq1}
	\|u\|_{\sigma, \theta}: = \bigg( \int_{\Sigma^T} (\sigma^{i j} \partial_{v_i} u \, \partial_{v_j} u + \sigma^{i j} v_i v_j u^2) \jb^{\theta}  \, dz\bigg)^{1/2}
\end{equation}
is finite. Here $\sigma$ is defined by \eqref{eq11.5}.
Since $\sigma$ is a positive definite matrix (see Remark \ref{remark 2.17}), $\|\cdot\|_{\sigma, \theta}$ is, indeed, a norm.
\end{definition}

\begin{remark}
			\label{remark 2.17}
Let $\theta \ge 0$ and $u \in L_{2, \theta} (\Sigma^T)$ be a function such that $\nabla_v u \in L_{2, \theta} (\Sigma^T)$. Then,
\begin{equation}
			\label{eq2.17}
	\|u\|_{\sigma, \theta} \le N \||u| + |\nabla_ v u|\|_{  L_{2, \theta - 1} (\Sigma^T) }.	
\end{equation}
This follows from the facts that, for any $v \in \bR^3$,
$$
	\sigma^{i j} (v)  v_i v_j \le N \jb^{-1}, \quad \sigma (v) \le N \jb^{-1}I_3.
$$
Furthermore, due to the inequality
$$
	\jb^{-3} I_3   \le  C_1	\sigma (v)
$$
for some constant $C_1 > 0$,  we have
$$
	\|u\|_{ L_{2, \theta - 3} (\Sigma^T)  } \le C_1  \|u\|_{\sigma, \theta}.
$$
See the details in Lemma 3 of \cite{G_02}.
\end{remark}

\begin{definition}
			\label{definition 11.0}
We say that $f$ is a  finite energy strong solution to   \eqref{eq11.1} if there exists $\theta \ge 0$ such that
\begin{enumerate}

\item 	$f \in  S_2  (\Sigma^T) \cap L_{2, \theta} (\Sigma^T) \cap H_{\sigma, \theta}  (\Sigma^T)$,

\item
 $
	f (T, \cdot), f (0, \cdot) \in   L_{2, \theta} (\Omega \times \bR^3),
$
$f_{\pm} \in L_{\infty} (\Sigma^T_{\pm}, |v \cdot n_x|)$,

\item for any $\phi \in C^1_0 (\overline{\Sigma^T})$,
\begin{align*}
	&	-\int_{\Sigma^T}  (Y \phi)   f \, dz
		 + \int_{\Omega \times \bR^3 }     \big(f (T, x, v) \phi (T, x, v)  - f (0, x, v) \phi (0, x, v)\big)\, dx dv \notag\\
 	 &
		+  \int_{\Sigma^T_{+}}    f_{+} \phi \,  |v \cdot n_x|   d\sigma dt
	-   \int_{\Sigma^T_{-}}   f_{-} \phi  \,  |v \cdot n_x|  d\sigma dt\\
	  &
	  +\int_{\Sigma^T}   [ (\sigma_G \nabla_v f) \cdot \nabla_v \phi - (a_g \cdot \nabla_v f) \phi  - (\overline{K}_g f) \phi]   \,dz = 0,
  \end{align*}

 \item  Eq. \eqref{eq11.1} is satisfied almost everywhere (including the initial and the boundary conditions).
\end{enumerate}
\end{definition}

\begin{assumption}
			\label{assumption 11.1}
The function  $g \in L_{\infty} ((0, T), C^{\varkappa/3, \varkappa}_{x, v} (\overline{\Omega}))$
is such that
\begin{itemize}
\item[--] $\nabla_v g \in L_{\infty} ((0, T), C^{\varkappa/3, \varkappa}_{x, v} (\overline{\Omega}))$ for some $\varkappa \in (0, 1]$, and
 \begin{equation}
			\label{con11.1}
	\|g\|_{   L_{\infty} ((0, T), C^{\varkappa/3, \varkappa}_{x, v} (\overline{\Omega})) } + \|\nabla_v g\|_{  L_{\infty} ((0, T), C^{\varkappa/3, \varkappa}_{x, v} (\overline{\Omega})) } \le K
 \end{equation}
for some $K > 0$.
\item[--] \begin{equation}
			\label{con11.1'}
		g_{-} (t, x, v) = g_{+} (t, x, R_x v) \quad   \forall z \in \Sigma^T_{-}.
	      \end{equation}
\end{itemize}
\end{assumption}

\begin{theorem}
			\label{theorem 11.2}
Let
\begin{itemize}
\item[--] $\Omega$ be a bounded $C^3$ domain,

\item[--] $\varkappa \in (0, 1)$, $T > 0$, $p > 14$ be numbers,

\item[--] Assumption \ref{assumption 11.1} (see \eqref{con11.1} - \eqref{con11.1'}) hold,

\item[--] $f_0  \in \cO_{2, \theta} \cap \cO_{\infty}$ (see \eqref{1.2.7}),

\item[--] $\|g\|_{ L_{\infty} (\Sigma^T ) } \le \varepsilon$.
\end{itemize}
Then, there exists a number $\theta_0=\theta_0(p,\varkappa) > 4$ such that for any $\theta \ge \theta_0$, there exist numbers  $\theta = \theta (p, \varkappa) > 4 $, $\varepsilon  = \varepsilon (\theta) \in (0, 1)$,
and $\theta'  = \theta' (p, \varkappa), \theta''  = \theta'' (p, \varkappa) \in (1, \theta -3)$
 such that
Eq. \eqref{11.1}   has a unique finite energy strong solution  $f$ (see Definition \ref{definition 11.0}), and furthermore,
\begin{equation}
			\label{eq11.2.0}
\begin{aligned}
    &	\|f\|_{S_{2,  \theta' } (\Sigma^T)} + \|f\|_{S_{p,  \theta'' } (\Sigma^T)} +  \|f (T, \cdot)\|_{ L_{2, \theta} (\Omega \times \bR^3)  } +  \|f\|_{  L_{2, \theta} (\Sigma^T) }  +   \|f\|_{ \sigma, \theta  }   \\
&	+ \|f\|_{  L_{\infty} ((0, T), C^{\alpha/3, \alpha}_{x, v} (\overline{\Omega}\times \bR^3))   } + \|\nabla_v f\|_{  L_{\infty} ((0, T), C^{\alpha/3, \alpha}_{x, v} (\overline{\Omega}\times \bR^3))   } \\
& \le N  (\|f_0\|_{\cO_{2, \theta} } + \|f_0\|_{\cO_{\infty} }),
\end{aligned}
\end{equation}
where $\alpha = 1 - 14/p$, the $\|\cdot\|_{\sigma, \theta}$-norm is defined in \eqref{eq1}, and
$
	N = N (p, \theta, \varkappa, \Omega, K, T).
$
\end{theorem}

\section{Method of the proof}
                                                            \label{subsection 1.4}
\begin{enumerate}
\item \textbf{Kinetic Fokker-Planck equation,}

\textit{Existence.}
To construct the finite energy weak solutions to \eqref{7.0}, we discretize the diffusion operator
$ \nabla_v \cdot (a \nabla_v f)$  and prove uniform bounds in Section \ref{section 2} by using the results of  \cite{BP_87}.
The discretized second-order operator $A_h$ (see \eqref{7.1.4}) is inspired by the representation of symmetric matrices in Theorem 3.1 of \cite{Kr_11}  (see Lemma \ref{theorem 7.1}).
Such representation (see \eqref{7.1.1}) was first stated  in the work  \cite{MW_53} without smoothness properties of the weights $\lambda_k$.
 The  form of $A_h$ enables us to prove an analogue of the $L_p$-energy inequality for the heat equation (see \eqref{eq7.5.3}).
 We then use the weak* compactness argument to conclude the existence of finite energy weak solutions.

\textit{Uniqueness and higher regularity.} Thanks to the specular reflection boundary condition, we are able to construct a
reflection operator (see \eqref{eq1.4.2.2}) to extend solutions to $x \in \bR^3$.
In the present authors' opinion, such a
reflection operator  is
useless for the study of the regularity for the Vlasov type equations in the absence of velocity diffusion.
In particular, when the same reflection operator is applied to such an equation,
it preserves the form of an equation but yields a drift term with discontinuous coefficients.
On the other
hand, in the presence of velocity diffusion, one can apply the $S_{2}$ regularity theory in $\bR^{7}_T$  (see Section \ref{section 3})
to the extended equation \eqref{eq1.19} and conclude that a finite energy weak solution with $\theta \ge 2$ is of class $S_{2} (\Sigma^T)$.
We then use the $S_p$ regularity results of \cite{DY_21a} (see Appendix \ref{Appendix D}) to conclude that $f \in S_{p} (\Sigma^T)$.
Unfortunately, the mirror extension argument works only for  very particular diffusion operators in the velocity variable,
for example,   $\Delta_v$ and $\nabla_v \cdot (\sigma_G \nabla_v)$, where $\sigma_G$ is defined in \eqref{eq11.5}. See the discussion in Section \ref{subsection 1.4.4}.

\item \textbf{Linearized Landau equation.}
To prove Theorem \ref{theorem 11.2}, we use the method of vanishing viscosity.
	  We first consider a  simplified version of  Eq. \eqref{eq11.1} given by Eq. \eqref{eq1.18}.
By the existence result for the generalized kinetic Fokker-Planck equation (see Theorem \ref{theorem 3.3}), we prove the existence of a finite energy weak  solution in the sense of Definition \ref{definition 7.0}.
To prove the unique solvability in  the class of the finite energy strong  solutions,  we use the mirror extension method combined with the $S_2$ estimate of \cite{DY_21a} (see Appendix \ref{Appendix D}).

\textit{Viscous linearized Landau equation.} By using a perturbation argument, we extend the aforementioned unique solvability result to the viscous linearized Landau equation, which contains the non-local term $\overline{K}_g f$ defined in \eqref{eq11.2}.
We then prove uniform in $\nu$  bounds  by combining the standard energy estimate  (see Lemma \ref{lemma 4.6})
with the $S_p$ estimates of \cite{DY_21a}.
Finally, by using the weak* compactness argument, we prove the existence of the finite energy strong  solution (see Definition \ref{definition 11.0}) to the linearized Landau equation \eqref{eq11.1}.
The uniqueness in the class of the finite energy strong solutions follows from the aforementioned energy estimate.
\end{enumerate}

Our argument for proving the existence/uniqueness of the finite energy strong solution for  Eqs. \eqref{7.0} and \eqref{eq1.18} goes as follows:
$$
		\textit{existence of finite energy weak solution}\, \,  f \rightarrow f \in S_2  \rightarrow \textit{uniqueness}.
$$

In the rest of this section, we
\begin{itemize}
\item  define the mirror extension operator and
show that
it preserves the form of Eq \eqref{7.0},

\item delineate the proofs of the $S_2$ bound, $S_p$ estimate, and the H\"older estimate,

\item elaborate on the importance of the condition \eqref{con11.1} in Assumption \ref{assumption 11.1}.
\end{itemize}

\subsection{A boundary-flattening diffeomorphism that preserves the specular reflection boundary condition}
                    \label{subsection 1.4.1}
We  may assume that there exists a sufficiently small number $r_0 > 0$ such that for any $x_0 \in \Omega$, there exists a function $\rho  = \rho (x_1, x_2)$  such that
\begin{equation}
\begin{aligned}
                \label{eq1.4.1}
     &   \partial \Omega \cap B_{r_0} (x_0) \subset \{x: x_3 = \rho (x_1, x_2) \},\\
    &\Omega_{r_0} (x_0) := \Omega \cap B_{r_0} (x_0) \subset \{x: x_3 < \rho (x_1, x_2)\},
\end{aligned}
\end{equation}
and  $\rho$ is a bounded function with bounded continuous partial derivatives up to order $3$.

Next, denote
$
    \rho_i = \frac{\partial \rho} {\partial x_i}, i = 1, 2.
$
  Following  \cite{GHJO_20} (see Section 7.1.1 of the reference),
we introduce the boundary-flattening diffeomorphism
$$
    \Psi: \Omega_{r_0} (x_0)  \times \bR^3 \to \bH_{-} = \bR^3_{-}
\times \bR^3, \quad (x, v) \to (y, w)
$$
given by
\begin{equation}
			\label{eq1.4.2.4}
    y = \psi (x), \quad w = D \psi (x) v,
\end{equation}
 where $\psi$
is defined
as the inverse of the transformation
\begin{equation}
			\label{eq1.4.2.8}
    \psi^{-1} (y)
    = \begin{pmatrix}
	y_1 \\ y_2 \\ \rho (y_1, y_2)
       \end{pmatrix}
	+ y_3 \begin{pmatrix}
			-\rho_1 \\ -\rho_2 \\ 1
	          \end{pmatrix}.
\end{equation}
It is easy to check that  $D \psi^{-1}$
sends the vector $(0, 0, 1)^T$ to $(- \rho_1, - \rho_2, 1)^T$,
which is an outward normal vector to $\partial \Omega$.
Furthermore, by direct computations
  (see \cite{GHJO_20}),
  \begin{equation}
                \label{7.3.6}
     R_x v = v - 2 (v\cdot n_x) n_x = \begin{pmatrix}w_1 + \rho_1 w_3\\
     w_2 + \rho_2 w_3 \\ \rho_1 w_1 + \rho_2 w_2 - w_3 \end{pmatrix}
     = \frac{\partial x}{\partial y}|_{y_3 = 0} (w_1, w_2, -w_3)^T,
    \end{equation}
which shows that $\Psi$ preserves the specular reflection boundary condition.

\subsubsection{Mirror extension transformation}
                    \label{subsubsection 1.4.2}
Denote
$$
    \widehat z = (t, y, w),     \quad J = \bigg|\frac{\partial x}{ \partial y}\bigg|^2,
$$
and, for any function $u$, supported on $\bR \times  \Omega_{r_0} (x_0) \times \bR^3$,  we set
\begin{equation}
			\label{eq1.4.2.1}
     \widehat u (\widehat z) = u (t, x(y), v (y, w)), \quad \widetilde u (\widehat z)  =  \widehat u (\widehat z) J (y).
\end{equation}
Then,  the mirror extension of $u$ is defined as
\begin{equation}
			\label{eq1.4.2.2}
    \overline{u} (t, y, w) : = \begin{cases} 	 \widetilde u (t,  y,  w), (t, y, w) \in \bH^T_{-},\\
								\widetilde u (t, R y, R w), (t, y, w) \in \bH^T_{+},
					\end{cases}
			\quad R =  \text{diag} (1, 1, -1).
\end{equation}
Note that, strictly speaking, $\overline{u}$ is not an extension of $\widehat u$, however, since the Jacobian $J \approx 1$
near $x_0$, $\overline{u}$ is close to the function
$$
   \widehat z \to \widehat u (t, R y, R w)
$$
provided that $r_0$ is sufficiently small.

\subsubsection{The generalized kinetic Fokker-Planck equation under the mirror extension mapping}
                                        \label{subsubsection 1.4.3}
We consider Eq. \eqref{7.0} and assume that there exists some $\delta \in (0, 1)$ such that for any $z$,
$a \in \text{Sym} (\delta)$, i.e.,
\eqref{eq2.1.0} holds.
Let   $f$ be a finite energy solution to Eq. \eqref{7.0} (see Definition \ref{definition 7.0}) supported on $\bR \times \Omega_{r_0/2} (x_0) \times \bR^3$.
Here we show that the mirror extension $\overline{f}$ (see \eqref{eq1.4.2.2}) also solves a generalized kinetic Fokker-Planck type equation.
To this end, we  change variables in the weak formulation of Eq. \eqref{7.0}  for some fixed $\phi \in C^1_0 (\overline{\Sigma^T})$  (see \eqref{7.0.0} in Definition \ref{definition 7.0}).

First, we find an equation satisfied by $\widetilde f$ (defined in \eqref{eq1.4.2.1})  on $\mathbb{H}_{-}^T$.
By direct computations (see Appendix \ref{Appendix A}),  $\widetilde f$ satisfies the identity
\begin{equation}
                \label{7.3.2}
 	\begin{aligned}
 	&- \int_{  \mathbb{H}^T_{-} }
      (Y \widehat \phi) \,     \widetilde f  d\widehat z
+	 \int_{\mathbb{H}_{-}}  \big(\widetilde   f  (T, y, w)
 	    \widehat \phi (T, y, w)   - \widetilde f  (0, y, w)
 	    \widehat \phi (0, y, w)\big) \, dy dw
	  \\
	   &
       + \int_{  \mathbb{H}_{-}^T }
       \big(   (\nabla_w \widetilde f)^T A \nabla_{w} \widehat \phi
       - \widetilde f X \cdot \nabla_{w} \widehat \phi
	+ \big(B \cdot \nabla_{w} \widetilde f \, \big) \, \widehat \phi
       + \lambda
         \widetilde f  \, \widehat \phi \,\big) \,  d\widehat z\\
         &
		+ \int_0^T \int_{\bR^2 \times \bR^3_{+}}  |w_3|    \widetilde f_{+} \widehat \phi \, dy_1 dy_2  dw dt \\
 &
	-   \int_0^T \int_{\bR^2 \times \bR^3_{-}}  |w_3|    \widetilde f_{-} \widehat \phi \, dy_1 dy_2 dw  dt
	= \int_{ \mathbb{H}^T_{-}} \widehat \phi \widetilde \sg \, d\widehat z,
	\end{aligned}
	\end{equation}
where $\widehat \phi$ is defined by \eqref{eq1.4.2.1}, and
\begin{align}
		\label{eq1.4.3.2}
   &  A  = \bigg(\frac{\partial y}{\partial x}\bigg)\,  \widehat a  \, \bigg(\frac{\partial y}{\partial x}\bigg)^T,  \quad   B = \bigg(\frac{\partial y}{\partial x}\bigg) b, \\
&
\label{eq1.4.3.5}
        X =  (X_1, X_2, X_3)^T =   \bigg(\frac{\partial y}{\partial x}\bigg) \bigg(\frac{\partial v}{\partial y}\bigg) w
        =  \bigg(\frac{\partial y}{\partial x}\bigg) \frac{\partial \big(\frac{\partial x}{\partial y} w\big)} {\partial y} w.
\end{align}

Next, we extend the coefficients  to $\bH^T_{-}$.
First, we set $B (t, \cdot, w), X (\cdot, w)$ to be $0$ if $y \in \bR^3_{-} \setminus \psi (\Omega_{r_0} (x_0))$.
Second, let $\kappa \in C^{\infty}_0 (\bR^3)$  be a function such that $0 \le \kappa \le 1$,
and $\kappa =  1$ on $|y| \le  3r_0/4 $, $\kappa = 0$ on $|y| \ge 7 r_0/8$ and denote
\begin{equation}
			\label{eq1.4.3.3}
    \mathcal{A} (t, y, w) = A (t, y, w) \kappa (y) + \delta I_3 (1-\kappa (y)).
\end{equation}
Note that for sufficiently small $r_0$,
\begin{equation}
                    \label{eq10.19}
\widehat f\text{ is supported on }B_{ 3 r_0/4 } (x_0),
\end{equation}
 and then, so is $\widetilde f$.
Since $\kappa = 1$ on the support of $\widetilde f$, Eq. \eqref{7.3.2}
holds with $\widehat \phi$ replaced with any $\eta \in C^1_0 (\overline{\bH^T_{-}})$.

We now extend $\mathcal{A}, B, X$  to $\bR^7_T$  by setting
\begin{align}
		\label{eq1.20}
 &           \mathbb{A} (t, y, w)
         = \begin{cases} \mathcal{A} (t, y, w), \quad (t, y, w) \in \bH^T_{-},\\
            R \,  \mathcal{A} (t, R y, R w)\,   R, \, \, (t, y, w) \in \bH^T_{+},
            \end{cases}  \\
&
	\label{eq1.4.3.6}
    \mathbb{X} (y, w) = \begin{cases} X (y, w), \, \, (y, w) \in \bH_{-}, \\ R \,  X (R y, R w), \,\,  (y, w) \in \bH_{+}, \end{cases} \\
&\label{eq1.4.3.4}
 \mathbb{B} (t,  y,  w) = \begin{cases} B (t, y, w), \, \, (t, y, w) \in \bH^T_{-}, \\
R \,  B (t, R y, R w), \, \, (t, y, w) \in \bH^T_{+}. \end{cases}
\end{align}

Finally, we check that $\overline{f}$ satisfies an equation on the whole space.
We fix an arbitrary function $\eta \in C^{1}_0 (\overline{\bR^7_T})$
    and denote by
    $
        \eta_{\pm}
    $
     the restriction of
    $\eta$ to $\overline{\mathbb{H}^T_{\pm}}$.
      Replacing $\widehat \phi$
    with $\eta_{+} (t, R y, R w)$
    in \eqref{7.3.2} and changing variables,
     we obtain
   \begin{equation}
                    \label{7.3.5}
    	\begin{aligned}
 	    &   -	\int_{ \mathbb{H}_{+}^T }
    (Y  \eta_{+})    \overline{f}  \,  d\widehat z
	+ \int_{ \mathbb{H}_{+}}   [\overline{f}  (T, y, w)
 	     \eta_{+} (T, y, w) -  \overline{f}  (0, y, w)
 	     \eta_{+} (0, y, w)]   \, dy dw
	  \\
	   &
       + \int_{\mathbb{H}_{+}^T }
       \big((\nabla_w \overline{f})^T \mathbb{A}  \nabla_{w} \eta_{+}
       - \overline{f} \, \mathbb{X} \cdot \nabla_w \eta_{+}
	+\mathbb{B} \cdot \nabla_{w}	\, \overline{f} \,   \eta_{+}
       + \lambda
         \overline{f} \,  \eta_{+}\big) \,  d\widehat z\\
         &
		+ \int_0^T \int_{\bR^2 \times \bR^3_{-}}
		 |w_3|     \widetilde f_{+} (t, y_1, y_2, R w)  \eta \, dy_1 dy_2 \, dw\, dt \\
    &
	-   \int_0^T \int_{\bR^2 \times \bR^3_{+}}    |w_3|   \widetilde f_{-} (t, y_1, y_2, R w) \eta\, dy_1 dy_2 \, dw\, dt\\
&	= \int_{ \mathbb{H}^T_{+}}  \overline{\sg} \eta_{+} \, d\widehat z.
	\end{aligned}
	   	\end{equation}
    Note that by \eqref{7.3.6}
and the fact that $f$
    satisfies the specular reflection boundary condition, we have
    \begin{equation}
                    \label{7.3.7}
        \widehat f_{+} (t, y_1, y_2,  R w)
        = \widehat  f_{-} (t, y_1, y_2,  w).
    \end{equation}
  Then, adding \eqref{7.3.5} to \eqref{7.3.2} with
  $\widehat \phi = \eta_{-}$
  and using \eqref{7.3.7} give
     	\begin{align*}
 	    &    -	\int
    (Y \eta)    \overline{f}  \,  d\widehat z
	+ \int   [\overline{f}  (T, y, w)
 	     \eta (T, y, w)  -  \overline{f}  (0, y, w)
 	     \eta (0, y, w)]   \, dy dw
	  \\
	   &
       + \int
       \big((\nabla_{w} \overline{f})^T \mathbb{A}  \nabla_{w} \eta
 - \overline{f} \mathbb{X} \cdot \nabla_{w} \eta
+ \mathbb{B} \cdot \nabla_{w} \overline{f} \eta
       + \lambda
         \bar f  \eta\big) \,  d\widehat z \\
         &= \int  \overline{\sg} \eta  \, d\widehat z .
	\end{align*}
Thus, in $\bR^7_T$, $\overline{f}$ satisfies the equation
\begin{equation}
			\label{eq1.19}
    Y \overline{f} - \nabla_w \cdot (\mathbb{A} \nabla_w \, \overline{f}) + \nabla_w \cdot (\mathbb{X} \overline{f}) + \mathbb{B} \cdot \nabla_{w} \overline{f} + \lambda \overline{f} = \overline{\sg}
\end{equation}
in the weak sense (see Definition \ref{definition 7.0} \eqref{7.0.0}).

\subsection{Strategy for proving the \texorpdfstring{$S_2$}{} estimate}
						\label{subsection 1.4.4}
\begin{itemize}
\item
\textbf{Interior estimate for Eqs. \eqref{7.0} and \eqref{eq1.18}.}  Away from the boundary $\partial \Omega$, by the  partition of unity argument, we may reduce these equations to the ones on the whole space $\bR^7_T$.
To show that $f \in S_2$ away from the boundary, we use  the unique solvability result  of \cite{DY_21a} (see Theorem \ref{theorem D.1}).
As discussed in the Appendix \ref{Appendix D} (see Remark \ref{remark 3.1}), this theorem is applicable if the leading coefficients are of class
$L_{\infty} ((0, T), C^{\varkappa/3, \varkappa}_{x, v} (\bR^6) )$, $\varkappa \in (0, 1]$, where $C^{\varkappa/3, \varkappa}_{x, v} (\bR^6)$  is defined in \eqref{1.2.0}.
 For Eq. \eqref{eq1.18}, the validity of this condition follows from  Assumption \ref{assumption 11.1} (\eqref{con11.1} - \eqref{con11.1'}). See Lemma \ref{lemma C.1}.

\item
\textbf{Boundary estimate.}
Near the boundary, we use the boundary flattening diffeomorphism $\Psi$ (see \eqref{eq1.4.2.4}) and the mirror extension transformation \eqref{eq1.4.2.2} to reduce Eqs. \eqref{7.0}, \eqref{eq1.18} to the generalized kinetic Fokker-Planck equation \eqref{eq1.19}.
Let us check if the leading coefficients belong to $L_{\infty} ((0, T), C^{\varkappa/3, \varkappa}_{x, v} (\bR^6) )$.
Due to \eqref{eq1.20}, for $(t, y, w) \in \bH^T_{+}$, we have
\begin{equation}
			\label{eq1.21}
	\mathbb{A} (t, y, w) = \begin{pmatrix}
				\mathcal{A}^{11} & \mathcal{A}^{12} & -\mathcal{A}^{13}\\
				\mathcal{A}^{12} & \mathcal{A}^{22} & -\mathcal{A}^{23}\\
				-\mathcal{A}^{13} & -\mathcal{A}^{23} & \mathcal{A}^{33}
				\end{pmatrix} (t, y, R w).
\end{equation}
Note that, unless  $\mathcal{A}$ satisfies the condition
\begin{equation}
			\label{eq1.23}
	\mathcal{A}^{i3} (t, y_1, y_2, 0, w) = - \mathcal{A}^{i3} (t, y_1, y_2, 0, R w), i = 1, 2,
\end{equation}
the function $\mathcal{A} (t, \cdot)$ might be discontinuous on the set  $\{y_3 = 0\} \times \bR^3$.
Fortunately, \eqref{eq1.23} holds for both Eqs. \eqref{7.0} and \eqref{eq1.18}. See the details in Appendix \ref{Appendix E}.
By the unique solvability result in the $S_2$ space (see Theorem \ref{theorem D.1}) applied to Eq. \eqref{eq1.19}, we conclude that $f \in S_2$ near the boundary.
Thus,  for large $\lambda$,  finite energy weak solutions to Eqs. \eqref{7.0} and \eqref{eq1.18} must be finite energy strong solutions in the sense of Definition \ref{definition 7.0}.
\end{itemize}

\subsection{Strategy for proving the \texorpdfstring{$S_p$}{} bound and H\"older continuity}
					\label{subsection 1.4.5}
We use a bootstrap argument combined with the localization method described above.
In particular, we estimate the $S_p$ norm of a localized solution by using  the a priori estimate in Theorem \ref{corollary 3.4} (see \eqref{2.2.1}).
When $p > 14$, by the embedding theorem for the $S_p$ space (see Theorem 2.1 of \cite{PR_98}),
our localized solutions and their derivatives in the velocity direction are of class $C_{\text{kin}}^{\alpha}, \alpha = 1 - 14/p$, where
$$
	C_{\text{kin}}^{\alpha} = \{u \in L_{\infty} (\bR^7_T): [u]_{  C^{\alpha}_{\text{kin}} (\bR^7_T) } < \infty\},
$$
and
$$
	[u]_{  C^{\alpha}_{\text{kin}}  } : = \sup_{z, z' \in \bR^7_T: z \ne z'} \frac{ |u (z) -u (z')| }{(|t-t'|^{1/2} + |x-x' - (t-t')v'|^{1/3} + |v-v'|)^{\alpha}} < \infty.
$$
Unfortunately, the condition $[u]_{  C^{\alpha}_{\text{kin}}  } < \infty$  is not preserved under the \textit{local} diffeomorphism $\Psi^{-1}$.
In other words, if $u (t, y, w) \in C^{\alpha}_{\text{kin}}$, then, $U (t, x, v) =   u (t,  \Psi (x, v))$ is not of class $C_{\text{kin}}^{\alpha}$
since $x \to U (t, \cdot, v)$ is not defined globally.

To overcome this issue, we work with a weaker space $L_{\infty} ((0, T), C^{\alpha/3, \alpha}_{x, v} (\bR^6)) \supset C^{\alpha}_{\text{kin}}$.
This space has two important properties which we state below.
\begin{itemize}
 \item If $u \in S_p ((-\infty, T) \times \bR^6)$, then, $u, \nabla_v u \in L_{\infty} ((-\infty, T), C^{\alpha/3, \alpha}_{x, v} (\bR^6))$.
	This follows from the aforementioned Morrey type embedding theorem for the $S_p$ spaces (see \cite{PR_98}).

\item The condition $u \in L_{\infty} ((0, T), C^{\alpha/3, \alpha}_{x, v} (\bR^6))$  is preserved under the local diffeomorphism $\Psi^{-1}$ in the sense explained above.
\end{itemize}
By the above reasoning, we are able to prove that, under certain assumptions,  if $f$ is a finite energy strong solution  to Eq. \eqref{7.0} or \eqref{eq1.18}, then $f  \in S_p (\bR^7_T)$, and, in addition,
$f, \nabla_v f \in L_{\infty} ((0, T), C^{\alpha/3, \alpha}_{x, v} (\bR^6))$.

\subsection{The reason we consider the linearized Landau equation with rough in time coefficients}
																\label{subsection 1.4.7}
To prove the existence of the local in time solution to the nonlinear Landau equation (Eq. \eqref{11.1} with $g =  f$), one can consider the Picard iteration sequence $f^{(0)} = f_0$,
\begin{equation}
			\label{eq2.4.1}
\begin{aligned}
	&Y f^{(n+1)} - \nabla_{v} \cdot (\sigma_{ F^{(n)} } \nabla_{v} f^{(n+1)}) + a_g \cdot \nabla_v f^{(n+1)} + \overline{K}_{f^{(n)} } f^{(n+1)} = 0 \, \, \text{in} \, \, \Sigma^T,\\
	 &
	f^{(n+1)} (0, x, v) = f_0 (x, v), x \in \Omega, v \in \bR^3,  \, \,
		 f^{(n+1)}_{-} (t, x, v) = f^{(n+1)}_{+} (t, x, R_x v), z \in \Sigma^T_{-},
\end{aligned}
\end{equation}
where $\sigma_{ F^{(n)} }$ is defined by \eqref{eq11.5} with $g$ replaced with $f^{(n)}$.
Let us consider the equation for $f^{(1)}$. Even if $f_0$ is very regular, by using our method,  we can only prove that
$f^{(1)}, \nabla_v f^{(1)} \in L_{\infty} ((0, T), C^{\alpha/3, \alpha}_{x, v} (\Sigma^T))$,
$\alpha \in (0, 1]$.
Then, for the equation  \eqref{eq2.4.1},   Assumption \ref{assumption 11.1} (see \eqref{con11.1} - \eqref{con11.1'}) is satisfied with $g = f^{(1)}$.


\section{Finite energy weak solutions to the generalized kinetic Fokker-Planck equation}
                                    \label{section 2}
In this section, we prove the existence of finite energy weak solution to Eq. \eqref{7.0} (see Theorem \ref{theorem 3.3}).
The conclusion of  Theorem \ref{theorem 3.3} implies the existence of finite energy weak solutions to Eqs. \eqref{7.0} and \eqref{eq1.18}.
This is obvious for Eq. \eqref{7.0}, and for Eq. \eqref{eq1.18}, it is verified in the proof of Proposition \ref{proposition 4.1}.

\subsection{Discretization of the second-order  operator}

The following lemma, which is due to N.V. Krylov, provides a way to rewrite a second-order operator
as  a pure second-order derivative operator. By using this result, we construct a finite difference approximation of the operator  $\nabla_v \cdot (a \nabla_v)$, which we denote by $A_h$. Thanks to Lemma \ref{theorem 7.1}, the stencil depends only on the lower eigenvalue bound of the matrix $a$, and in addition, the `diffusion' coefficients of $A_h$ are as regular as  $a$.
\begin{lemma}[Theorem 3.1 of \cite{Kr_11}]
                \label{theorem 7.1}
Denote
	$
		l_i = e_i, i = 1, 2, 3.
	$
	Then, there exists a number $d_1 > 3$ and \bigskip
	\begin{itemize}
	\item[--]	 vectors
	$
		l_k \in \bR^3, k = 4, \ldots, d_1
	$,
	\item[--] real-analytic functions
	$
		\lambda_k, k = 1, \ldots, d_1,
	$
	 on
	 $
		\text{Sym} (\delta)
	$ (see Section \ref{subsection 1.2}),
	and a number $\delta_1  = \delta_1 (\delta) > 0$
	\end{itemize}
 		 such that, for any
	$
		a \in \text{Sym} (\delta)
	$,
	\begin{equation}
	            \label{7.1.1}
	    	a^{i j}  =  \sum_{k = 1}^{d_1} \lambda_k (a) l_k^i l_k^j,
			\quad \delta_1 < |\lambda_k|  < \delta_1^{-1}, k = 1, \ldots, d_1.
	\end{equation}
\end{lemma}

Let $d_1$, $\Lambda = \{l_1, \ldots, l_{d_1}\}$
and $\lambda_k, k = 1, \ldots, d_1,$ be the number, stencil, and
functions in Lemma \ref{theorem 7.1}.
We set
$$
    a_k (z) = \lambda_k (a (z)), k = 1, \ldots, d_1.
$$
Then, by the above lemma, there exist a constant $\delta_1 = \delta_1 (\delta) > 0$ such that for any $z$,
\begin{equation}
            \label{7.1.2}
     \delta_1 \le |a_k (z)| \le \delta_1^{-1} ,\quad k = 1, \ldots, d_1.
\end{equation}

For any $h \in \bR$ and any function $u$ on $\bR^3$, we define the following  operators:
\begin{align}
&	T_{h, l} u (v) = u (v+hl),
	\quad
 	\delta_{h, l} u (v) = \frac{u (v+h l ) -  u (v)}{h},\notag\\
& \label{7.1.4}
	A_h u = - \sum_{k = 1}^{d_1} \delta_{h, -l_k} (a_k \delta_{h, l_k} u).
\end{align}
We now check the consistency with $\nabla_{v} \cdot (a \nabla_{v})$.
Denote
\begin{equation}
			\label{7.1.6}
    \Delta_{h, \xi} = \frac{T_{h, \xi} - 2 I_3+ T_{h, -\xi} }{h^2},
\quad
     \partial_{(\xi)} = \xi_i \partial_{v_i}, \quad \partial_{(\xi)(\xi)} = \xi_i \xi_j \partial_{v_i} \partial_{v_j}.
\end{equation}
Observe that the following product rule holds:
\begin{equation}
                \label{7.5.1}
    \delta_{h, l} (f g)
    =  g \delta_{h, l} f
    + (T_{h, l} f) \delta_{h, l} g.
    \end{equation}
We fix arbitrary $\phi \in C^2_0 (\bR^3)$.
Then, by
\eqref{7.5.1}, we have
 \begin{equation}
                \label{7.1.3}
             A_h \phi
        = - a_k (\delta_{h, -l_k} \delta_{h, l_k} \phi)
        + (- \delta_{h, -l_k} a_k)\,  T_{h, -l_k} \delta_{h, l_k} \phi.
  \end{equation}
We note that the last term equals
$$
	(- \delta_{h, -l_k} a_k)\,  (-\delta_{h, -l_k} \phi).
$$
Then,  using the Cauchy-Schwartz inequality,  Assumption \ref{assumption 2.2} (see \eqref{eq2.2.0}),  and the fact that $\phi$ has bounded support, we get
$$
	\lim_{h \to 0} (- \delta_{h, -l_k} a_k)\,  (-\delta_{h, -l_k} \phi) = \partial_{(l_k)} a_k \partial_{(l_k)} \phi \quad \text{in} \, \, L_{2} (\Sigma^T).
$$
By this, \eqref{7.1.3}, and \eqref{7.1.1}, we conclude
\begin{equation}
				\label{eq2.1.1}
\begin{aligned}
	\lim_{h \to 0}	A_h \phi &=  a_k \partial_{(l_k) (l_k) } \phi  +  \partial_{(l_k)} a_k \partial_{(l_k)} \phi\\
  	&= \nabla_{v} \cdot (a \nabla_{v} \phi) \quad \text{in} \, \, L_{2} (\Sigma^T).
\end{aligned}
\end{equation}

\subsection{Discretized equation}
For $\varepsilon, h \in (0,1)$,
we consider the equation
\begin{align}
                \label{7.1}
    & Y f_{\varepsilon, h} - A_h f_{\varepsilon, h}  + b^i \delta_{ h, e_i} f_{\varepsilon, h}   + \lambda f_{\varepsilon, h} = \sg  \quad \text{in} \, \,  \Sigma^T \notag\\
& (f_{ \varepsilon, h})_{-} (t, x, v) = (1-\varepsilon) (f_{\varepsilon, h})_{+} (t, x, R_x v)\quad \text{on} \, \,  \Sigma^T_{-},\\
	& f_{\varepsilon, h} (0, x, v) = f_0 (x, v), \, \, (x, v) \in \Omega \times \bR^3.\notag
\end{align}	

Below we state the definition of the solution space. To do that, we first need to introduce  a set of test functions from \cite{BP_87}.

\begin{definition}
                \label{definition 2.2.3}
 We say that $\phi  \in \mathsf{\Phi}$ if
\begin{itemize}[--]
    \item $\phi$ is continuously differentiable along the  characteristic lines $(t+s, x + v s, v)$,

    \item $\phi$, $Y \phi$ are bounded functions on $\Sigma^T$,

    \item $\phi$ has bounded support, and there is a positive lower bound of the length of the aforementioned  characteristic lines inside $\Sigma^T$ which intersect the support of $\phi$.
\end{itemize}

\end{definition}

\begin{remark}
		\label{remark 2.2.4}
One can show that
$$
	C^1_0 \big(([0, T] \times \overline{\Omega} \times \bR^3)
\setminus ((0, T) \times \gamma_0 \cup \{0\} \times \partial \Omega \times \bR^3   \cup \{T\} \times \partial \Omega \times \bR^3)\big) \subset \mathsf{\Phi}.
$$
A proof can be found in Lemma 2.1 of \cite{G_93}.
\end{remark}

\begin{definition}
                \label{definition 2.2.1}
For $p \in [1, \infty)$, and $\theta \ge 0$,
 $E_{p, \theta} (\Sigma^T)$ is the Banach space of functions $u$
 with the  following properties:
\begin{itemize}
    \item[--]
$u, Y u \in L_{p, \theta} (\Sigma^T)$,

\item[--] There exist functions
$
     u_{\pm} \in L_p (\Sigma^T_{\pm}, |v \cdot n_x|),
$
$
     u (T, \cdot), u (0, \cdot) \in L_p (\Omega \times \bR^3)
 $
 such that
for any  $\phi \in \mathsf{\Phi}$ (see Definition \ref{definition 2.2.3}),
 the following Green's identity holds:
 	\begin{equation}
 	                \label{7.2}
 	\begin{aligned}
 	    &		\int_{\Sigma^T}  (Y u)  \phi
  +   (Y \phi)    u \, dz
	\\
	  &
		 =  \int_{\Omega \times \bR^3 }       u (T, x, v) \phi (T, x, v) \, dx dv
		 -
		 \int_{\Omega \times \bR^3 }     u (0, x, v)  \phi (0, x, v) \, dx dv
	  \\
 	   &
		\quad +  \int_{\Sigma^T_{+}}    u_{+} \phi \, |v \cdot n_x| \, d\sigma  dt
	-  \int_{\Sigma^T_{-}}    u_{-} \phi \, |v \cdot n_x| \, d\sigma  dt.
	\end{aligned}
	   	\end{equation}
\end{itemize}
\end{definition}

\begin{remark}
			\label{remark 2.2.3}
By Proposition 1 of \cite{BP_87}, if $u, Y u \in L_p (\Sigma^T)$,
then, there exists unique  functions $u_{\pm} \in L_{p, \text{loc} } (\Sigma^T_{\pm})$,
$  u (T, \cdot), u (0, \cdot)  \in L_{p, \text{loc} } (\Omega \times \bR^3)$ such that \eqref{7.2} holds.
See p. 393 of \cite{BP_87} for the definition of $L_{p, \text{loc} } (\Sigma^T_{\pm})$ and $L_{p, \text{loc} } (\Omega \times \bR^3)$.
\end{remark}

\begin{remark}
                \label{remark 2.2.2}
For any $\tau \in (0, T)$,
 there exist functions $u_{\tau},  u_{\pm; \tau}$ defined on $\Omega \times \bR^3$ and $\Sigma_{\pm}^{\tau}$, respectively,
 such that the identity \eqref{7.2} holds with $\tau, u_{\tau},  u_{\pm; \tau}$ in place of $T,  u (T, \cdot)$ and  $u_{\pm}$, respectively.
 Furthermore, it follows from the proof of Proposition 1 in \cite{BP_87} that for a.e. $\tau \in (0, T)$,
 $$
    u_{\tau} (\cdot)  = u|_{t = \tau}, \quad u_{\pm;  \tau}  = (u I_{t \in (0, \tau) })|_{ \gamma_{\pm} }.
 $$
\end{remark}

\begin{proposition}
            \label{proposition 7.4}
Under assumptions of Theorem \ref{theorem 3.3}, for any numbers $p > 1, \theta \ge 0, \varepsilon, h \in (0, 1]$,
and $\sg \in L_{p, \theta} (\Sigma^T)$,
 Eq. \eqref{7.1} has a unique (strong) solution $f_{\varepsilon, h} \in E_{p,  \theta } (\Sigma^T)$.
\end{proposition}

\begin{proof}
By Assumptions \ref{assumption 2.1} - \ref{assumption 2.2} (see \eqref{eq2.1.0} - \eqref{eq2.2.0}),
$A_h$ is a bounded operator on $L_{p, \theta} (\Sigma^T)$.
Now the assertion follows from
 Theorem 1 of \cite{BP_87}.
\end{proof}

\subsection{Uniform bounds for the discretized equation}

The following energy identity contained Proposition 1 of \cite{BP_87} is crucial in the proof of the uniform $L_p$ bounds for $f_{\varepsilon, h}$.

\begin{lemma}
                \label{lemma 7.1}
For any numbers $p \in [1, \infty), \theta \ge 0$ and $f \in E_{p, \theta} (\Sigma^T)$,
 one has
\begin{equation}
		\label{eq7.1.1}
\begin{aligned}
    &    \int_{\Omega \times \bR^3} (|f (T, x, v)|^p  - |f (0, x, v)|^p) \, \jb^{\theta} dx dv\\
&
   +  \int_{\Sigma^T_{+}}    |f_{+}|^p  \jb^{\theta}\,  |v \cdot n_x| d\sigma \, dt
    -  \int_{ \Sigma^T_{-} }    |f_{-}|^p  \jb^{\theta}\,  |v \cdot n_x| d\sigma \, dt\\
 &
  =
    p \int_{\Sigma^T} (Y f) |f|^{p-1} (\text{sgn} f) \jb^{\theta} \, dz.
 \end{aligned}
\end{equation}
\end{lemma}

\begin{lemma}
             \label{lemma 7.5}
Under the assumptions of Theorem \ref{theorem 3.3},  for any $\varepsilon, h \in (0, 1]$ and $\theta \ge 0$,
there exists
$\lambda_0  = \lambda_0 (\delta, \theta, K) > 0$
such that
for any
$\lambda \ge \lambda_0$,
$\sg \in L_{2, \theta} (\Sigma^T)
\cap L_{\infty} (\Sigma^T)$,
one has
\begin{equation}
                \label{eq7.5.0}
\begin{aligned}
   & \int_{\Omega \times \bR^3}  |f_{\varepsilon, h}|^2 (T, x, v)   \jb^{\theta} \, dx dv
    +  \delta_1 \sum_{k=1}^{d_1} \int_{\Sigma^T}
     |\delta_{h, l_k}f_{\varepsilon, h}|^2 \jb^{\theta} \, dz  \\
&
    +  (\lambda/2) \int_{\Sigma^T}
    |f_{\varepsilon, h}|^2  \jb^{\theta}\, dz
    + \varepsilon  \int_{\Sigma^T_{+}}
      |(f_{ \varepsilon, h})_{+}|^2  \jb^{\theta} \,  |v \cdot n_x|    d\sigma  dt  \\
      &
    \le
   \lambda^{-1}	 \int_{\Sigma^T}
   |\sg|^2 \jb^{\theta} \, dz
 +  \int_{   \Omega \times \bR^3 }
   f_0^2 \jb^{\theta} \, dx dv,
\end{aligned}
\end{equation}
and
\begin{align*}
    &   \max\{ \|f (T, \cdot)\|_{ L_{\infty} (\Omega \times \bR^3) },
    \|f_{\varepsilon, h}\|_{ L_{\infty} (\Sigma^T) },
   \| (f_{ \varepsilon, h})_{\pm}\|_{ L_{\infty} (\Sigma^T_{\pm}, |v \cdot n_x|)}\}\\
  &  \le
    	\lambda^{-1}  \| \sg\|_{ L_{\infty} (\Sigma^T) }
    	+  \|f_0\|_{L_{\infty} (\Omega \times \bR^3)  },
 \end{align*}
 where $\delta_1 = \delta_1 (\delta) > 0$ is defined on page \pageref{7.1.2}.
\end{lemma}

\begin{proof}
For the sake of convenience, we omit the summation with respect to $k \in \{1, \ldots, d_1\}$.

\textbf{Weighted $L_2$-estimate.}
First,  by Lemma \ref{lemma 7.1} with $p = 2$,
\begin{align*}
    &    \int_{\Omega \times \bR^3} \jb^{ \theta} |f_{\varepsilon, h}|^2 (T, x, v) \, dx dv
    + 2 \lambda \int_{\Sigma^T}
  \jb^{ \theta}   |f_{\varepsilon, h}|^2 \, dz\\
&
    + \varepsilon \int_{ \Sigma^T_{+}} \jb^{ \theta}   |(f_{\varepsilon, h})_{+}|^2  \, |v \cdot n_x|  d\sigma  dt \\
 &
   \underbrace{ - 2 \int_{\Sigma^T} \jb^{ \theta} (A_h f_{\varepsilon, h}) \,  f_{\varepsilon, h} \, dz }_{ = I_1}
	+  \underbrace{ 2 \int_{\Sigma^T} \jb^{ \theta} b^i (\delta_{h, e_i} f_{\varepsilon, h}) \,  f_{\varepsilon, h} \, dz}_{ = I_2} \\
    & \le 2 \int_{\Sigma^T}
     \jb^{\theta} \sg  f_{\varepsilon, h} \, dz + \int_{\Omega \times \bR^3} f^2_0 \jb^{\theta} \, dxdv.
 \end{align*}
  Furthermore, by using  the product rule (see \eqref{7.5.1})
and  change of variables, we get
 \begin{align*}
	I_1 &= 2 \int_{\Sigma^T} \jb^{ \theta} f_{\varepsilon, h} \, \delta_{h, -l_k} (a_k  \delta_{ h, l_k} f_{\varepsilon, h})  \,   \, dz  \\
		&= 2 \int_{\Sigma^T}  a_k  \, (\delta_{h, l_k} f_{\varepsilon, h}) \,  \delta_{ h, l_k} (f_{\varepsilon, h} \jb^{ \theta}) \,   \, dz    \quad \text{(change of variables)}\\
 	  &= 2 \int_{\Sigma^T} \jb^{  \theta} a_k   |\delta_{ h, l_k} f_{\varepsilon, h}|^2 \, dz\\
 	    &\quad + 2\int_{\Sigma^T} a_k (\delta_{h, l_k} f_{\varepsilon, h}) (T_{h, l_k} f_{\varepsilon, h}) \delta_{ h, l_k} \jb^{ \theta}  \, dz =: I_{1, 1} + I_{1, 2}  \quad \text{(product rule)}.
 \end{align*}
  By \eqref{7.1.2},
   $$
	I_{1, 1} \ge 2 \delta_1 \int_{\bR^3} \jb^{  \theta}  |\delta_{ h, l_k} f_{\varepsilon, h}|^2 \, dz.
   $$
       By the mean value theorem,
       for $h \in (0, 1]$ and fixed $k$,
       one has
       $$
           |\delta_{ h, l_k} \jb^{\theta}|
         \le N (l_k, \theta) \jb^{\theta - 1}.
       $$
       Next, separating $\delta_{h, l_k} f_{\varepsilon, h}$
       from $T_{h, l_k} f_{\varepsilon, h}$ by the
       Cauchy-Schwartz inequality,
       we get
        \begin{align*}
		I_{1, 2} \ge
		& -  (\delta_1/2)
		 \int_{\Sigma^T} \jb^{ \theta}    |\delta_{ h, l_k} f_{\varepsilon, h}|^2 \, dz\\
		&-  N (\theta, \Lambda) \delta_1^{-3 }
		\int_{\Sigma^T} \jb^{ \theta} | T_{h, l_k} f_{\varepsilon, h}|^2  \, dz.
     \end{align*}
     By a change of variables and the triangle inequality,
     $$
         	\int_{\Sigma^T} \jb^{ \theta} | T_{h, l_k} f_{\varepsilon, h}|^2  \, dz
         	= \int_{\Sigma^T}  T_{h, -l_k} \jb^{ \theta} | f_{\varepsilon, h}|^2  \, dz
         	\le N (\Lambda, \theta)
         	\int_{\Sigma^T}   \jb^{ \theta} | f_{\varepsilon, h}|^2  \, dz.
     $$
	Furthermore, 	
	        \begin{align*}
	&	  2\int_{\Sigma^T} \jb^{ \theta} b^i (\delta_{h, e_i} f_{\varepsilon, h}) \,  f_{\varepsilon, h} \, dz\\
	&
	\ge - (\delta_1/2) \int_{\Sigma^T} \jb^{\theta}|\delta_{h, e_i} f_{\varepsilon, h}|^2 \, dz
		-  N (K, \delta_1) \int_{\Sigma^T} \jb^{\theta}|f_{\varepsilon, h}|^2 \, dz.
	     \end{align*}
     Thus, by the above
     and the Cauchy-Schwartz inequality, we obtain
     	        \begin{align*}
   & \int_{\Omega \times \bR^3} \jb^{  \theta } |f_{\varepsilon, h}|^2 (T, x, v) \, dx dv
    +
  \delta_1
    \int_{\Sigma^T}
   \jb^{ \theta } |\delta_{h, l_k} f_{\varepsilon, h}|^2 \, dz\\
&
  +   \varepsilon   \int_{\Sigma^T_{+}} \jb^{ \theta }      |f_{+, \varepsilon, h}|^2  \, |v \cdot n_x| d\sigma  dt \\
&
   + ( \lambda   - N)
    \int_{\Sigma^T}
  \jb^{   \theta }   |f_{\varepsilon, h}|^2 \, dz
 \le   \lambda^{-1} \int_{\Sigma^T}  \jb^{ \theta } \sg^2   \, dz  + \int_{\Omega \times \bR^3} f^2_0 \jb^{\theta} \, dxdv,
  	     \end{align*}
  where $N = N (d_1, \delta_1, \Lambda, K)$.
  Thus, for $\lambda_0 >   2 N$ and $\lambda \ge \lambda_0$, the weighted energy estimate hold.

  \textbf{$L_{\infty}$ estimate.}
  \textit{Step 1: Higher regularity of $f_{\varepsilon, h}$.}
Here we show that
$$
    f_{\varepsilon, h} \in E_p (\Sigma^T), \, \,  \forall p \in [2, \infty),
$$
which allows us to apply Lemma \ref{lemma 7.1} in Step 2.

Fix any $p \in (2, \infty)$.
 By Proposition \ref{proposition 7.4} with $\theta = 0$, Eq. \eqref{7.1} has a unique solution  $\widetilde f_{\varepsilon, h} \in E_p (\Sigma^T)$.
Denote
$$
    \mu_{ n }(v) =  e^{-|v|^2/n},
    \quad \widetilde f^{(n)}_{\varepsilon, h} =  \mu_n \widetilde f_{\varepsilon, h}.
$$
Note that $\widetilde f^{(n)}_{\varepsilon, h} \in E_2 (\Sigma^T)$, and
$$
     (\widetilde f^{(n)}_{\varepsilon, h})_{-} (t, x, v)= (1-\varepsilon) (\widetilde f^{(n)}_{\varepsilon, h})_{+} (t, x, R_x v) \, \,  \text{on} \, \, \Sigma^T_{-}.
$$
Furthermore, since $\mu_n$ satisfies the specular reflection boundary condition, the function $F^{(n)} = \widetilde f^{(n)}_{\varepsilon, h} - f_{\varepsilon, h}$ satisfies the equation
\begin{align*}
 &       Y F^{(n)} - A_h F^{(n)} +  b^i \delta_{ h, e_i} F^{(n)} + \lambda F^{(n)} = \sg  (\mu_n  - 1) + \text{Comm}_n, \\
&
    F^{(n)}_{-} (t, x, v)) = (1-\varepsilon) F^{(n)}_{+} (t, x, R_x v),\\
&
     F^{(n)} (0, \cdot) = f_0 (\mu_n - 1),
\end{align*}
where
$$
   \text{Comm}_n =      - A_h [ \widetilde f_{\varepsilon, h} \mu_n]   + A_h [ \widetilde f_{\varepsilon, h}] \mu_n    - b^i [\mu_n (\delta_{ h, e_i} \widetilde f_{\varepsilon, h})  - \delta_{ h, e_i} \widetilde f_{\varepsilon, h}^{(n)}].
$$
Then, since $F^{(n)} \in E_2 (\Sigma^T)$, by the above $L_2$ estimate \eqref{eq7.5.0}, for $\lambda \ge \lambda_0$, we get
\begin{equation}
\begin{aligned}
             \label{eq7.5.1}
        \|F^{(n)} (T, \cdot)\|_{ L_2 (\Omega \times \bR^3) }  & +  \|F^{(n)}_{\pm}\|_{ L_2 (\Sigma^T_{\pm}, |v \cdot n_x|) } \\
	&+   \lambda \|F^{(n)}\|_{ L_2 (\Sigma^T) }
  	 \le  N \|f_0 (\mu_n  - 1)\|_{ L_2 (\Omega \times \bR^3) } \\
	& + N \|\sg  (\mu_n  - 1)\|_{ L_2 (\Sigma^T) } + N \|\text{Comm}_n\|_{ L_2 (\Sigma^T) },
\end{aligned}
\end{equation}
where $N$ is independent of $n$ and $F^{(n)}$.
Note that the first two terms on the right-hand side of \eqref{eq7.5.1} converge to $0$ as $n \to \infty$ by the dominated convergence theorem.
Thus, to prove that $\widetilde f_{\varepsilon, h}$ and $f_{\varepsilon, h}$ coincide a.e. in $\Sigma^T$
 along with their initial values and traces, it suffices to show  that
$\text{Comm}_n \to 0$ in $L_2 (\Sigma^T)$ as $n \to \infty$.

First, by  \eqref{7.5.1},  \eqref{7.1.3} and \eqref{7.1.6}, for a function $u$ on $\bR^3$, we have
\begin{align}
	& \delta_{ h, e_i} (u \mu_n) - \mu_n \delta_{ h, e_i} u = \delta_{h, e_i} \mu_n \,  T_h  u,\notag\\
    &
	\label{eq7.5.2}
	 A_h [u \mu_n] = a_k  (\Delta_{h, l_k} [u \mu_n])  + (\delta_{h, -l_k} a_k)\,  \delta_{h, -l_k} [u \mu_n].
\end{align}
Next, by using \eqref{7.5.1}, \eqref{7.1.3},  the product rule for the second-order differences
\begin{align*}
    \Delta_{h, l_k} [u \mu_n] &= \mu_n \Delta_{h, l_k} u + u  \Delta_{h, l_k} \mu_n + (\delta_{h, l_k} u)(\delta_{h, l_k} \mu_n)\\
&
    + (\delta_{h, -l_k} u)(\delta_{h, -l_k} \mu_n),
 \end{align*}
 and the product rule for the first-order differences (see \eqref{7.5.1}),
 we conclude
 \begin{align*}
  &     A_h [u \mu_n] =  \mu_n A_h u \\
 &
    + a_k \big( u  \Delta_{h, l_k} \mu_n + (\delta_{h, l_k} u)(\delta_{h, l_k} \mu_n) + (\delta_{h, -l_k} u)(\delta_{h, -l_k} \mu_n)\big)\\
  &
        +(\delta_{h, -l_k} a_k) (\delta_{h, -l_k} \mu_n) T_{h, -l_k} u.
   \end{align*}
Then, by  the above identity and the mean value theorem,  for $n \ge 1$,
$$
    \|\text{Comm}_n\|_{ L_2 (\Sigma^T) } \le N  n^{-1} \big(\|\widetilde f_{\varepsilon, h}\|_{ L_2 (\Sigma^T) } + \sum_{k = 1}^{d_1} \|\delta_{h, l_k} \widetilde f_{\varepsilon, h}|\|_{ L_2 (\Sigma^T)  }\big),
$$
where $N = N (h, \delta, K,  \Lambda)$. Hence $\text{Comm}_n \to 0$
as $n \to \infty$ in $L_2 (\Sigma^T)$.
Thus, the assertion of this step is proved.

\textit{Step 2: $L_p$ bound.}
Fix arbitrary $p \in [2, \infty)$.
  Let us start with the term containing $A_{h}$. By a change of variables
   and the inequality
$$
	(a-b) (|a|^q a - |b|^q b ) \ge N_0 (q)  (a-b)^2 (|a|^q + |b|^q), \, \, q \ge 0, a, b \in \bR, a\neq b,
$$
 we have
\begin{align*}
    &   p \int_{\bR^3}  \delta_{h, -l_k} (a_k  \delta_{ h, l_k} f_{\varepsilon, h})  |f_{\varepsilon, h}|^{p-2} f_{\varepsilon, h}\,   \, dv  \\
&	= p \int_{ \bR^3} a_k (\delta_{h, l_k}  f_{\varepsilon, h})   \, \delta_{h, l_k} (|f_{\varepsilon, h}|^{p-2} f_{\varepsilon, h})  \, dv
	\ge N_0 p \delta_1 \int_{\bR^3}  |f_{\varepsilon, h}|^{p-2} |\delta_{h, l_k} f_{\varepsilon, h}|^2 \, dv.
\end{align*}

Furthermore, by the  Young's inequality,
\begin{align}
&\label{eq7.5.4}
	p \int_{\bR^3} b^i (\delta_{h, e_i} f_{\varepsilon, h}) 	 |f_{\varepsilon, h}|^{p-2} f_{\varepsilon, h} \, dv \\
&	\ge  - N_0 p (\delta_1/2)  \int_{\bR^3}  |f_{\varepsilon, h}|^{p-2} |\delta_{ h, e_i } f_{\varepsilon, h}|^2 \, dv   -   2 N_0^{-1} \delta_1^{-1} K^2 p  \int_{\bR^3} |f_{\varepsilon, h}|^p \, dv \notag.
 \end{align}
Since the stencil $\Lambda$ contains the standard basis $e_i, i = 1, \ldots, d$,  we may replace $|\delta_{h,  e_i} f_{\varepsilon, h}|^2$  with  the sum $\sum_{k = 1}^{d_1} |\delta_{h, l_k} f_{\varepsilon, h}|^2$
in the first integral on the right-hand side of \eqref{eq7.5.4}.

 Thus, by the above
 and the Young's inequality, we obtain
\begin{equation}
			\label{eq7.5.3}
 \begin{aligned}
  &  \int_{\Omega \times \bR^3} |f_{\varepsilon, h}|^p (T, x, v) \, dx dv
	+ p N_0 (\delta_1/2) \int_{\Sigma^T} |f_{\varepsilon, h}|^{p-2} |\delta_{h, l_k} f_{\varepsilon, h}|^2 \, dz\\
&
  +   \varepsilon  \int_{\Sigma^T_{+}}    |(f_{\varepsilon, h})_{+}|^p  \, |v \cdot n_x| \, d\sigma  dt+ p (\lambda/2   - N_1 (K, \delta_1) )
    \int_{\Sigma^T}
     |f_{\varepsilon, h}|^p \, dz \\
 &
 \le   2^{p-1} \lambda^{1-p} \int_{ \Sigma^T }  |\sg|^p  \, dz  + \int_{\Omega\times \bR^3}  |f_0|^p  \, dz.
 \end{aligned}
\end{equation}
Finally, we set  $\lambda_0 =  2 N_1$,
 and take the $p$-th root in the above inequality and pass to the limit as $p \to \infty$.
 This gives the desired $L_{\infty}$ bound.
\end{proof}

\subsection{Verification of the weak formulation}

The goal of this section is to prove the following Green's type identity for the solution $f_{\varepsilon, h}$ (see \eqref{eq7.7.1}).
This result is used in the proof of  Theorem \ref{theorem 3.3} when we show that a solution constructed by the weak* compactness method  satisfies the weak formulation (see \eqref{7.0.0} in Definition \ref{definition 7.0}).

By $C^{1, 1, 2}_0 (\overline{\Sigma^T})$ we denote the space of all functions $\phi$ with compact support such that $\partial_t \phi$,$\nabla_x \phi$, $\nabla_v \phi$, $D^2_v \phi \in C (\overline{\Sigma^T})$.
\begin{lemma}
                \label{lemma 7.7}
Under the assumptions of Proposition \ref{proposition 7.4},
for any $\phi \in C^1_0 (\overline{\Sigma^T})$,
\begin{align}
        			\label{eq7.7.1}
 	    &	
 	     -  \int_{\Sigma^T} (Y \phi)    f_{\varepsilon, h} \, dz
    + \int_{\Omega \times \bR^3 }    (f_{\varepsilon, h}  (T, x, v) \phi (T, x, v) -   f_{\varepsilon, h}  (0, x, v) \phi (0, x, v)) \, dx dv \notag
	\\
 	   &
		+ \int_{\Sigma^T_{+}}    (f_{\varepsilon, h})_{ +} \phi \, |v \cdot n_x|  d\sigma dt
	-    \int_{\Sigma^T_{-}}     (f_{\varepsilon, h})_{-} \phi \, |v \cdot n_x| d\sigma dt\\
	&
	-\int_{\Sigma^T} f_{\varepsilon, h} A_h \phi \, dz
	+\int_{\Sigma^T}  f_{\varepsilon, h} \delta_{ h, -e_i} (b^i \phi) \, dz
	+ \lambda  \int_{\Sigma^T} f_{\varepsilon, h} \phi \, dz
	 = \int_{\Sigma^T} \sg \phi \, dz\notag.
	\end{align}
\end{lemma}

It follows from   \eqref{7.2} that \eqref{eq7.7.1} holds for any $\phi \in \Phi$ (see Definition \ref{definition 2.2.3}).
To prove Lemma \ref{lemma 7.7}, it suffices to extend the identity \eqref{7.2} for $\phi \in C^1_0 (\overline{\Sigma^T})$, which is done in the next lemma.

\begin{lemma}
            \label{lemma 7.6}
For any  $u \in E_{2} (\Sigma^T)$,
the Green's formula
\eqref{7.2}
holds for any
$
\phi \in C^{1}_{0} (\overline{ \Sigma^T}).
$
\end{lemma}

\begin{proof}
Let $\eta_j, j = 1, \ldots, m$
be a partition of unity in $\Omega$. Since
$$
    \sum_j \eta_j (Y \phi)
    = \sum_j Y (\phi \eta_j)
    - \sum_j \phi v \cdot \nabla_x \eta_j = \sum_j Y (\phi \eta_j),
$$
we may assume that
$\text{supp} \,  \phi \subset [0, T]\times \overline{B_r (x_0)} \times \overline{B_r (v_0)}$,
where $x_0 \in \partial \Omega$,
and $r > 0$ is sufficiently small.
Furthermore, we may also assume that
\eqref{eq1.4.1} holds. We will reduce the problem to the case when $\Omega  = \bR^3$
by using the boundary flattening diffeomorphism $\Pi$ defined below, which is somewhat simpler than $\Psi$ defined in  \eqref{eq1.4.2.4} - \eqref{eq1.4.2.8}.
 The same argument also works if we use $\Psi$ instead of $\Pi$ but require $\Omega$ to be a $C^3$ domain.

Let $\Pi: \Omega \cap B_{ r_0} (x_0) \times \bR^3 \to \bR^3_{-}
\times \bR^3,  (x, v) \to (y, w)$ be the diffeomorphism defined as $\Psi$ in Subsection \ref{subsection 1.4} but with $\psi^{-1}$ replaced with
\begin{equation}
			\label{eq3.11.1}
    \pi^{-1} (y) = (y_1, y_2, y_3 + \rho (y_1, y_2)).
\end{equation}
In particular,
$$
	w = (D \pi) v.
$$
We fix any smooth function $\xi$  on $\bR$ such that
$$
\int \xi' (t) \, dt = 1, \quad \quad
    \begin{cases}
    \xi (t)  = 0,\quad \quad t \le 1,\\
     \xi (t) \in (0, 1), \, \, t \in (1, 2), \\
    \xi (t) = 1, \quad  \quad t \ge 2.
    \end{cases}
$$
Furthermore, for $\varepsilon >0$, denote
\begin{align*}
  & \xi_{\varepsilon} (y, w) = \xi \big(\frac{y^2_3 + w_3^2}{\varepsilon^2}\big),
    \quad
     \chi_{\varepsilon} (x, v) =  \xi_{\varepsilon} (\Pi (x, v)),\\
     &
     \phi_{\varepsilon} (z) = \phi (z) \chi_{\varepsilon} (x, v) \xi \bigg(\frac{t}{\varepsilon}\bigg) \xi \bigg(\frac{T-t}{\varepsilon}\bigg).
\end{align*}
Note that $\phi_{\varepsilon}$ vanishes near $t = 0$, $t = T$ and the grazing set $(0, T) \times \gamma_0$ (see \eqref{eq1.2.0}).
By Remark \ref{remark 2.2.4},
 $\phi_{\varepsilon} \in \mathsf{\Phi}$ (see Definition \ref{definition 2.2.3}).
Therefore, by \eqref{7.2},
 \begin{align*}
      		\int_{\Sigma^T} \big[ (Y u)  \phi_{\varepsilon}
  +   (Y \phi_{\varepsilon})    u\big] \, dz
	-    \int_{\Sigma^T_{+}}     u_{+} \phi_{\varepsilon} \, |v \cdot n_x|\, d\sigma  dt
	+  \int_{\Sigma^T_{-}}    u_{-} \phi_{\varepsilon} \, |v \cdot n_x|\, d\sigma  dt = 0.
\end{align*}
 We will pass to the limit as $\varepsilon \to 0$ in this identity.
Note that by the dominated convergence theorem we only need to show that
\begin{align*}
 &\lim_{\varepsilon \to 0}     \int_{\Sigma^T}
 (Y \phi_{\varepsilon})    u  \,  dz\\
& = \int_{\Sigma^T}  (Y \phi) u  \, dz -   \int_{\Sigma^T} \big[u (T, x, v) \phi (T, x, v) - u (0, x, v) \phi (0, x, v)\big] \, dxdv.
\end{align*}
Invoke the notation of Subsection \ref{subsection 1.4.1}
and replace the diffeomorphism $\Psi$ with $\Pi$.
Then, by \eqref{eq.A1} and the fact that $J  = |\frac{\partial x}{\partial y}|^2   \equiv 1$, we get
\begin{equation}
			\label{eq7.6.3}
\begin{aligned}
 &  \int_{\Sigma^T}
 (Y \phi_{\varepsilon})    u  \,  dz    = \int_{  \Sigma^T }  (\partial_t   \phi_{\varepsilon}) u \, dy dw dt \\
&
    \quad+ \int_{ \mathbb{H}^T_{-} }  (w \cdot \nabla_y  \widehat \phi_{\varepsilon}) \widehat u \, dy dw dt - \int_{ \mathbb{H}^T_{-} }  (X \cdot \nabla_w \widehat \phi_{\varepsilon}) \widehat u \, dy dw dt
    =:  I_1 + I_2 + I_3,
\end{aligned}
\end{equation}
where $\widehat u$ is defined in \eqref{eq1.4.2.1}.

\textit{Convergence of $I_1$.}
We split $I_1$ into 5 integrals given by
\begin{align*}
& I_{1, 1} = \int_{ \Sigma^T } \xi \big(\frac{t}{\varepsilon}\big) \xi \big(\frac{T-t}{\varepsilon}\big) u  \partial_t \phi \chi_{\varepsilon}  \,  dz, \\
    &    I_{1, 2} = \varepsilon^{-1} \int_{ \Sigma^T } \xi' \big(\frac{t}{\varepsilon}\big) \xi \big(\frac{T-t}{\varepsilon}\big) u (0, x, v) \phi \chi_{\varepsilon}  \, dz,\\
    & I_{1, 3} = - \varepsilon^{-1}\int_{ \Sigma^T } \xi' \big(\frac{T-t}{\varepsilon}\big)  \xi \big(\frac{t}{\varepsilon}\big) u (T, x, v) \phi \chi_{\varepsilon} \, dz,\\
    & I_{1, 4} =  \varepsilon^{-1} \int_{ \Sigma^T } \xi' \big(\frac{t}{\varepsilon}\big)  \xi \big(\frac{T-t}{\varepsilon}\big)
    \big[u (t, x, v) -  u (0, x, v)] \phi \chi_{\varepsilon} \, dz,\\
    &  I_{1, 5} = - \varepsilon^{-1} \int_{ \Sigma^T } \xi \big(\frac{t}{\varepsilon}\big)  \xi ' \big(\frac{T-t}{\varepsilon}\big)
    \big[u (t, x, v) -  u (T, x, v)] \phi \chi_{\varepsilon} \, dz.
\end{align*}
The first integral converges to $\int_{ \Sigma^T }  u  \partial_t \phi   \, dy dw dt$ due to the dominated convergence theorem.
By a change of  variables and the dominated convergence theorem combined with the fact that
$\int \xi' (t) \, dt = 1$, we conclude that
$$
    I_{1, 2} \to   \int_{\Sigma^T } u (0, x, v) \phi (0, x, v) \, dxdv, \quad
     I_{1, 3} \to - \int_{\Sigma^T } u (T, x, v) \phi (T, x, v) \, dxdv.
$$
Changing variables and using the Cauchy-Schwartz inequality give
$$
    |I_{1, 4}| \le  N \varepsilon^{-1} \int_{\varepsilon}^{2 \varepsilon} \|u (t, \cdot) - u (0, \cdot)\|_{ L_2 (\Omega \times \bR^3)  } \, dt.
$$
The expression on the right-hand side of the above inequality converges to $0$ as $\varepsilon \to 0$ because $u \in C ([0, T], L_2 (\Omega \times \bR^3))$ (see Lemma \ref{lemma B.4}).
 Similarly, the same convergence holds for $I_{1, 5}$.
Hence, by the above, we conclude
\begin{equation}
			\label{eq7.6.2}
I_1 \to   \int_{\Sigma^T } [u (0, x, v) \phi (0, x, v) - u (T, x, v) \phi (T, x, v)] \, dxdv.
\end{equation}

\textit{ Convergence of $I_2$ and $I_3$.}
Note that
 \begin{align*}
&	I_2 =   I_{2, 1} + I_{2, 2} : =   \int_{\bH^T_{-}}  (w \cdot \nabla_y \widehat \phi) \xi_{\varepsilon} (y, w) \xi \big(\frac{t}{\varepsilon}\big) \xi\big(\frac{T-t}{\varepsilon}\big) \widehat u \, dydwdt  \\
	&\quad +  2\int_{   \bH^T_{-}:  \varepsilon^2 < y_3^2 + w_3^2 < 2 \varepsilon^2 }  \widehat u  \widehat \phi   \xi \big(\frac{t}{\varepsilon}\big) \xi \big(\frac{T-t}{\varepsilon}\big)   \frac{ w_3  y_3}{\varepsilon^2}
	\, \xi' \big(\frac{y_3^2 + w^2_3}{\varepsilon^2}\big) \,  dydwdt,\\
&
    I_3 =  I_{3, 1} + I_{3, 2} : =   - \int_{ \mathbb{H}^T_{-} }  (X \cdot \nabla_w \widehat \phi)\xi_{\varepsilon} (y, w)   \xi \big(\frac{t}{\varepsilon}\big) \xi \big(\frac{T-t}{\varepsilon}\big)    \widehat u \,  dydwdt\\
&- 2 \varepsilon^{-2}
    \int_{  \bH^T_{-}: \varepsilon^2 < y_3^2 + w_3^2 < 2 \varepsilon^2 }
     \widehat u  \widehat \phi \xi \big(\frac{t}{\varepsilon}\big) \xi \big(\frac{T-t}{\varepsilon}\big)
    X_3  w_3 \,  \xi' \big(\frac{ y_3^2 + w_3^2}{\varepsilon^2}\big) \,  dydwdt.
\end{align*}
By the dominated convergence theorem,
\begin{align*}
\lim_{\varepsilon \to 0} I_{2, 1} + I_{3, 1} =  \int_{\bH^T_{-}}  (w \cdot \nabla_y \widehat \phi)  \widehat u \, dydwdt
	- \int_{ \mathbb{H}^T_{-} }  (X \cdot \nabla_w \widehat \phi)   \widehat u \,  dydwdt.
\end{align*}
Then, by  the above equality,  \eqref{eq7.6.3} -  \eqref{eq7.6.2}, it suffices to show that
\begin{equation}
			\label{eq7.6.4}
I_{2, 2}, I_{3, 2} \to  0.
\end{equation}
By the dominated convergence theorem,
$$
	|I_{2, 2 }| \le  \int_{    \bH^T_{-}: \varepsilon^2 < |y_3^2 + w_3^2| < 2 \varepsilon^2 } |\widehat u \widehat \phi| \,   dydwdt \to 0 \quad \text{as} \, \, \varepsilon \to 0.
 $$
Furthermore, since $\Omega$ is a $C^{2}$ bounded domain, it follows from \eqref{eq3.11.1} and  \eqref{eq1.4.3.5} that $X$ is  bounded on the support of $\widehat \phi$.
By this, the Cauchy-Schwartz inequality,
and a change of variables, we get
\begin{align*}
    |I_{3, 2 }|
    &\le
    N \varepsilon^{-1}
     \|\widehat u \, I_{ y_3^2 + w_3^2 <  2 \varepsilon^2 }\|_{ L_2 (\bH^T_{-}  )  }\\ &
    \times \|  X_3 \widehat \phi  \|_{   L_{2}^{t, y_1, y_2, w_1, w_2   } (\bR^5_T) L_{\infty}^{y_3, w_3} ((0, \infty) \times \bR)}   \|\xi' \big(\frac{   y_3^2 + w_3^2}{\varepsilon^2}\big) \|_{ L_2 (\bR^2) } \to 0
\end{align*}
as $\varepsilon \to 0$.
Thus, \eqref{eq7.6.4} holds, and the lemma is proved.
\end{proof}

\subsection{Proof of Theorem \ref{theorem 3.3}}
By Proposition \ref{proposition 7.4}, for any $\varepsilon, h \in (0, 1)$,
there exists a unique (strong) solution
 $f_{\varepsilon, h} \in E_{2, \theta} (\Sigma^T)$  to Eq. \eqref{7.1}.
 Furthermore, by the uniform bounds in Lemma \ref{lemma 7.5}, the Banach-Alaoglu theorem,
and the Eberlein-Smulian theorem, there exists a function $f$ that satisfies the condition $(\ref{i}$ of Definition \ref{definition 7.0} and
\begin{equation}
                \label{7.2.2}
    \begin{aligned}
&  f_{\varepsilon, h} \to f  \, \, \text{weakly in} \, \, L_{2, \theta} (\Sigma^T), \quad
f_{\varepsilon, h} (T, \cdot)
\to f (T, \cdot)
\, \, \text{weakly in} \, \, L_{2, \theta} (\Omega \times \bR^3),\\
&
 (f_{\varepsilon, h})_{ \pm} \to f_{\pm}  \, \, \text{in the weak* topology of} \, \, L_{\infty} (\Sigma^T_{\pm}, |v \cdot n_x|).
\end{aligned}
\end{equation}
It remains  to show that $f$ satisfies conditions $(\ref{ii}$ and $(\ref{iii}$
of Definition \ref{definition 7.0}.

Fix any $\phi \in C^{1, 1, 2}_0 (\overline{\Sigma^T})$. Then, by  Lemma \ref{lemma 7.7}, the weak formulation \eqref{eq7.7.1} holds.
By  \eqref{7.2.2}, we pass to the limit  as $\varepsilon, h \to 0$   in  \eqref{7.1} and conclude that the specular reflection boundary condition is satisfied.
Next, note that by  \eqref{eq2.2.0} in Assumption \ref{assumption 2.2}, $\nabla_v b \in L_{\infty} (\Sigma^T)$, and, therefore,
	$$
		\delta_{ h, -e_i} (b^i \phi) \to -\partial_{v_i} (b^i \phi) \quad \text{as} \, \, h \to 0 \, \,  \text{in}\, \,L_{\infty} (\Sigma^T).
	$$
Combining this with \eqref{eq2.1.1} and \eqref{7.2.2}, we pass to the limit in  \eqref{eq7.7.1} (with respect to a subsequence).
This proves the validity of the condition $(\ref{iii}$ of Definition \ref{definition 7.0} with $\phi \in C^{1, 1, 2}_0 (\overline{ \Sigma^T})$. By using a limiting argument, we prove that the weak formulation \eqref{7.0.0} holds with $\phi \in C^1_0 (\overline{ \Sigma^T})$.
Finally, passing to the limit in the bounds in Lemma \ref{lemma 7.5}, we obtain the estimates \eqref{7.2.1}.

\section{Regularity of the finite energy weak solutions to the kinetic Fokker-Planck equation}
                                                        \label{section 3}
Here we prove Theorem \ref{theorem 7.3}, Corollary \ref{corollary 7.3.1}, and Theorem \ref{theorem 1.9}.

\begin{proof}[Proof of Theorem  \ref{theorem 7.3}]

Let $\zeta \in C^{\infty}_0 (\bR)$  such that $\zeta (0) = 1$.
Replacing $f$ with $f - f_0 \zeta$ and $\sg$ with
$$
    	\sg   - f_0 \partial_t \zeta - v \cdot \nabla_x f_0 \zeta   + \Delta_v (f_0 \zeta)  - b \cdot \nabla_v (f_0 \zeta)  -  \lambda (f_0 \zeta) \in L_{2, \theta} (\Sigma^T),
$$
 we may assume that $f_0 \equiv 0$.

Let $\eta_k, k =  1, \ldots, m,$ be the standard partition of unity
	 in $\Omega$ (see, for example, Section 8.4 of \cite{Kr_08})
	 such that
	 supp $\eta_1 \subset \Omega$,
	$0 \le \eta_k \le 1$, $k = 1, \ldots, m$,
	and
\begin{equation}
			\label{eq1.7.11}
  |\nabla_x \eta_k| \le N/r_0, \quad
  \begin{cases}
    \eta_k = 1\quad \text{in} \, \,  B_{r_0/4} (x_k)\\
    \eta_k = 0 \quad  \text{in} \, \,  B_{ r_0/2 }^c (x_k)
    \end{cases},\quad k  = 2, \ldots, m,
\end{equation}
	    where  $x_i \in \partial \Omega$,
	 and $r_0$ is the number in Subsection \ref{subsection 1.4.1} such that \eqref{eq1.4.1} and \eqref{eq10.19} hold.

    Note that
	\begin{equation}
					\label{eq1.7.10}
		f_k := f \eta_k \jb^{\theta-2}
	\end{equation}
    satisfies the identity
    \begin{equation}
                    \label{eq1.7.2}
	\begin{aligned}
            &Y f_k  - \Delta_v f_k  + b \cdot \nabla_v f_k + \lambda f_k \\
	&= \sg\eta_k \jb^{\theta-2} +  f (\jb^{\theta-2} v \cdot \nabla_x \eta_k  \\
	&+ \eta_k  b \cdot \nabla_v \jb^{\theta-2} - \Delta_v \jb^{\theta-2} \eta_k) - 2 \nabla_v f \cdot \nabla_v \jb^{\theta-2} \eta_k =: \sg_k
	\end{aligned}
    \end{equation}
    in the weak sense, i.e.,
       \begin{align*}
      & - \int_{\Sigma^T} (Y \phi) f_{k} \, dz +  \int_{    \Omega \times \bR^3 } f_k (T, x, v) \phi (T, x, v) \, dx dv \\
    &
        + \int_{\Sigma^T } \nabla_v   f_{k} \cdot \nabla_{v} \phi  \, dz
	+ \int_{\Sigma^T}   (b \cdot \nabla_{v} f_{k}) \phi  \, dz
        + \lambda  \int_{  \Sigma^T } f_{k} \phi \, dz
        =  \int_{ \Sigma^T } \sg_{k} \phi \, dz.
     \end{align*}

     \textbf{Interior estimate.}
     We extend $f_1$ and $\sg_1$ for negative $t$ by replacing them with $f_1 1_{t \ge 0}$ and $\sg_1 1_{t \ge 0}$, respectively. 	
Since  $f (0, \cdot) \equiv 0$, by the chain rule for distributions,
	we conclude that $f_1 \in  \bS_2 ((-\infty, T) \times \bR^6)$.
     Then,  $f_1$ satisfies \eqref{eq1.7.2} in $\bH^{-1}_2 ((-\infty, T) \times \bR^6)$.

	Let $
 	  \lambda_0  > 0
	$
	be the constant in Theorem \ref{theorem D.1} with $p  = 2$.
	Then, due to  Theorem \ref{theorem D.1} $(ii)$,
	for any $\lambda \ge \lambda_0$, the equation
	$$
		Y u_1 - \Delta_v u_1 + b \cdot \nabla_v u_1 + \lambda u_1 =  \sg_1, \quad u_1 (0, \cdot, \cdot) \equiv 0
	$$
	has a unique solution $u_1 \in S_2 ((-\infty, T) \times \bR^6)$.
	Then, $U_1 = f_1 - u_1 \in  \bS_2 ((-\infty, T) \times \bR^6)$ satisfies the equation
	$$
			Y U_1 - \Delta_v U_1 + b \cdot \nabla_v U_1 + \lambda U_1  = 0 \quad \text{in} \, \, \bH^{-1}_2 ((-\infty, T) \times \bR^6), \quad  U_1 (0, \cdot, \cdot) \equiv 0.
	$$
	Therefore, by the energy identity of Lemma \ref{lemma B.5}, integration by parts, and the Cauchy-Schwartz inequality,
	for a.e. $s \in (-\infty, T)$, we obtain
	\begin{align*}
		\int_{  \bR^6}  U_1^2 (s, x, v) \, dxdv  &+  \int_{  (-\infty, s) \times \bR^6 }  |\nabla_v U_1|^2 \, dz \\
		&+ (\lambda - N_1 (K)) \int_{  (-\infty, s) \times \bR^6 } U_1^2 \, dz \le 0.
	\end{align*}
	It follows that for $\lambda >   N_1 (K)$,  we have $U_1 = 0$ a.e. in $(-\infty, T) \times \bR^6$.
	Then, by Theorem \ref{theorem D.1} $(i)$,
	  for $\lambda \ge \lambda_0$  with, possibly, larger $\lambda_0$,
	 one has
    \begin{equation}
                                        \label{eq1.7.7}
	\begin{aligned}
       &   \|f_1\|_{S_2 (\Sigma^T) }  =    \|u_1\|_{S_2 (\Sigma^T) }  \le  N \|\sg_1\|_{ L_2 (\Sigma^T)} \\
	 & \le N  (\|\sg \|_{ L_{2, \theta - 2} (\Sigma^T) } + \| f\|_{ L_{2, \theta - 1} (\Sigma^T)} + \|\nabla_v f\|_{ L_{2, \theta-3} (\Sigma^T)}),
	\end{aligned}
    \end{equation}
	where $N = N (K, \Omega, \theta)$.
	In addition, by the embedding theorem for $S_p$ spaces (see Theorem 2.1 of \cite{PR_98}),
	$$
		\||f_1| + |\nabla_v f_1| \|_{ L_{7/3} (\Sigma^T)}
	$$
	is bounded by the right-hand side of \eqref{eq1.7.7}.

    \textbf{Boundary estimate.} We fix some $k \in \{2, \ldots, m\}$.
    We redo the construction of the mirror extension mapping in Section \ref{subsubsection 1.4.2} with $x_0$ replaced with $x_k$.
    Then, by the argument of Section \ref{subsubsection 1.4.3}, $\overline{f_k}$ satisfies the equation
  \begin{equation}
				\label{eq1.7.1}
       \partial_t \overline{f_k} + w \cdot \nabla_y \overline{f_k} -  \nabla_w \cdot (\mathbb{A} \nabla_w \overline{f_k}) + \mathbb{B} \cdot \nabla_v \overline{f_k} + \lambda \overline{f_k}
       = \overline{\sg_k} - \nabla_w \cdot (\mathbb{X} \overline{f_k})
\end{equation}
	in $\bH^{-1}_2 ((-\infty, T) \times \bR^6)$   and  $\overline{f_k} (0, \cdot, \cdot) = 0$.
    Here
\begin{itemize}
\item[--] $\mathbb{A}$ is defined by \eqref{eq1.4.3.2}, \eqref{eq1.4.3.3}, \eqref{eq1.20} with $a = I_3$,
\item[--] $\mathbb{B}$ is defined by formulas \eqref{eq1.4.3.2}, \eqref{eq1.4.3.4},
\item[--] $\mathbb{X}$ is given by \eqref{eq1.4.3.5}, \eqref{eq1.4.3.6}.
\end{itemize}
First, observe that
since $\psi$ is a local $C^3$ diffeomorphism,
\begin{equation}
			\label{eq1.7.4}
    |\mathbb{X}| \le N (\Omega) |w|^2, \quad   |\nabla_w \mathbb{X}|\le  N (\Omega) |w|,
\end{equation}
and, therefore by \eqref{eq1.7.4} and \eqref{eq1.7.2},  the right-hand side of \eqref{eq1.7.1} belongs to $L_2 ((-\infty, T) \times \bR^6)$.
   Next, note that $\mathbb{A}$ satisfies  the nondegeneracy condition (see \eqref{eq2.1.0} in Assumption  \ref{assumption 2.1}) with $\delta = \delta (\Omega) > 0$ because $\psi$ defined by \eqref{eq1.4.2.8} is a local diffeomorphism.
	Furthermore, by the conclusion of Appendix \ref{Appendix E},   $\mathbb{A} \in L_{\infty} ((0, T), C^{\varkappa/3, \varkappa}_{x, v} (\overline{\Omega} \times \bR^3))$,
and, thus,  by Remark \ref{remark 3.1}, Theorem \ref{theorem D.1} $(ii)$ is applicable.
	By this theorem,  there exist $\lambda_0  = \lambda_0 (K, \Omega) > 0$ such that for any $\lambda \ge \lambda_0$, the equation
	$$
		 \partial_t u_k + w \cdot \nabla_y u_k -  \nabla_w \cdot (\mathbb{A} \nabla_w u_k) + \mathbb{B} \cdot \nabla_w u_k + \lambda u_k =   \overline{\sg_k} - \nabla_w \cdot (\mathbb{X} \overline{f_k}),
		\quad u_k (0, \cdot, \cdot)  = 0,
	$$
	has a unique solution $u_k \in S_2 ((-\infty, T) \times \bR^6)$.
Then, repeating the bootstrap argument used for the interior estimate, we conclude that $\overline{f_k} \equiv u_k$.
Hence,  by Theorem \ref{theorem D.1} $(i)$,
\begin{equation}
			\label{eq1.7.8}
\begin{aligned}
    &    \|\overline{f_k}\|_{ S_2 ((0, T) \times \bR^6) }
    \le N  \bigg(\|\overline{\sg_k}\|_{ L_2 ((0, T) \times \bR^6) } \\
   & + \| \nabla_w \overline{f_k}\|_{ L_{2, 2} ((0, T) \times \bR^6) }
	+ \|\overline{f_k}\|_{ L_{2, 1} ((0, T) \times \bR^6) }\bigg),
\end{aligned}
\end{equation}
where $N = N (K, \Omega, \theta)$.
Again, by the embedding theorem for $S_p$ spaces,
	$$
		\||\overline{f_k}| + |\nabla_v \overline{f_k}| \|_{ L_{7/3} (\Sigma^T)}
	$$
is controlled by the right-hand side of \eqref{eq1.7.8}.
Then, by Lemma \ref{lemma B.2},  we have
\begin{align*}
        \|f_k\|_{ S_2 (\Sigma^T) } &+ \||f_k| + |\nabla_v f_k| \|_{ L_{7/3} (\Sigma^T)} \\
 &   \le N   \bigg(\|\sg\|_{ L_{2, \theta - 2} (\Sigma^T) }  + \| \nabla_v f\|_{ L_{2, \theta} (\Sigma^T) }
	+ \|f\|_{ L_{2, \theta-1} (\Sigma^T) }\bigg).
\end{align*}
Finally, combining the above estimate with \eqref{eq1.7.7}, we prove the theorem.
\end{proof}

\begin{proof}[Proof of Corollary \ref{corollary 7.3.1}]
										\label{proof of corollary 7.3.1}
Let $f_1, f_2$ be finite energy weak solutions to Eq. \eqref{7.0}.
Then, by Theorem \ref{theorem 7.3},
$f_1, f_2 \in S_2 (\Sigma^T)$.

Next, let $\lambda' > 0$ be a number, which we will choose later,
and denote
$F =(f_1 - f_2) e^{-\lambda' t}$.
Note that $F \in S_2 (\Sigma^T)$
satisfies Eq. \eqref{7.0} with $\sg \equiv 0$, $f_0 \equiv 0$,
and $\lambda + \lambda'$ in place of $\lambda$.
Then, by a variant of the energy identity (see Lemma \ref{lemma B.3}), we  get
\begin{align*}
    &    \int_{\Omega \times \bR^3} |F (T, x, v)|^2 \, dx dv
    + 2 (\lambda + \lambda') \int_{\Sigma^T}
    |F|^2 \, dz\\
 & -
    2 \int_{\Sigma^T}  (\Delta_{v}   F) F \, dz
+  2 \int_{\Sigma^T} (b \cdot  \nabla_{v} F) F\, dz
    = 0.
 \end{align*}
Integrating by parts and using the  Cauchy-Schwartz inequality,
we obtain
$$
    \int_{\Omega \times \bR^3} |F (T, x, v)|^2  \, dx dv
  +
      \int_{\Sigma^T}   |\nabla_{v} F|^2  \, dz
	+ 2 (\lambda + \lambda' - N (K)) \int_{\Sigma^T}
    |F|^2 \, dz
 \le 0.
 $$
Thus, taking $\lambda' >  N$, we conclude $F \equiv 0$. The corollary is proved.
\end{proof}

To prove Theorem \ref{theorem 1.9}, we need the following result.
\begin{lemma}
			\label{lemma 3.7}
Let
\begin{itemize}
\item[--] $p \ge 2,  T > 0, \lambda \ge 0, \theta \ge 2$ be  numbers, and $r > 1$ be determined by the relation
     \begin{equation}
                \label{7.3.0.0}
        \frac{1}{r} = \frac{1}{p} - \frac{1}{14},
     \end{equation}
\item[--] Assumptions \ref{assumption 2.1} - \ref{assumption 7.3} be satisfied,
\item[--] $f \in S_{p, \theta'} (\Sigma^T)$ with $\theta' \ge \theta - 2$, and $f, \nabla_v f  \in L_{r, \theta} (\Sigma^T)$,
\item[--] $f_{\pm} \in L_{\infty} (\Sigma^T_{\pm}, |v \cdot n_x|)$, $f (T, \cdot) \in L_{2} (\Omega \times \bR^3)$, $f (\cdot, 0) \equiv 0$,
\item[--]  $\sg \in  L_{r, \theta - 2} (\Sigma^T)$,
\item[--] $f$ satisfy Eq. \eqref{7.0} with  $f_0 \equiv 0$ in the weak sense (see Definition \ref{definition 7.0}).
\end{itemize}
Then,  the following assertions hold.

$(i)$  One has $f \in S_{r, \theta - 2} (\Sigma^T)$, and
 \begin{equation}
                \label{eq3.7.0}
    \|f\|_{ S_{r, \theta - 2} (\Sigma^T) } \le N  \bigg(\|\sg\|_{ L_{r, \theta - 2} (\Sigma^T) }
   +   \|\nabla_v  f\|_{ L_{r, \theta} (\Sigma^T) }
	+ 	\|f\|_{ L_{r, \theta - 1} (\Sigma^T)  }\bigg),
 \end{equation}
 where $N = N (p, K, \theta, \Omega)$.

$(ii)$ If  $p < 14$, then, $f, \nabla_v f \in L_{r, \theta - 2} (\Sigma^T)$,
    and, furthermore,
$$
\||f| + |\nabla f|\|_{  L_{r, \theta - 2} (\Sigma^T)  }
$$
is less than the right-hand side of \eqref{eq3.7.0}.

 $(iii)$    If $p > 14$, then, $f, \nabla_v f \in L_{\infty} ((0, T), C^{\alpha/3, \alpha}_{x, v} (\overline{\Omega} \times \bR^3) )$,
     where $\alpha = 1  - 14/p$, and $C^{\alpha/3, \alpha}_{x, v} (\overline{\Omega} \times \bR^3)$ is defined in \eqref{1.2.0}.
     In addition, the norms
$$
	\|f\|_{  L_{\infty} ((0, T), C^{\alpha/3, \alpha}_{x, v} (\overline{\Omega} \times \bR^3) )  }, \|\nabla f|\|_{  L_{\infty} ((0, T), C^{\alpha/3, \alpha}_{x, v} (\overline{\Omega} \times \bR^3) )  }
$$
are bounded above by the right-hand side of \eqref{eq3.7.0}.
\end{lemma}

\begin{proof}[Proof of Lemma \ref{lemma 3.7}]
Denote
\begin{equation}
			\label{eq7.2.6}
    B^p = \begin{cases}
            L_r ((-\infty, T) \times \bR^6), \, \, \text{if} \, \, p \in [2, 14),   \text{where} \, \,  r \, \,  \text{is determined by} \, \eqref{7.3.0.0},\\
           C^{\alpha/3, \alpha}_{x, v} ((-\infty, T) \times \bR^6),  \, \,  \text{if} \, \, p > 14, \, \, \text{where} \, \,  \alpha = 1-14/p.
    \end{cases}
\end{equation}
We follow the argument of Theorem \ref{theorem 7.3}  closely. Again, as in the proof of that theorem, we may assume that $f_0 \equiv 0$.

\textbf{Interior estimate.}
Recall that
\begin{equation}
			\label{eq7.2.8}
	f_1 = f \eta_1 \jb^{\theta - 2} \in S_p ((0, T) \times \bR^6)
\end{equation}
 solves Eq. \eqref{eq1.7.2} with the zero initial-value condition
and note that due to the assumptions of this lemma, for any $k  = 1, \ldots, m$, we have
\begin{align*}
	\sg_k &=  \sg \eta_k \jb^{\theta-2} +  f (\jb^{\theta-2} v \cdot \nabla_x \eta_k  \\
	&+ \eta_k  b \cdot \nabla_v \jb^{\theta-2} - \Delta_v \jb^{\theta-2} \eta_k) - 2 \nabla_v f_k \cdot \nabla_v \jb^{\theta-2} \eta_k
 \in L_r ((0, T) \times \bR^6).
\end{align*}
Then, due to Lemma \ref{lemma B.6},   $f_1 \in S_r ((0, T) \times \bR^6)$.
Using  Theorem \ref{corollary 3.4} combined with Remark \ref{remark 3.4}, we get
    \begin{equation}
	\begin{aligned}
                                        \label{eq7.2.2}
          &\|f_1\|_{S_r (\Sigma^T)}
       \le  N (p, K, \theta, \Omega) \bigg(\|\sg\|_{ L_{r, \theta-2} (\Sigma^T) } \\
	&+ \|f\|_{ L_{r, \theta-1} (\Sigma^T) }
       +  \|\nabla_v f\|_{ L_{r, \theta-3} (\Sigma^T) }\bigg).
    \end{aligned}
  \end{equation}

\textbf{Boundary estimate.}
Recall that  $f_k$ is given by  \eqref{eq1.7.10}, and $\overline{f_k}$ is its ``mirror extension'' defined as in \eqref{eq1.4.2.2}.
The function $\overline{f_k} \in S_p ((-\infty, T) \times \bR^6)$ satisfies Eq. \eqref{eq1.7.1}.
In addition, due to \eqref{eq1.7.4},
$$
	 \overline{\sg_k} - \nabla_w \cdot (\mathbb{X} \overline{f_k}) \in L_r ((-\infty, T) \times \bR^6).
$$
 Then, again, by Lemma \ref{lemma B.6}, we conclude that $\overline{f_k}  \in S_r ((-\infty, T) \times \bR^6)$.

Next, by the a priori estimate of Corollary \ref{corollary 3.4} applied to Eq.  \eqref{eq1.7.1} and using \eqref{eq1.7.4}, we get
\begin{align*}
        \|\overline{f_k}\|_{ S_r ((-\infty, T) \times \bR^6) }
    \le N  &  \big(\|\overline{\sg_k}\|_{ L_r ((-\infty, T) \times \bR^6) }
	+  \| \nabla_w \overline{f_k}\|_{ L_{r, 2} ((-\infty, T) \times \bR^6) }\\
   &
	+ \|\overline{f_k}\|_{ L_{r, 1} ((-\infty, T) \times \bR^6) }),
\end{align*}
where $N = N (p, K, \theta, \Omega) > 0$.
Then, by this and Lemma \ref{lemma B.2},
\begin{equation}
			\label{eq7.2.4}
\begin{aligned}
      \|f_k\|_{ S_r (\bR^7_T) }
    \le N  \big(\|\sg\|_{ L_{r, \theta-2} (\Sigma^T) }
 	+  \| \nabla_v f\|_{ L_{r, \theta} (\Sigma^T) }
	+ \| f\|_{ L_{r, \theta-1} (\Sigma^T) }\big).
\end{aligned}
\end{equation}
Thus, the desired estimate \eqref{eq3.7.0} follows from  \eqref{eq7.2.2} combined with \eqref{eq7.2.4}.
As in the proof of Theorem \ref{theorem 7.3}, by using the embedding theorem for the $S_p$ space (Theorem 2.1 of \cite{PR_98}),
we prove the assertions $(ii)$ and $(iii)$.
\end{proof}

\begin{proof}[Proof of Theorem \ref{theorem 1.9}]
We will proof the theorem by using a bootstrap argument.
As in the proof of Theorem \ref{theorem 7.3}, we may assume that $f_0 \equiv 0$.

Let  $r_k, k \ge 1$  be the numbers defined as follows:
\begin{equation}
			\label{8.1.7}
\begin{aligned}
&    r_1 = 2, \quad    \frac{1}{r_k} = \frac{1}{r_{k-1}} - \frac{1}{14},\, \,  k = 2, \ldots, 6, \\
&
    \frac{1}{r_7} = \frac{1}{p} + \frac{1}{14}, \, \,
r_8 = p.
\end{aligned}
\end{equation}
Note that $r_6 = 7$, $r_7 \in \big(7, 14)$ since $p > 14$.
First, by Theorem \ref{theorem 7.3},  the inequality \eqref{7.3.0} holds.
By this estimate,   Lemma \ref{lemma 3.7},  and an induction argument for $k = 1, \ldots, 6$, we conclude
\begin{equation}
			\label{8.1.6}
    \begin{aligned}
&  \|f\|_{ S_{7, \theta - 12} (\Sigma^T)} + \||f| + |\nabla_v f|\|_{L_{14, \theta - 12} (\Sigma^T) }\\
  & \quad \le N   \sum_{k = 1}^6 \|\sg\|_{ L_{r_k, \theta-2 k} (\Sigma^T)  }+ N \|\nabla_v f\|_{L_{2, \theta} (\Sigma^T) }
 +  N \| f \|_{ L_{2, \theta - 1} (\Sigma^T)}.
 \end{aligned}
\end{equation}
Then, by \eqref{8.1.6} and the interpolation inequality, $f, \nabla_v f \in L_{r_7, \theta - 12} (\Sigma^T)$.
Again, by Lemma \ref{lemma 3.7}, \eqref{8.1.6}, and the fact that $r_8 = p$,
we get
\begin{equation}
	            \label{8.1.5}
\begin{aligned}
    &\|f\|_{ S_{r_7, \theta-14} (\Sigma^T) }  +  \||f| + |\nabla_v f|\|_{  L_{p, \theta - 14} (\Sigma^T) }\\
&\le N  \big(\|\sg\|_{ L_{2, \theta-2} (\Sigma^T)  }
  + \|\sg\|_{ L_{p, \theta-4} (\Sigma^T)  }\\
 &\quad +\|\nabla_v f\|_{L_{2, \theta} (\Sigma^T) }
\quad +  \| f \|_{ L_{2, \theta - 1} (\Sigma^T)}\big).
\end{aligned}
\end{equation}
Finally, applying Lemma \ref{lemma 3.7} once more, we prove the theorem.
\end{proof}

\section{Proof of the main result for the linearized Landau equation}
						\label{section 4}
\subsection{Unique solvability result for the simplified viscous  linearized Landau equation}
																														\label{subsection 4.1}
Here, we prove the existence/uniqueness of the finite energy strong solutions to Eq. \eqref{eq1.18} in the sense of Definition \ref{definition 7.0}.
We follow the scheme that we used to show the existence and uniqueness  for the kinetic Fokker-Planck equation \eqref{7.0}.

\begin{proposition}[Existence of finite energy weak solution]
			\label{proposition 4.1}
Let
\begin{itemize}
\item[--] $\Omega$ be a bounded $C^2$ domain,
\item[--]  $T > 0$, $\nu \in (0, 1]$, $\varkappa \in (0, 1]$, $\theta \ge 0$  be numbers,
\item[--] Assumption  \ref{assumption 11.1} (see \eqref{con11.1} - \eqref{con11.1'})  be satisfied,
\item[--] $f_0 \in  L_{2, \theta} (\Omega \times \bR^3) \cap L_{\infty} (\Omega \times \bR^3)$, $h \in L_{2, \theta} (\Sigma^T) \cap L_{\infty} (\Sigma^T)$.
\end{itemize}
Then,  there exists  a number $\varepsilon \in (0, 1)$ (independent of $\Omega, T, \nu, \varkappa, \theta$) such that, if, additionally,
$$
	\|g\|_{  L_{\infty} ( \Sigma^T ) } \le \varepsilon,
$$
then, there exists $\lambda_0 = \lambda_0 (\nu, K, \theta) > 0$, such that
for any $\lambda \ge \lambda_0$, Eq. \eqref{eq1.18} has a   finite energy weak solution $f$ in the sense of Definition \ref{definition 7.0}, and, furthermore,
\begin{align}
	\label{eq4.1.2}
         & \|f (T, \cdot)\|_{ L_{2, \theta} (\Omega \times \bR^3) }
	+ \|\nabla_v f\|_{ L_{2, \theta} (\Sigma^T) } +    \lambda^{1/2} \|f\|_{ L_{2, \theta} (\Sigma^T) }\\
	&
	\le N  (\lambda^{-1/2}\|h\|_{ L_{2, \theta} (\Sigma^T) } + \|f_0\|_{  L_{2, \theta} (\Omega \times \bR^3)  }), \notag\\
	&\label{eq4.1.3}
	 \max\{ \|f\|_{ L_{\infty} (\Sigma^T)}, \| f_{\pm}\|_{ L_{\infty} (\Sigma^T_{\pm}, |v \cdot n_x|)  }\}  \\
  &   \le
        \lambda^{-1} \|h\|_{ L_{\infty} (\Sigma^T) } +  \|f_0\|_{  L_{\infty} (\Omega \times \bR^3)  },\notag
\end{align}
where $N =  N (\theta, K, \nu) > 0$.
\end{proposition}

\begin{proof}
We use Theorem \ref{theorem 3.3}. We need to check that $\sigma_G + \nu I_d$ and $a_g$ satisfy  Assumptions \ref{assumption 2.1} -  \ref{assumption 2.2} (see \eqref{eq2.1.0} - \eqref{eq2.2.0}).

\textit{Assumption \ref{assumption 2.1}.}
By Lemma 3 of \cite{G_02}, for sufficiently small $\varepsilon$, there exist constants $C_1, C_2 > 0$ such that
\begin{equation}
				\label{eq11.2.1}
	C_1 \jb^{-3} I_3 \le \sigma_G (z)  \le C_2 \jb^{-1} I_3, \, \, \forall z \in \bR^7_T.
\end{equation}
Then, \eqref{eq2.1.0} in Assumption \ref{assumption 2.1} holds with $\delta  = \nu$.

\textit{Assumption  \ref{assumption 7.3} and \ref{assumption 2.2}.}
 By using Lemma  3  in \cite{G_02} combined with Assumption \ref{assumption 11.1} (\eqref{con11.1} - \eqref{con11.1'}), one can show that for sufficiently small $\varepsilon > 0$,
 \begin{equation}
			\label{11.2.9}
	\||a_g| + |\nabla_v a_g| + |\nabla_v \sigma_G|\|_{  L_{\infty} (\bR^7_T) } \le N (K),
  \end{equation}
and, hence, \eqref{eq2.3.0} and \eqref{eq2.2.0}   hold.

Now all the assertions of this proposition follow directly from Theorem \ref{theorem 3.3}.
\end{proof}

The following proposition is analogous to Theorems \ref{theorem 7.3} and \ref{theorem 1.9}.
\begin{proposition}[$S_p$ regularity of the simplified linearized viscous Landau equation]
			\label{proposition 4.2}
Let
\begin{itemize}
\item[--] $\Omega$ be a bounded $C^3$ domain,
\item[--]  $T > 0$, $\nu  \in (0, 1]$, $\varkappa \in (0, 1]$, $\theta \ge 16$, $p > 14$  be numbers,
\item[--] $f_0 \in  \cO_{2, \theta-2} \cap \cO_{p, \theta-4}$, $h \in L_{2, \theta-2} (\Sigma^T) \cap L_{p, \theta-4} (\Sigma^T)$,
\item[--] Assumption \ref{assumption 11.1}  (\eqref{con11.1} - \eqref{con11.1'}) be satisfied,
\item[--] $	\|g\|_{  L_{\infty} ( \Sigma^T ) } \le \varepsilon$,
\item[--] $f$ be a finite energy  weak solution to Eq. \eqref{eq1.18} with parameter $\theta$ in the sense of Definition \ref{definition 7.0},  which exists due to Proposition \ref{proposition 4.1}.
\end{itemize}
Then,   there exists  a number $\varepsilon \in (0, 1)$ (independent of $\Omega, T, \nu, \varkappa, \theta$)
 and
  $\lambda_0 = \lambda_0 (\nu, K, \theta, \Omega) \ge 1$ such that  $f \in S_{2, \theta-2} (\Sigma^T) \cap S_{p, \theta - 16} (\Sigma^T)$,
and, furthermore,
\begin{equation}
			\label{eq4.2.1}
\begin{aligned}
	\|f\|_{  S_{2, \theta-2} (\Sigma^T)  } + \|f\|_{ S_{p, \theta - 16} (\Sigma^T)  }
	&\le N  (\|\nabla_v f\|_{ L_{2, \theta} (\Sigma^T) }  + \|f\|_{  L_{2, \theta-1} (\Sigma^T)  }\\
	& \quad+ \|h\|_{ L_{2, \theta - 2} (\Sigma^T)  } + \|h\|_{  L_{p, \theta - 4} (\Sigma^T) }\\
		&\quad+ |f_0|_{  \cO_{2, \theta - 2} } + |f_0|_{  \cO_{p, \theta - 4} }),
\end{aligned}
\end{equation}
where $N = N (\nu, K, \theta, \varkappa, \Omega) > 0$.
\end{proposition}

\begin{proof}
We fix $\varepsilon > 0$ sufficiently small such that \eqref{eq11.2.1} holds.

\textbf{Step 1: $S_2$-regularity.}
We inspect the argument of Theorem \ref{theorem 7.3}.
As in the aforementioned theorem, since the right-hand side of \eqref{eq4.2.1} contains the terms $|f_0|_{ \cO_{2, \theta-2} }, |f_0|_{ \cO_{p, \theta-4} }$,  we may assume that $f_0 \equiv 0$.
  We use a partition of unity argument combined with the mirror extension method and the $S_p$ estimate of Theorem \ref{theorem D.1}.
Let $\eta_k, k  = 1, \ldots, m$ be the partition of unity defined in the proof of Theorem \ref{theorem 7.3} (see \eqref{eq1.7.11}).
Note that $f_k = f \eta_k \jb^{\theta-2}$ satisfies the equation
\begin{equation}
			\label{eq4.1.1}
	Y f_k -  \nabla_v \cdot \big((\sigma_G + \nu I_3) \nabla_v f_k\big) - a_g \cdot \nabla_v f_k + \lambda f_k  =: h_k,
\end{equation}
where
\begin{align*}
&	h_k =  h\eta_k \jb^{\theta-2} +  f ( \jb^{\theta-2} v \cdot \nabla_x \eta_k) -  (\sigma_G^{i j} + \nu \delta_{i j}) (\partial_{v_i v_j} \jb^{\theta-2}) f \eta_k  \\
	& +  (-\partial_{v_i} \sigma_G^{i j} - a_g^j) (\partial_{v_j} \jb^{\theta-2}) f \eta_k
	  - 2  (\sigma^{i j}_G + \nu \delta_{i j}) (\partial_{v_{  i }}  \jb^{\theta-2}) (\partial_{v_j} f) \eta_k.
\end{align*}

\textit{Interior estimate.}
We conclude that $f_1 \in S_2 ((-\infty, T) \times \Omega \times \bR^3)$ by using the same bootstrap argument as in the proof of  Theorem \ref{theorem 7.3}.
One minor difference is that one needs to apply Theorem \ref{theorem D.1} of Appendix \ref{Appendix D} with $a = \sigma_G + \nu I_3$, $b^{   i }  = -a_g^{i} - \partial_{v_j} \sigma^{ji}, i = 1, 2, 3$, and $c = 0$. Let us check the conditions of this theorem.
Note that Assumptions  \ref{assumption 2.1} (see \eqref{eq2.1.0}) and \ref{assumption 3.2} (see \eqref{eq3.2.0}) hold due to \eqref{eq11.2.1} and \eqref{11.2.9}.
Furthermore, by  Lemma \ref{lemma C.1},  $\sigma_G \in  L_{\infty} ((0, T), C^{\varkappa/3, \varkappa}_{x, v} (\bR^6))$ so that due to Remark \ref{remark 3.1},  Assumption \ref{assumption 3.1} $(\gamma)$ (see \eqref{eq3.1.0}) is satisfied for any $\gamma \in (0, 1)$.
Then, by the aforementioned theorem,  for sufficiently large $\lambda_0 = \lambda_0 (\nu, K, \varkappa, \theta, \Omega) \ge 1$ and any $\lambda \ge \lambda_0$,  one has $f_1 \in S_2 ((-\infty, T) \times  \bR^6)$, and
\begin{equation}
			\label{eq1.4.4}
\begin{aligned}
	& \|f_1\|_{  S_2 ((-\infty, T)  \times \bR^6)  }
	 \le  N \|h_1\|_{  L_2 ((-\infty, T) \times  \bR^6)  } \\
	&\le N (\|h\|_{  L_{2, \theta} (\Sigma^T) }   + \|f \|_{  L_{2, \theta-1} (\Sigma^T) } + \|\nabla_v f \|_{  L_{2, \theta-3} (\Sigma^T) }),
\end{aligned}
\end{equation}
where $N = N (\nu, K, \varkappa, \theta, \Omega) > 0$.

\textit{Boundary estimate.} As in the proof of Theorem \ref{theorem 7.3}, for $k \in \{2, \ldots, m\}$, the function $\overline{f_k}$ satisfies the identity
\begin{equation}
			\label{eq1.4.3}
\begin{aligned}
	&       \partial_t \overline{f_k} + w \cdot \nabla_y \overline{f_k} -  \nabla_w \cdot (\mathbb{A} \nabla_w \overline{f_k}) + \mathbb{B} \cdot \nabla_v \overline{f_k} + \lambda \overline{f_k}\\
  &     = \overline{h_k} - \nabla_w \cdot (\mathbb{X} \overline{f_k}), \quad \overline{f_k} (0, \cdot) \equiv 0,
\end{aligned}
\end{equation}
where
\begin{itemize}
\item[--]
$\mathbb{A}$ is defined by \eqref{eq1.4.3.2}  \eqref{eq1.4.3.3}, and \eqref{eq1.20} with $a$ replaced with $\sigma_G + \nu I_d$,
\item[--] $\mathbb{B}$ is defined by \eqref{eq1.4.3.2} and \eqref{eq1.4.3.4} with $b$ replaced with $-a_g$,
\item[--] $\mathbb{X}$ is given by \eqref{eq1.4.3.5} and \eqref{eq1.4.3.6}.
\end{itemize}
 This time, one needs to apply Theorem \ref{theorem D.1} to Eq. \eqref{eq1.4.3}.
Let us check its assumptions.
Note that by \eqref{eq11.2.1} -  \eqref{11.2.9}, $\mathbb{A}$ and $\mathbb{B}$ are bounded functions, and by the discussion in Appendix \ref{Appendix E},
and by Lemma \ref{lemma C.1},
$$
	\|\mathbb{A}\|_{  L_{\infty} ((0, T), C^{\varkappa/3, \varkappa}_{x, v} (\bR^6))  } \le N (K, \varkappa).
$$
Then, by Theorem \ref{theorem D.1} and \eqref{eq1.7.4},   we get
\begin{equation}
			\label{eq1.4.5}
\begin{aligned}
&	\|\overline{f_k}\|_{  S_2 ((-\infty, T)  \times \bR^6 )  } \le N  (\|\overline{h_k}\|_{  L_{2} ((-\infty, T)  \times \bR^6 )   }  +  \|\nabla_w \cdot (\mathbb{X} \overline{f_k})\|_{  L_{2} ((-\infty, T)  \times \bR^6 )   })\\
&
	\le N (\|h\|_{  L_{2, \theta - 2} (\Sigma^T)  } + \|f\|_{ L_{2, \theta -  1} (\Sigma^T) } + \|\nabla_v f\|_{ L_{2, \theta} (\Sigma^T) }),
\end{aligned}
\end{equation}
where $N = N (\nu, K, \theta, \Omega) > 0$.
Finally combining \eqref{eq1.4.4} with \eqref{eq1.4.5} and using Lemma \ref{lemma B.2}, we conclude that $f \in S_{2, \theta - 2} (\Sigma^T)$ and  that the estimate \eqref{eq4.2.1} holds.
As in the proof of Theorem \ref{theorem 7.3}, the $L_{7/3, \theta-2} (\Sigma^T)$ norm estimate of $f, \nabla_v f$ is obtained by combining the above estimates with the embedding theorem for $S_p$ spaces in \cite{PR_98}.

\textbf{Step 2: $S_p$ regularity}
To finish the proof, we repeat  word-for-word the argument  of Theorem \ref{theorem 1.9} and  Lemma \ref{lemma 3.7}.
\end{proof}

\begin{corollary}
			\label{corollary 4.3}
Any two finite energy solutions to Eq. \eqref{eq1.18} coincide.
\end{corollary}

\begin{proof}
 Since due to Proposition \ref{proposition 4.2},  any  finite energy weak solution to Eq. \eqref{eq1.18} must be of class $S_2 (\Sigma^T)$, we may use the energy identity  of Lemma \ref{lemma B.3}.
The rest of the argument is similar to that of   Corollary  \ref{corollary 7.3.1} (see page \pageref{proof of corollary 7.3.1}).
\end{proof}

\subsection{Unique solvability result for the  viscous  linearized Landau equation}

We now prove the unique solvability  for large $\lambda > 0$ for the  equation
\begin{equation}
			\label{eq5.2.0}
\begin{aligned}
	&Y f  - \nabla_v \cdot (\sigma_G \nabla_v f) - \nu \Delta_v f  -  a_g \cdot \nabla_v f + C f + \lambda f = h \,\, \text{in} \, \, \Sigma^T, \\
&
	 f (0, \cdot) = f_0 (\cdot) \,\, \text{in} \, \, \Omega \times \bR^3,\quad
 f_{-} (t, x, v) =  f_{+} (t, x, R_x v), \,   z \in \Sigma^T_{-},
\end{aligned}
\end{equation}
 where $C$ is some linear operator, for example $- \overline{K}_g$ defined in \eqref{eq11.2}.

To implement a perturbation argument, we will work with the following weighted kinetic Sobolev spaces.
\begin{definition}
			\label{definition 10.1}
Let $T > 0$, $\theta \ge 16$, $p > 14$,  $\lambda > 0$ be numbers.
 We say that $f \in \cK_{\theta, \lambda, p} (\Sigma^T)$
if the following hold:
 \begin{enumerate}
     \item
$f \in S_{2, \theta-2} (\Sigma^T) \cap S_{p, \theta - 16} (\Sigma^T)  \cap L_{2, \theta} (\Sigma^T) \cap L_{\infty} (\Sigma^T)$,\\

\item $\nabla_v f \in L_{2, \theta} (\Sigma^T)$,\\
\item
$f (0, \cdot) \in L_{2, \theta} (\Omega \times \bR^3) \cap L_{\infty} (\Omega \times \bR^3) \cap \cO_{2, \theta - 2} \cap \cO_{p, \theta - 16}$,
$
	f (T, \cdot) \in   L_{2, \theta} (\Omega \times \bR^3),
$
 $f_{\pm} \in L_{\infty} (\Sigma^T_{\pm}, |v \cdot n_x|)$, \\

 \item $f_{-} (t, x, v) =  f_{+} (t, x, R_x v)$  a.e. on $\Sigma^T_{-}$.\\

\end{enumerate}

The $\cK_{\theta, \lambda, p} (\Sigma^T)$-norm is defined as
\begin{equation}
			\label{eq10.0}
\begin{aligned}
&	\|f\|_{\cK_{\theta, \lambda, p} (\Sigma^T) } =    \lambda   \|f\|_{L_{2, \theta} (\Sigma^T)}   +  \lambda \|f\|_{L_{\infty} (\Sigma^T)} \\
&
	  + \|\nabla_v f\|_{L_{2, \theta} (\Sigma^T)} +  \|f\|_{  S_{2, \theta - 2} (\Sigma^T) } + \|f\|_{  S_{p, \theta - 16} (\Sigma^T) }\\
&
	+   \|f_{\pm}\|_{L_{\infty} (\Sigma^T_{\pm}, |v \cdot n_x|)}  +  \|f (T, \cdot)\|_{L_{2, \theta} (\Omega \times \bR^3)}\\
&
       + \|f (0, \cdot)\|_{  L_{2, \theta} (\Omega \times \bR^3)  } + \|f (0, \cdot)\|_{  L_{\infty } (\Omega \times \bR^3)  }+  \|f (0, \cdot)\|_{ \cO_{2, \theta - 2}  } +   \|f (0, \cdot)\|_{ \cO_{p, \theta - 16}  }.
 \end{aligned}
\end{equation}
\end{definition}

\begin{assumption}
			\label{assumption 10.1}
There exists a  linear operator $C$ on
$\cK_{\theta, \lambda, p}(\Sigma^T)$
and a constant $\kappa  > 0$
such for any $u \in  \cK_{\theta, \lambda, p} (\Sigma^T)$,
\begin{equation}		
			\label{eq10.1.0}
	\|C  u\|_{ L_{2, \theta} (\Sigma^T)  } + \|C  u\|_{   L_{\infty}  (\Sigma^T) }
	\le \kappa (\|u\|_{ L_{2, \theta} (\Sigma^T)  } +  \|u\|_{ L_{\infty}  (\Sigma^T) }).
\end{equation}
\end{assumption}

\begin{remark}
			\label{remark 10.1}
An example of an operator $C$ that satisfies \eqref{eq10.1.0} is the operator $\overline{K}_g$ given by \eqref{eq11.2}.
If Assumption \ref{assumption 11.1} (see \eqref{con11.1} - \eqref{con11.1'}) holds and $\|g\|_{ L_{\infty} (\Sigma^T) }$ is sufficiently small,
then, by Lemmas 2.9 and formula $(7.5)$ of \cite{KGH},
for any  $u \in L_{2, \theta} (\Sigma^T) \cap L_{\infty} (\Sigma^T)$ such that $\nabla_v u \in L_{2} (\Sigma^T)$, one has
\begin{align*}
	&\|\overline{ K}_g u\|_{ L_{2, \theta} (\Sigma^T)  } \le N (\theta)  (\|u\|_{ L_{2, \theta} (\Sigma^T)  } +  \|u\|_{ L_{\infty} (\Sigma^T)  }),\\
	&   \|\overline{ K}_g u\|_{ L_{\infty} (\Sigma^T)  }  \le N (\theta) \|u\|_{ L_{\infty} (\Sigma^T)  }.
\end{align*}
\end{remark}

\begin{proposition}
			\label{lemma 10.2}
Let
\begin{itemize}
\item[--] $\Omega$ be a bounded $C^3$ domain,
\item[--]  $T > 0$, $\nu \in (0, 1]$, $\varkappa \in (0, 1]$, $\theta \ge 16$  be numbers,
\item[--] $f_0 \in  L_{2, \theta} (\Omega \times \bR^3) \cap L_{\infty} (\Omega \times \bR^3)  \cap \cO_{2, \theta - 2 }  \cap \cO_{p, \theta - 4  }$, $h \in L_{2, \theta} (\Sigma^T) \cap L_{\infty} (\Sigma^T)$,
\item[--] Assumption \ref{assumption 11.1} (see \eqref{con11.1} - \eqref{con11.1'}) be satisfied,
\item [--] $\|g\|_{  L_{\infty} ( \sigma^T ) } \le \varepsilon$.
\end{itemize}
Then,  for sufficiently small $\varepsilon > 0$ (independent of $\Omega, T, \nu, \varkappa, \theta$,   there exists some $\lambda_0  = \lambda_0 (\theta,  K, \varkappa, \Omega, \nu) > 0$  such that
Eq. \eqref{eq5.2.0}
has a unique finite energy  strong solution in the sense of Definition \ref{definition 7.0}, and
\begin{equation}
			\label{eq10.2.1}
	\|f\|_{   \cK_{\theta, \lambda, p} (\Sigma^T) } \le N (\|h\|_{ L_{2, \theta} (\Sigma^T) } + \|h\|_{ L_{\infty} (\Sigma^T) }  + |f_0|_{\cO_{2, \theta-2} } + |f_0|_{  \cO_{p, \theta-4} }   ),
\end{equation}
where $N = N (\theta, K, \Omega, p, \varkappa, \nu) > 0$.
\end{proposition}

\begin{proof}
Let $\varepsilon > 0$ be a number such that \eqref{eq11.2.1} holds.
For an arbitrary function $\xi \in  \cK_{\theta, \lambda, p} (\Sigma^T)$, let us consider the equation
\begin{align}
			\label{10.3}
	&Y u  - \nabla_v \cdot (\sigma_G \nabla_v u) - \nu \Delta_v u   - a_g \cdot \nabla_v u  + \lambda u = h  - C \xi \,\, \text{in} \, \, \Sigma^T,    \\
&	u (0, \cdot) = f_0 (\cdot) \,\, \text{in} \, \, \Omega \times \bR^3,
\quad  u_{-} (t, x, v) =  u_{+} (t, x, R_x v), \,   z \in \Sigma^T_{-}\notag.
\end{align}
Note that by Assumption \ref{assumption 10.1} (see \eqref{eq10.1.0}),
the right-hand side of Eq. \eqref{10.3} is of class $L_{2, \theta}  (\Sigma^T) \cap L_{\infty} (\Sigma^T)$.
Then, by this,  Propositions \ref{proposition 4.1} - \ref{proposition 4.2} and Corollary  \ref{corollary 4.3},
there exists $\lambda_0  = \lambda_0 (\theta, K, \Omega, \varkappa, \nu) \ge 1$,
 such that
for any $\lambda \ge \lambda_0$, Eq. \eqref{10.3},
has a unique finite energy strong solution $u$ in the sense of Definition \ref{definition 7.0}, and, in addition,
\begin{equation}
			\label{10.4}
\begin{aligned}
	&  \|u (T, \cdot)\|_{L_{2, \theta} (\Omega \times \bR^3)} + \lambda^{1/2} \|u\|_{L_{2, \theta} (\Sigma^T)}    + \|\nabla_v u\|_{L_{2, \theta} (\Sigma^T)} \\
	&\le N  \big(\|f_0\|_{L_{2, \theta} (\Omega \times \bR^3)} +  \lambda^{-1/2}  \|h\|_{L_{2, \theta} (\Sigma^T)}  +  \lambda^{-1/2} \|\xi\|_{L_{2, \theta} (\Sigma^T)} + \lambda^{-1/2} \|\xi\|_{L_{\infty} (\Sigma^T)}\big),\\
	&  \max\{\|u (T, \cdot)\|_{L_{\infty} (\Omega \times \bR^3)},    \|u\|_{L_{\infty} (\Sigma^T)}, \|u_{\pm}\|_{L_{\infty} (\Sigma^T_{\pm}, |v \cdot n_x|)}\}\\
	&\le N  \big(\|f_0\|_{L_{\infty} (\Sigma^T)}  +  \lambda^{-1} \|h\|_{L_{\infty} (\Sigma^T)}
	+ \lambda^{-1}  \|\xi\|_{L_{2, \theta} (\Sigma^T)} + \lambda^{-1} \|\xi\|_{L_{\infty} (\Sigma^T)}\big), \\
	& \|u\|_{ S_{2, \theta - 2} (\Sigma^T) } + \|u\|_{  S_{p, \theta-16} (\Sigma^T) } \le N \big(\||u| + |\nabla_v u|\|_{ L_{2, \theta} (\Sigma^T)} \\
	& +  |f_0|_{  \cO_{2, \theta - 2} } +  |f_0|_{  \cO_{p, \theta - 4} }  +  \|h\|_{L_{2, \theta - 2} (\Sigma^T)} +   \|h\|_{L_{p, \theta - 4} (\Sigma^T)}\\
	&   + \|\xi\|_{L_{2, \theta} (\Sigma^T)} + \|\xi\|_{L_{\infty} (\Sigma^T)}),
\end{aligned}
\end{equation}
where $N = N (\theta, p,   K, \Omega, \kappa, \varkappa,  \nu) > 0$.
Furthermore, for any $\xi \in  \cK_{\theta, \lambda, p} (\Sigma^T)$,
we denote $R_{\lambda} \xi = u$, where $u$ is the finite energy strong solution to \eqref{10.3}.
Thus, by \eqref{10.4}, $R_{\lambda}$ is a bounded linear operator on $\cK_{\theta, \lambda, p} (\Sigma^T)$.

Next, note that, since a solution to \eqref{eq5.2.0} is  a fixed point of $R_{\lambda}$, it suffices to show that
$R_{\lambda}$ is a contraction on $\cK_{\theta, \lambda, p} (\Sigma^T)$.
For any $\xi_1, \xi_2 \in \cK_{\theta, \lambda, p} (\Sigma^T)$, by   \eqref{eq10.0}, \eqref{10.4},  and the fact that $\lambda_0 \ge 1$, we get
$$
	\|R_{\lambda} \xi_1 - R_{\lambda} \xi_2\|_{\cK_{\theta, \lambda, p}(\Sigma^T)} \le N \big(\|\xi_1 - \xi_2 \|_{L_{\infty}  (\Sigma^T)}
	 +    \|\xi_1 - \xi_2\|_{L_{2, \theta} (\Sigma^T)}\big),
$$
where $N$ is independent of $\lambda$.
Furthermore,  by  the definition of the $\cK_{\theta, \lambda, p} (\Sigma^T)$ norm, for $\lambda \ge \lambda_0$,
$$
		\|\xi_1 - \xi_2  \|_{L_{2, \theta} (\Sigma^T)}  + \|\xi_1 - \xi_2 \|_{L_{\infty}  (\Sigma^T)} \le   \lambda^{-1}_0  \|\xi_1  - \xi_2\|_{\cK_{\theta, \lambda, p} (\Sigma^T) }.
 $$
Thus, for $\lambda_0 > N +1$, $R_{\lambda}$ is a contraction mapping.
Thus, Eq. \eqref{eq5.2.0} has a unique solution $f$ of class $\cK_{\theta, \lambda, p} (\Sigma^T)$.
To prove that $f$ satisfies the weak formulation  \eqref{7.0.0} for any $\phi \in C^1_0 (\overline{\Sigma^T})$, we define a Picard iteration sequence $f^{(0)} \equiv f_0$, $f^{(n+1)} = R_{\lambda} f^{(n)}$,
which converges to $f$ in $\cK_{\theta, \lambda, p} (\Sigma^T)$.
We pass to the limit in the weak formulation for $f^{(n)}$ and use \eqref{eq10.1.0} in Assumption \ref{assumption 10.1}.
 \end{proof}

\subsection{Proof of Theorem \ref{theorem 11.2}}
To prove the existence part, we work with  the viscous linearized Landau equation
\begin{equation}
		\label{11.2}
\begin{aligned}
	&Y  f +  \sfL  f - \nu \Delta_v  f +  \mathsf{\Gamma} [g,  f]  +  \lambda f = h,\\
	&  f (0, \cdot, \cdot) = f_0,  \quad  f_{-} (t, x, v) = f_{+} (t, x, R_x v), z \in \Sigma^T_{-}.
\end{aligned}
\end{equation}
 An equivalent form of this equation is
\begin{equation}
		\label{11.2div}
\begin{aligned}
	&Y f  - \nabla_v \cdot (\sigma_G \nabla_v f) - \nu \Delta_v f  -  a_g \cdot \nabla_v f -\overline{K}_g f + \lambda f = h \,\, \text{in} \, \, \Sigma^T, \\
&
	 f (0, \cdot) = f_0 (\cdot) \,\, \text{in} \, \, \Omega \times \bR^3,\quad
 f_{-} (t, x, v) =  f_{+} (t, x, R_x v), \,   z \in \Sigma^T_{-}.
\end{aligned}
\end{equation}

Thanks to Proposition \ref{lemma 10.2}, Eq.  \eqref{11.2div} has a unique strong solution (in the sense of Definition \ref{definition 11.0}) of class $\cK_{\theta, \lambda, p} (\Sigma^T)$.
However, the a priori bound in \eqref{eq10.2.1} has  a constant depending on $\nu$, and, hence, such estimate cannot be used in the method of vanishing viscosity. To establish a priori estimates that are uniform in $\nu$,
we use the following ingredients: the weighted $L_2$ bound (see Lemma \ref{lemma 4.6}), and the $S_p$ estimate in Proposition \ref{proposition 4.7}.
Once we prove the former, we bootstrap by using the latter.
We point out that in contrast to the kinetic Fokker-Planck equation,
we do not derive a priori bounds of $\|f_{\pm}\|_{ L_{\infty} (\Sigma^T_{\pm}, |v \cdot n_x|)}$ by applying the $L_p$ energy identity for the operator $Y$ (c.f. Lemma \ref{lemma 7.5}).
The present authors are not aware of such estimates for Eq. \eqref{11.2}. Instead, we use the $S_p$ estimates to  bound  $\|f_{\pm}\|_{  L_{\infty} (\Sigma^T_{\pm}, |v \cdot n_x|) }$.

\begin{lemma}
			\label{lemma 4.6}
Let $\nu \in [0, 1], \lambda,  \theta \ge 0$, $T > 0$ be numbers,
 $f$ be a finite energy strong solution to  Eq.  \eqref{11.2}  in the sense of Definition \ref{definition 11.0}.
Assume that $h \in L_{2, \theta} (\Sigma^T)$
and $f_0 \in L_{2, \theta} (\Omega \times \bR^3)$.
Then, there exists a constant $\varepsilon  = \varepsilon (\theta) > 0$
 such that if
$$
	\|g\|_{  L_{\infty} ( \Sigma^T ) } \le \varepsilon,
$$
then, one has
\begin{align*}
& \|f (T, \cdot)\|_{ L_{2, \theta} (\Omega \times \bR^3) } +    \| f\|_{ \sigma, \theta  }
	  + (\lambda^{1/2}  + 1) \|f\|_{ L_{2, \theta } (\Sigma^T)  }\\
	& \le   N  (\|h\|_{  L_{2, \theta} (\Sigma^T) } + \|f_0\|_{ L_{2, \theta} (\Omega \times \bR^3) }),
\end{align*}
where $\|\cdot\|_{\sigma, \theta}$ is defined in \eqref{eq1} and $N = N (\theta, T) > 0$.
\end{lemma}

\begin{proof}
We follow the argument of Lemma 8.2 in \cite{KGH}.

Let $\lambda' > 0$ be a number which we will determine later.
Multiplying Eq. \eqref{11.2} by $e^{-2\lambda' t}$,  and using the energy identity   in Lemma \ref{lemma B.3} with $\jb^{\theta}$ give
\begin{equation}
			\label{eq4.6.1}
\begin{aligned}
&	 \|f (T, \cdot)  e^{-\lambda' T} \|^2_{ L_{2, \theta} (\Omega \times \bR^3) }
	+ 2\int_{\Sigma^T} (\sfL f) f  \jb^{\theta} e^{-2\lambda' t} \, dz \\
&
	+ 2 \int_{\Sigma^T} \mathsf{\Gamma} [g, f]  f \jb^{\theta} e^{-2\lambda' t} \, dz + 2 (\lambda + \lambda') \|f e^{-\lambda' t}\|^2_{  L_{2, \theta} (\Sigma^T)  }\\
	&\le  2  \int_{\Sigma^T} f h \jb^{\theta} e^{-2\lambda' t} \, dz +   \|f_0\|^2_{ L_{2, \theta} (\Omega \times \bR^3) }.
 \end{aligned}
\end{equation}
By Lemmas 2.7 and 2.8 of \cite{KGH} and \eqref{eq2.17},  there exists $N_0, N_1 > 0$, depending only on $\theta$, such that
 \begin{align*}
 &
	   \int_{\Sigma^T} (\sfL f) f  \jb^{\theta} e^{-2\lambda' t} \, dz \ge (1/2) \|f  e^{-\lambda' t}\|^2_{\sigma, \theta} - N_0 \|f e^{-\lambda' t}\|^2_{  L_{2, \theta} (\Sigma^T) },\\
 &
	\int_{\Sigma^T} \mathsf{\Gamma} [g, f]  f\jb^{\theta} e^{-2\lambda' t} \, dz \le   N_1  \|g\|_{ L_{\infty} (\bR^7_T) }  \|f e^{-\lambda' t}\|^2_{\sigma, \theta} \le N_1 \varepsilon  \|f e^{-\lambda' t}\|^2_{\sigma, \theta}.
\end{align*}
Combining this with \eqref{eq4.6.1} and the Cauchy-Schwartz inequality,  we obtain
\begin{align*}
	 \|f (T, \cdot) e^{-\lambda' T}\|^2_{ L_{2, \theta} (\Omega \times \bR^3) }  &+  (1 - 2 N_1 \varepsilon) \|f e^{-\lambda' t}\|^2_{\sigma, \theta}
	+    (2\lambda  + \lambda' -  N_0 (\theta))  \|f e^{-\lambda' t}\|^2_{ L_{2, \theta} (\Sigma^T) }\\
&
 \le
	   \|f_0\|^2_{ L_{2, \theta} (\Omega \times \bR^3) } +  (\lambda')^{-1} \|h e^{-\lambda' t}\|^2_{  L_{2, \theta} (\Sigma^T) }.
\end{align*}
Taking $\varepsilon < (4N_1)^{-1}$  and $\lambda' > 2 N_0$, we prove the desired estimate.
\end{proof}

\begin{remark}
Note that by Remark \ref{remark 2.17}, one can extract the estimate of $\|f\|_{ L_{2, \theta - 3} (\Sigma^T)  }$ from the above result.
\end{remark}

We will use the next proposition with $\theta$ replaced with $\theta - 3$.
\begin{proposition}[$S_p$ bound]
				\label{proposition 4.7}
Let
\begin{itemize}
\item[--] $\Omega$ be a bounded $C^3$ domain,
\item[--] $\nu \in [0, 1]$, $\lambda \ge 0$,  $\varkappa \in (0, 1]$, $p > 14$ be numbers,
\item[--] Assumption \ref{assumption 11.1} (see \eqref{con11.1} - \eqref{con11.1'})  hold,
\item[--] $\|g\|_{ L_{\infty} ( \Sigma^T )} \le \varepsilon$,
\item[--]   $$
			\|f\|_{ L_{2, \theta} (\Sigma^T) } + \|\nabla_v f\|_{ L_{2, \theta} (\Sigma^T)   } \le M
		$$
		for some $M > 0$,
\item[--]  $f$ is a finite energy strong solution to Eq.  \eqref{11.2div}  (see Definition \ref{definition 11.0}) with $f_0 \equiv 0$ and $h \in L_{2, \theta} (\Sigma^T) \cap L_{p, \theta} (\Sigma^T)$,
\item[--]  $f \in \cK_{\lambda, \theta, p} (\Sigma^T)$ (see Definition \ref{definition 10.1}).
\end{itemize}
Then, there exist numbers $\varepsilon \in (0, 1)$, $\theta =  \theta (\varkappa, p) > 1$, and $\theta' = \theta' (\varkappa, p), \theta''  = \theta'' (\varkappa, p) \in (1, \theta)$  such that
\begin{equation}
			\label{eq4.7.0}
\begin{aligned}
&\|f\|_{ S_{2, \theta'  } (\Sigma^T) } +	\|f\|_{ S_{p,  \theta'' } (\Sigma^T) }  + \|f_{\pm}\|_{  L_{\infty} (\Sigma^T_{\pm}, |v \cdot n_x|)  }  \\
& + \|f\|_{ L_{\infty} ((0, T), C^{\alpha/3, \alpha}_{x, v} (\overline{\Omega} \times \bR^3))  } +  \|\nabla_v f\|_{ L_{\infty} ((0, T), C^{\alpha/3, \alpha}_{x, v} (\overline{\Omega} \times \bR^3))  }
	+ \|f\|_{  L_{\infty} (\Sigma^T_{\pm}, |v \cdot n_x|) }  \\
&
	\le N (M  + \|h\|_{ L_{2, \theta} (\Sigma^T) } +  \|h\|_{ L_{p, \theta} (\Sigma^T) }),
\end{aligned}
\end{equation}
where $N = N (\theta, p, \varkappa, K, \Omega) > 0$, and $\alpha = 1- 14/p$.
\end{proposition}

To prove this estimate, we need a technical result, which is similar to Lemma \ref{lemma 3.7}. We will state the lemma  after we introduce some notation.

Let  $\zeta_0 \in C^{\infty}_0 (B_2)$,  $\zeta  \in C^{\infty}_0 (\{ 2^{-1} <  |v| < 2^{3/2} \})$ be radially symmetric functions such that
    \begin{itemize}[--]
    \item
    $\zeta_0   = 1$ on $\overline{B}_{2^{1/2}}$,
    and   $0 \le \zeta_0 < 1$ on the complement of this set,

    \item
        $\zeta  = 1$ if $\{ 2^{-1/2} \le |v| \le 2\}$,
    and $0 \le \zeta < 1$ otherwise.
\end{itemize}
 For $n  \in \{ 1, 2, \ldots\}$ and a bounded measurable function $b = (b_1, \ldots, b_d)$, we set
\begin{align}
	\label{eq4.7.16}
	&\zeta_n (v) = \zeta (v 2^{-n}),\\
	&\label{eq4.7.7}
	\sigma_{G, n} = (\sigma_G  + \nu I_3) \zeta_n + (1 -  \zeta_n) I_3,\\
	&\cL_n = Y  - \nabla_v \cdot (\sigma_{G, n} \nabla_v) + b \cdot \nabla_v\notag.
\end{align}

\begin{lemma}
			\label{lemma 4.8}
Let
\begin{itemize}
\item[--] $\Omega$ be a bounded $C^3$ domain,
\item[--]  $T > 0$, $\nu \in [0, 1]$, $\theta \ge 2$,  $\varkappa \in (0, 1]$, $p \ge 2$, $\lambda \ge 0$ be numbers,
\item[--] Assumption  \ref{assumption 11.1} (see \eqref{con11.1} - \eqref{con11.1'}) be satisfied,
\item[--] $\|b\|_{ L_{\infty} (\Sigma^T) } \le K$,
\item[--] $\|g\|_{  L_{\infty} (\Sigma^T  ) } \le \varepsilon$ for some $\varepsilon > 0$,
\item[--]$f \in S_{p,  \widetilde \theta (\Sigma^T)} \cap L_{p, \theta} (\Sigma^T)$ for some $\widetilde \theta \ge \theta -2$, $\nabla_v f \in L_{p, \theta} (\Sigma^T)$, $f_{\pm} \in L_{\infty} (\Sigma^T_{\pm}, |v \cdot n_x|)$,
\item[--] $f$ satisfies the equation
\begin{align*}
&	\cL_n  f + \lambda f = h \, \,   \text{in} \, \Sigma^T,\\
 & f (0, \cdot) \equiv 0 \,\, \text{in} \, \, \Omega \times \bR^3
\end{align*}
in the weak    sense (see  Definition \ref{definition 7.0})
with $h \in L_{p, \theta - 2} (\Sigma^T)$
 and the specular reflection boundary condition
$$
   f_{-} (t, x, v)=  f_{+} (t, x, R_x v).
$$
 almost everywhere.
\end{itemize}
Then,   there exists  a number $\varepsilon \in (0, 1)$ (independent of $\Omega, T, \nu, \varkappa, \theta$)
and $\beta = \beta (p, \varkappa) > 0$ such that
\begin{equation}
			\label{eq4.8.0}
     \|f\|_{  S_{p, \theta  - 2} (\Sigma^T)  } \le N 2^{ \beta n} (\|h\|_{ L_{p, \theta - 2} (\Sigma^T) } + \|f\|_{ L_{p, \theta - 1} (\Sigma^T) } + \|\nabla_v f\|_{  L_{p, \theta} (\Sigma^T) }),
\end{equation}
where $N = N (K, \varkappa, p, \Omega, \theta)$.

Furthermore,
\begin{itemize}
\item if $p < 14$,  the norm $\||f| +|\nabla_v f|\|_{ L_{r, \theta-2} (\Sigma^T) }$ is bounded by the right-hand side of  \eqref{eq4.8.0},
where $r$ is given by \eqref{7.3.0.0}.

\item if $p > 14$, then, for $\alpha = 1 - 14/p$, the norms
$$
	\|f\|_{ L_{\infty} ((0, T), C^{\alpha/3, \alpha}_{x, v} (\overline{\Omega} \times \bR^3)) }, \|\nabla_v f\|_{ L_{\infty} ((0, T), C^{\alpha/3, \alpha}_{x, v} (\overline{\Omega} \times \bR^3))  },
	 \|f_{\pm}\|_{ L_{\infty} (\Sigma^T_{\pm}, \jb^{\theta-2}) }
$$
are bounded by the right-hand side of \eqref{eq4.8.0}.
\end{itemize}

\end{lemma}

\begin{proof}
We follow the argument of Lemma \ref{lemma 3.7}.
Let  $\varepsilon > 0$ be a number such that \eqref{eq11.2.1} is true and
 $f_k, k = 1, \ldots, m$ be the functions defined by \eqref{eq1.7.10}.

\textbf{Interior estimate.}
We prove the lemma by applying the a priori estimate in Theorem \ref{corollary 3.4} to Eq. \eqref{eq4.1.1}.
 Let us check its assumptions.

\textit{Assumption \ref{assumption 2.1} (see \eqref{eq2.1.0}).}
Due to \eqref{eq11.2.1} and \eqref{eq4.7.7},  for sufficiently small $\varepsilon > 0$,
$$
	N_0 2^{-3n} \le	\sigma_{G, n} \le N_0^{-1}  I_3
$$
 for some constant $N_0 > 0$ independent of $n$.

\textit{Assumption \ref{assumption 3.2} (see \eqref{eq3.2.0}).}
By \eqref{eq11.2.1} - \eqref{11.2.9} and \eqref{eq4.7.7}, one has
$$
	\|\nabla_v \sigma_{G, n}\|_{  L_{\infty} (\Sigma^T)  } \le N (K),
$$
and hence, \eqref{eq3.2.0} holds with $b^i = - \partial_{v_j} \sigma^{j i}_{G, n}, i = 1, 2, 3,$ and $c = 0$.

\textit{Assumption \ref{assumption 3.1} (see \eqref{eq3.1.0}).}
By Assumption \ref{assumption 11.1} (see \eqref{con11.1} - \eqref{con11.1'}) and Lemma \ref{lemma C.1}, we have
$$
	\|\sigma_{G, n}\|_{  L_{\infty} ((0, T), C^{\varkappa/3, \varkappa}_{x, v} (\bR^6))} \le N (K, \varkappa).
$$
Then, by Remark \ref{remark 3.1}, for any  $r \in (0, 1)$ and $Q_r (z_0)$,
$$
	\text{osc}_{x, v} (\sigma_{G, n}, Q_r (z_0)) \le N (K, \varkappa) r^{\varkappa},
$$
where $\text{osc}_{x, v} (\sigma_{G, n}, Q_r (z_0))$ is defined by \eqref{eq3.1}.
Hence,  for any $\gamma \in (0, 1)$, \eqref{eq3.1.0} in Assumption  \ref{assumption 3.1} $(\gamma)$ holds with
\begin{equation}
			\label{eq4.8.2}
	R_0  = N (K,  \varkappa) \gamma^{1/\varkappa}.
\end{equation}
Furthermore, let
$$
	\beta = \beta (p) > 0, \quad \kappa = \kappa (p) > 0, \quad \gamma_{   \star  } = \delta^{\kappa}   \widetilde \gamma_{\star} (p) > 0
$$
be the numbers in Theorem \ref{corollary 3.4} with $\delta  = 2^{-3n}$.
Then,  by  the above, \eqref{eq3.1.0} in Assumption \ref{assumption 3.1} $(\gamma_{ \star }   )$
holds with
$$
	R_0    = N_1 (K, \varkappa, p) 2^{-3 \kappa n/\varkappa}.
$$

Next,  by Theorem \ref{corollary 3.4} and Eq. \eqref{eq4.1.1},
\begin{equation}
			\label{eq4.8.1}
\begin{aligned}
	\|f_1\|_{ S_p (\bR^7_T)) } & \le N  2^{ 3\beta n} \|h_1\|_{ L_p (\bR^7_T) } + N 2^{ 6 \kappa n/\varkappa } \|f_1\|_{ L_p (\bR^7_T) }     \\
	&\le N 2^{ 3 \beta n + 6 \kappa n/\varkappa} (\|h\|_{ L_{p, \theta-2} (\bR^7_T) } + \|f\|_{ L_{p, \theta - 1} (\bR^7_T) } + \|\nabla_v f\|_{   L_{p, \theta} (\bR^7_T) }),
\end{aligned}
\end{equation}
  where $N = N (p, K,  \varkappa, \theta, \Omega) > 0$.
Next, recall the notation $B^p$ (see \eqref{eq7.2.6}).
By the embedding theorem for the $S_p$ space (see \cite{PR_98}), $\|f_1\|_{ B^p}$ is bounded above by the right-hand side of \eqref{eq4.8.1}.

\textbf{Boundary estimate.}
Recall that $\overline{f_k}, k  = 2, \ldots, n$, defined as the mirror extension of $f_k$, satisfies Eq. \eqref{eq1.4.3}
where
\begin{itemize}
\item[--]
$\mathbb{A}$ is defined by \eqref{eq1.4.3.2}, \eqref{eq1.4.3.3}, and \eqref{eq1.20} with $a$ replaced with $\sigma_{G, n}$,  where the latter is given by \eqref{eq4.7.7},
\item[--] $\mathbb{B}$ is defined by \eqref{eq1.4.3.2} and \eqref{eq1.4.3.4} with $b$,
\item[--] $\mathbb{X}$ is given by \eqref{eq1.4.3.5} and \eqref{eq1.4.3.6}.
\end{itemize}
By the conclusion in Appendix \ref{Appendix E} (see Lemma \ref{lemma E.1}), since $\zeta_n$ is a radially symmetric cutoff function, we have
$$
	\mathbb{A} \in L_{\infty} ((0, T), C^{\varkappa/3, \varkappa}_{x, v} (\bR^6)),
$$
and, furthermore,   by Lemma \ref{lemma C.1},
\begin{equation}
							\label{eq4.8.5}
	\|\mathbb{A}\|_{ L_{\infty} ((0, T), C^{\varkappa/3, \varkappa}_{x, v} (\bR^6)) } \le N (K, \varkappa)
\end{equation}
because $|\zeta_n| + |\nabla_v \zeta_n| \le N$ with $N$ independent of $n$.
Hence, as above, for any $\gamma_{ \star } > 0$, the function $\mathbb{A}$ satisfies \eqref{eq3.1.0} in Assumption \ref{assumption 3.1} $(\gamma_{  \star })$
with $R_0$ given by \eqref{eq4.8.2}.
Next, by Theorem \ref{corollary 3.4} and Eq. \eqref{eq1.4.3} combined with the estimates of $\mathbb{X}, \nabla_w \mathbb{X}$ in \eqref{eq1.7.4}, we obtain
\begin{equation}
			\label{eq4.8.3}
\begin{aligned}
	&\|\overline{f_k}\|_{ S_p (\bR^7_T)  } \le N  2^{3 \beta n}  (\||\overline{h_k}| + |\nabla_w \cdot (\mathbb{X} \overline{f_k})\|_{  L_p (\bR^7_T)  })  + N 2^{ 6 \kappa n/\varkappa} \|\overline{f_k}\|_{ L_p (\bR^7_T)  } \\
	&\le N 2^{3 \beta n + 6 \kappa n/\varkappa}  (\|h\|_{  L_{p, \theta-2} (\bR^7_T)  } + \|f\|_{  L_{p, \theta-1} (\bR^7_T)  } + \|\nabla_v f \|_{  L_{p, \theta} (\Sigma^T)  }),
\end{aligned}
\end{equation}
where $N = N (p, K, \varkappa, \theta, \Omega)$.
Combining the above inequality with \eqref{eq4.8.1}, we prove the desired estimate of $\|f\|_{  S_{p, \theta - 2} (\Sigma^T) }$.
Again, by using the  embedding theorem for the $S_p$ spaces, we bound
the norms of $\|\overline{f_k}\|_{  B^p }, k \ge 2$, where $B^p$ is defined by \eqref{eq7.2.6}. This and the bound of $\|f_1\|_{ B^p }$
yield the estimates of
$$
	\|f_{k}\|_{  L_{\infty} ((0, T), C^{\alpha/3, \alpha}_{x, v} (\overline{\Omega} \times \bR^3)) }, \|\nabla_v f_{k}\|_{  L_{\infty} ((0, T), C^{\alpha/3, \alpha}_{x, v} (\overline{\Omega} \times \bR^3)) }
$$
with $\alpha =1 - 14/p$, when $p > 14$.
Since $f_k = f \eta_j \jb^{\theta-2}$, we also obtain the bound of $ \|f_{\pm}\|_{ L_{\infty} (\Sigma^T_{\pm}, \jb^{\theta-2}) }$.
The lemma is proved.
\end{proof}

\begin{proof}[Proof of  Proposition \ref{proposition 4.7}]

We  follow the argument of Lemma \ref{lemma 3.7} with minor modifications.
The central part of the argument is the following assertion.

\textbf{Claim.}
Let $\varepsilon > 0$ be a sufficiently small number such that  \eqref{eq11.2.1} is satisfied. Let 
 $r  \in [2, \infty)\setminus \{14\}$ and  $\beta = \beta (r, \varkappa)$ be the number in the statement of Lemma \ref{lemma 4.8},
and $\theta >   2 + \beta$.
Assume that  $f \in \cK_{\theta,\lambda, p} (\Sigma^T)$ is a function that satisfies  Eq. \eqref{11.2div} with $f_0 \equiv 0$ weakly and in the almost everywhere sense (see conditions $(3)$ and $(4)$ of Definition \ref{definition 11.0}),
and, in addition, one has
 $$
	\||f| +  |\nabla_v f|\|_{  L_{r, \theta_r} (\Sigma^T) } \le M_r
 $$
for some  $2 + \beta  < \theta_r  \le \theta$ and  $M_r > 0$.
Then,  the following assertions hold:
 \begin{enumerate}
 \item  \label{4.7.i}
   \begin{equation}
			\label{eq4.7.10}
  \|f\|_{ S_{r,  \theta_r - 2 - \beta}  (\Sigma^T) }
    \le N   \big(\|h\|_{ L_{r, \theta_r} (\Sigma^T)   }
 +  M_r\big),
\end{equation}
where
$
	N = N (r, K, \theta,  \Omega, \varkappa) > 0.
$
\\
     \item  \label{4.7.ii} If $r \in [2, 14)$,
	$$
		\||f| + |\nabla_v f|\|_{  L_{r',   \theta_r - 2 - \beta } (\Sigma^T) }
	$$
	is bounded by the right-hand side of \eqref{eq4.7.10}
	where $r'$ is determined by the relation
	$$
		        \frac{1}{r'} = \frac{1}{r} - \frac{1}{14}.
	$$

	\item \label{4.7.iii}
	If $r > 14$, then for $\alpha = 1 - 14/r$, the norms
	$$
		\|f\|_{ L_{\infty} ((0, T), C^{\alpha/3, \alpha}_{x, v} (\overline{\Omega} \times \bR^3))  }, \|\nabla_v f\|_{ L_{\infty} ((0, T), C^{\alpha/3, \alpha}_{x, v} (\overline{\Omega} \times \bR^3))  }, \|f_{\pm}\|_{ L_{\infty} (\Sigma^T_{\pm}, \jb^{  \theta_r - 2 - \beta}) }
	$$
	are bounded by the  right-hand side of \eqref{eq4.7.10}.
  \end{enumerate}
If this claim is valid, then, repeating the argument of Step 2 in the proof of Theorem \ref{theorem 1.9} (see p. \pageref{8.1.7}), we prove the desired estimate \eqref{eq4.7.0}. In particular, we use  the same powers $r_k, k = 1, \ldots, 8$ defined in \eqref{8.1.7}, and similar  weight parameters
$$
	\theta_1 = \theta, \, \, \theta_{k+1} = \theta_k - 2 -  \beta, \, \, k = 1, \ldots, 7.
$$

\textbf{Proof of the claim.} Let  $\zeta_n, n \in \{0, 1, 2, \ldots\}$ be the functions defined above Lemma \ref{lemma 4.8} (see p. \pageref{eq4.7.16}), and $\xi_0 \in C^{\infty}_0 (B_{2^{1/2}})$,  $\xi  \in C^{\infty}_0 (\{2^{-1/2} <  |v| < 2\})$ be  functions such that
    \begin{itemize}[--]
    \item $\xi_0 = 1$ on $\overline{B}_1$, $0 \le \xi_0 < 1$ on $\{|v| > 1\}$,
    \item   $\xi = 1$ on $\{1 \le |v| \le 2^{1/2}\}$, otherwise  $0 \le \xi < 1$.
    \end{itemize}
 For $n  \in \{ 1, 2, \ldots\}$, we set $\xi_n (v) = \xi (v 2^{-n})$ and observe that
$$
	 \zeta_n  = 1 \, \,  \, \text{on} \, \,  \text{supp}\,  \xi_n.
$$
In this proof, we assume that $N$  is a constant  depending only on $K, r, \theta, \Omega, \varkappa$.

By direct calculations,  the function $f^{(n)} : = f \xi_n$ satisfies
$$
	\mathfrak{L}_n f^{(n)} + \lambda  f^{(n)} =  \eta^{(n)} \, \, \text{in} \, \, \Sigma^T,   \quad f^{(n)} (0, \cdot) \equiv 0,
$$
and the specular reflection boundary condition,
where
\begin{align*}
&	\mathfrak{L}_n = Y  - \nabla_v \cdot (\sigma_{G, n} \cdot \nabla_v)   - a_g \cdot \nabla_v + \lambda,\\
&\sigma_{G, n} = (\sigma_G  + \nu I_3) \zeta_n + (1 -  \zeta_n) I_3, \\
&	\eta^{(n)} = \xi_n h  + \xi_n \overline{K}_g f    -  a_g \cdot (\nabla_v \xi_n)  f +  (\partial_{v_i} \sigma^{i j}_{G, n}) (\partial_{v_j}  \xi_{ n}) f  \\
&\quad \quad  \quad -   \sigma^{i j}_{G, n}  (\partial_{v_i v_j}  \xi_{ n }) f
- 2 \sigma^{i j}_{G, n} (\partial_{v_j}  \xi_{ n}) (\partial_{v_i}  f).
 \end{align*}

By Lemma \ref{lemma 4.8}  with  $\theta_r -  \beta$ in place of $\theta$, we get
\begin{equation}
			\label{eq4.7.12}
\begin{aligned}
	\|f^{(n)}\|_{ S_{p,  \theta_r  - \beta  - 2  } (\Sigma^T)  }&  \le N  2^{ \beta n} \big(\||f^{(n)}| + |\nabla_v f^{(n)}|\|_{ L_{r, \theta_r - \beta} (\Sigma^T) } \\
	&
	\quad +  \|\eta^{(n)}\|_{ L_{r, \theta_r - \beta- 2 } (\Sigma^T) }\big).
\end{aligned}
\end{equation}
By Lemma \ref{lemma C.2} with $\theta_r$ in place of $\theta$, we get
\begin{equation}
			\label{eq4.7.14}
	\|\jb^{  \theta_r - \beta - 2  } \xi_n \overline{K}_g f\|_{ L_r (\Sigma^T) } \le N  2^{     - \beta n  -2 n } \||f| + |\nabla_v f|\|_{ L_{r, \theta_r} (\Sigma^T) },
\end{equation}
and, therefore,
\begin{equation}
			\label{eq4.7.13}
\begin{aligned}
	\|f^{(n)}\|_{ S_{p,   \theta_r - \beta  - 2 } (\Sigma^T)  } & \le N   \|(|f| + |\nabla_v f| + |h|) |\zeta_n| \|_{ L_{r, \theta_r} (\Sigma^T) }\\
	&   \quad +  N  2^{  - 2 n  }  \||f| + |\nabla_v f|\|_{ L_{r, \theta_r} (\Sigma^T) }.
\end{aligned}
\end{equation}
 Raising \eqref{eq4.7.13} to the power $p$ and summing up, we prove the validity of the claim $(\ref{4.7.i})$.
Finally, the assertions $(\ref{4.7.ii})$ and $(\ref{4.7.iii})$ of the claim follow from Lemma \ref{lemma 4.8} and the  estimate \eqref{eq4.7.13}.
The proposition is proved.
\end{proof}

\begin{proof}[Proof of Theorem \ref{theorem 11.2}]
\textbf{Uniqueness.}
Let $\varepsilon = \varepsilon (\theta)$ be a number in Lemma \ref{lemma 4.6}.
Let $f_1$ and $f_2$ be any  two finite energy strong solutions to Eq. \eqref{11.1}. Note that $u = f_1 - f_2$
satisfies Eq. \eqref{11.2} with $\nu = 0, \lambda = 0$, and $f_0 \equiv 0$.
The uniqueness now follows from Lemma \ref{lemma 4.6}.

\textbf{Existence.}
We assume, additionally, that $\varepsilon > 0$ is small enough so that \eqref{eq11.2.1} holds.
Replacing $f$ with $f - f_0 \phi$ where $\phi = \phi (t)$ is a cutoff function such that
$\phi (0) = 1$, we reduce \eqref{eq11.1} to the forced Landau equation
\begin{equation}
			\label{eq11.1'}
\begin{aligned}
	&Y f - \nabla_{v} \cdot (\sigma_G \nabla_{v} f) - a_g \cdot \nabla_v f - \overline{K}_g f = h \, \, \text{in} \, \, \Sigma^T,\\
	 &
	f (0, x, v) = 0, \, \,  (x, v) \in \Omega \times \bR^3,  \quad
		 f_{-} (t, x, v) = f_{+} (t, x, R_x v), z \in \Sigma^T_{-},
\end{aligned}
\end{equation}
where
$$
	h =   - \partial_t \phi f_0  - v \cdot \nabla_x (f_0) \phi + \nabla_v \cdot (\sigma_G \nabla_v f_0) \phi + a_g \cdot (\nabla_v f_0) \phi +  \phi \overline{K}_g f_0.
$$
Note that due to the $L_{\infty}$ estimates of $\sigma_G, \nabla_v \sigma_G, a_g$ (see \eqref{eq11.2.1} - \eqref{11.2.9}) and Remark \ref{remark 10.1}, one has
\begin{equation}
			\label{eq11.2.3}
	\|h\|_{  L_{2, \theta} (\Sigma^T) } + \|h\|_{  L_{\infty} (\Sigma^T) } \le N (K, \theta, \Omega) ( |f_0|_{ \cO_{2, \theta} } + |f_0|_{ \cO_{\infty} }).
\end{equation}

\textbf{Step 1: well-posedness of a viscous approximation scheme.}
We consider the equation
\begin{equation}
		\label{eq11.2.12}
\begin{aligned}
	&Y f^{(\nu)}  - \nabla_v \cdot (\sigma_G \nabla_v f^{(\nu)}) - \nu \Delta_v f^{(\nu)}  - a_g \cdot \nabla_v f^{(\nu)} - \overline{K}_g f^{(\nu)}  = h \,\, \text{in} \, \, \Sigma^T, \\
&
	 f^{(\nu)} (0, \cdot) = 0  \,\, \text{in} \, \, \Omega \times \bR^3,\,\,
 f^{(\nu)}_{-} (t, x, v) =  f^{(\nu)}_{+} (t, x, R_x v), \,   z \in \Sigma^T_{-}.
\end{aligned}
\end{equation}

 By Proposition \ref{lemma 10.2}, for any $\nu \in (0, 1]$ and  $\theta \ge 16$,
there exists $\lambda = \lambda (\theta, K, \nu, \varkappa, \Omega) > 0$ such that
Eq. \eqref{eq11.2.12} has a unique finite energy strong solution $f^{(\nu)}_{\lambda} \in \cK_{\theta, \lambda, p} (\Sigma^T)$
 in the sense of Definition \ref{definition 11.0}.
Then, the function
$$
	f^{(\nu)}: = e^{ \lambda t} f^{(\nu)}_{\lambda} \in  \cK_{\theta, \lambda, p } (\Sigma^T)
$$
is a finite energy strong solution to Eq. \eqref{eq11.1'} in the sense of  Definition \ref{definition 11.0}.

\textbf{Step 2: uniform bounds for $f^{(\nu)}$.}

\textit{Weighted energy bound.}
By   Lemma \ref{lemma 4.6} and Remark \ref{remark 2.17}, for sufficiently small $\varepsilon  = \varepsilon (\theta) > 0$,  we have
\begin{align}
			\label{11.2.8}
	& \|f^{(\nu)} (T, \cdot)\|_{ L_{2, \theta} (\Omega \times \bR^3) } + \|f^{(\nu)}\|_{  \sigma, \theta  } +  \|\nabla_v f^{(\nu)}\|_{ L_{2, \theta - 3} (\Sigma^T)  }   \\
	&  +    \|f^{(\nu)}\|_{ L_{2, \theta } (\Sigma^T)  }
	 \le   N (\theta, T) \|h\|_{ L_{2, \theta} ( \Sigma^T) }.\notag
\end{align}

\textit{$S_p$ bound.}
 By Proposition \ref{proposition 4.7} with $\theta - 3$ in place of $\theta$,
for sufficiently small $\varepsilon > 0$ and sufficiently large $\theta = \theta (\varkappa, p) > 4$, and $\theta' = \theta' (\varkappa, p), \theta'' = \theta'' (\varkappa, p) \in (1, \theta -3)$,  one has
\begin{align}
			\label{eq11.2.10}
& \|f^{(\nu)}\|_{ S_{2,  \theta' } (\Sigma^T) } +  \|f^{(\nu)}\|_{ S_{p, \theta'' } (\Sigma^T) }
	+ \|f^{(\nu)}_{\pm}\|_{ L_{\infty} (\Sigma^T_{\pm}, |v \cdot n_x|) }\\
&
	\quad\le N (\|h\|_{ L_{2, \theta-3} (\Sigma^T) } + \|h\|_{ L_{p, \theta-3} (\Sigma^T) }  + \|f^{(\nu)}\|_{ L_{2, \theta-3} (\Sigma^T) } + \|\nabla_v f^{(\nu)}|\|_{ L_{2, \theta - 3} (\Sigma^T) })\notag,
\end{align}
where
 $N = N (\theta, p, \varkappa, K,  \Omega)$.
Combining \eqref{11.2.8} with \eqref{eq11.2.10} and \eqref{eq11.2.3} gives
\begin{align}
	& \label{eq11.2.7}
	\|f^{(\nu)}\|_{ S_{2,  \theta' } (\Sigma^T) } +  \|f^{(\nu)}\|_{ S_{p,  \theta'' } (\Sigma^T) }  +  \|f^{(\nu)}_{\pm}\|_{ L_{\infty} (\Sigma^T_{\pm}, |v \cdot n_x|) }\\
	&
		 \quad \le N (\|h\|_{ L_{2, \theta} (\Sigma^T) } + \|h\|_{ L_{\infty} (\Sigma^T) })\notag,
\end{align}
where $N = N (\theta, p, \varkappa, K,  \Omega, T)$.

\textbf{Step 3: limiting argument.}
By \eqref{eq11.2.10}, the Banach-Alaoglu theorem, and   Eberlein-Smulian theorem,  there exists a subsequence $\nu'$ and  function $f \in S_{2, \theta' } (\Sigma^T)$ such that
\begin{itemize}
\item[--]$f^{(\nu')}  \to f$ weakly in $S_{2,  \theta' } (\Sigma^T)$;

\item[--] $f^{(\nu')}_{\pm} \to f_{\pm}$ in the weak* topology of $L_{\infty} (\Sigma^T_{\pm}, |v \cdot n_x|)$, respectively;

\item[--] $f^{(\nu')} (T, \cdot) \to f (T, \cdot)$ weakly  in $L_{2, \theta} (\Omega \times \bR^3)$;

\item[--] $f^{(\nu')}  \to f$ weakly in $H_{\sigma, \theta} (\Sigma^T)$.
\end{itemize}
Passing to the limit in   the Green's identity \eqref{7.2} (see Remark \ref{remark 2.2.3}),  we conclude that $f (T, \cdot), f_{\pm}$ are, indeed, traces of the function $f$.
Using the weak and the weak* convergence, we pass to the limit in \eqref{11.2.8} and \eqref{eq11.2.7} and prove
\begin{equation}
		\label{eq11.2.13}
\begin{aligned}
		&\|f (T, \cdot)\|_{ L_{2, \theta} (\Omega \times \bR^3) } + \|f\|_{  \sigma, \theta  } +  \|\nabla_v f\|_{ L_{2, \theta - 3} (\Sigma^T)  } \\
		& + \|f\|_{ S_{2,  \theta' } (\Sigma^T) } +  \|f\|_{ S_{p,  \theta'' } (\Sigma^T) }  +  \|f_{\pm}\|_{ L_{\infty} (\Sigma^T_{\pm}, |v \cdot n_x|) }\\
		&\le  N (\|h\|_{ L_{2, \theta} (\Sigma^T) } + \|h\|_{ L_{\infty} (\Sigma^T) }).
\end{aligned}
\end{equation}

Next, we show that $f$ satisfies the weak formulation of Eq. \eqref{eq11.1'} (cf. Definition \ref{definition 11.0}).
Recall that due to Lemma \ref{lemma 10.2},  $f^{(\nu)}$ satisfies the weak formulation of \eqref{eq11.2.12} with any $\phi \in C^1_0 (\overline{\Sigma^T})$.
Testing Eq. \eqref{eq11.2.12}  with $\phi \in C^1_0 (\overline{\Sigma^T} )$
gives
\begin{equation}
			\label{eq11.2.11}
\begin{aligned}
	&\int_{\Omega \times \bR^3} f^{(\nu)} (T, \cdot) \phi (T, x, v) \, dxdv \\
	&+ \int_{\Sigma^T_{+}} f^{(\nu)}_{+} \phi \, |v \cdot n_x| \, d\sigma dt
	  - \int_{\Sigma^T_{-}}  f^{(\nu)}_{-} \phi |v \cdot n_x| \, d\sigma dt \\
	&+ \int_{\Sigma^T}  [(\sigma_G + \nu I_3) \nabla_v f^{(\nu)}] \cdot \nabla_v \phi \, dz   = \int_{\Sigma^T} [a_g \cdot \nabla_v f^{(\nu)} + \overline{K}_g f^{(\nu)} + h] \phi \, dz.
\end{aligned}
\end{equation}
It follows from the definition of $\sfK$ (see \eqref{eq11.1.2}) that
$$
	\int_{\Sigma^T} (\sfK f^{(\nu)}) \phi \, dz = \int_{\Sigma^T} (\sfK  \phi) f^{(\nu)} \, dz,
$$
and, hence, by the explicit expression of the term $J_g f^{(\nu)}$ (see \eqref{eq11.2.1'}) and the fact that $\overline{K}_g = \sfK + J_g$, we obtain
$$
	\int_{\Sigma^T} (\overline{K}_g f^{(\nu)}) \phi \, dz = \int_{\Sigma^T} (\overline{K}_g  \phi) f^{(\nu)} \, dz.
$$
Then, by the weak convergence $f^{(\nu)} \to f$ in $L_2 (\Sigma^T)$ and Lemma \ref{lemma C.2}, we conclude
$$
	\lim_{\nu \to 0} \int_{\Sigma^T} (\overline{K}_g f^{(\nu)}) \phi \, dz  = \int_{\Sigma^T} (\overline{K}_g   f) \phi \, dz.
$$
Finally, by what just said and the aforementioned weak convergence,  we pass to the limit in \eqref{eq11.2.11} and  we conclude that $f$  satisfies the weak formulation of \eqref{eq11.1}.
By using the weak* convergence of $f^{\nu}_{\pm}$ in $L_{\infty} (\Sigma^T_{\pm}, |v \cdot n_x|)$, we  pass to the limit in the specular boundary condition for $f^{(\nu)}$ and prove that this boundary condition holds for the function $f$ as well.
 Using the  Green's identity again (see Remark \ref{remark 2.2.3}), we prove that $f$ satisfies  Eq. \eqref{eq11.1} a.e. and, thus, $f$ is a finite energy strong solution in the sense of Definition \ref{definition 11.0}.

\textbf{Step 4:  H\"older estimate.}
To bound the
$$
	L_{\infty} ((0, T), C^{\alpha/3, \alpha}_{x, v} (\overline{\Omega} \times \bR^3))$$
 norm
of $f$ and $\nabla_v f$, we repeat the above $L_2$ to $S_p$ bootstrap argument.
Again,  by Proposition \ref{proposition 4.7} with  $\theta$ replaced with $\theta''$, for sufficiently small $\varepsilon > 0$ and sufficiently large $\theta = \theta (\varkappa, p)$, we get
\begin{align*}
& \|f\|_{  L_{\infty} ((0, T), C^{\alpha/3, \alpha}_{x, v} (\overline{\Omega} \times \bR^3)) } + \|\nabla_v f\|_{  L_{\infty} ((0, T), C^{\alpha/3, \alpha}_{x, v} (\overline{\Omega} \times \bR^3)) }\\
&
	\quad\le N (\|h\|_{ L_{2, \theta''} (\Sigma^T) } + \|h\|_{ L_{p, \theta''} (\Sigma^T) }  + \|f\|_{ L_{2, \theta''} (\Sigma^T) } + \|\nabla_v f\|_{ L_{2, \theta''} (\Sigma^T) }).
\end{align*}
Combining this with \eqref{eq11.2.13}, we obtain
\begin{equation}
			\label{eq11.2.14}
\begin{aligned}
& \|f\|_{  L_{\infty} ((0, T), C^{\alpha/3, \alpha}_{x, v} (\overline{\Omega} \times \bR^3)) } + \|\nabla_v f\|_{  L_{\infty} ((0, T), C^{\alpha/3, \alpha}_{x, v} (\overline{\Omega} \times \bR^3)) }\\
&
	\quad\le N (\|h\|_{ L_{2, \theta''} (\Sigma^T) } + \|h\|_{ L_{\infty} (\Sigma^T) }).
\end{aligned}
\end{equation}

Finally, recall that  at the beginning of the proof, we replaced $f$ with $f  - f_0 \phi$. By this,  the definition of the $\cO_{p, \theta}$ norm (see \eqref{1.2.7}),
and \eqref{eq11.2.13}, and \eqref{eq11.2.14} , we prove the desired estimate \eqref{eq11.2.0}
\end{proof}

\appendix
\section{Verification of the identity (\ref{7.3.2})}
            \label{Appendix A}

Invoke the assumptions and the notation of Section \ref{subsection 1.4.1}
and denote
$$
    M = \big(\frac{\partial x}{\partial y}\big).
$$
Let $f$ be a finite energy weak solution to Eq. \eqref{7.0} (see Definition \ref{definition 7.0}) supported on $\bR \times \Omega_{r_0/2} \times \bR^3$ and $\phi \in C^1_0 (\overline{\Sigma^T})$.
The goal of this section is to justify the identity \eqref{7.3.2}.

\textit{Drift term.} By using a change of variables and the identity
\begin{equation}
                \label{eqA.5}
     \big(\nabla_v f \big) (t, x (y), v (y, w))  = \bigg(\big(\frac{\partial x}{ \partial y}\big)^T\bigg)^{-1} \nabla_w \widehat f  (t, y, w)  = (M^{-1})^T \nabla_w \widehat f  (t, y, w),
\end{equation}
we get
\begin{equation}
			\label{eqA3}
    \int_{\Sigma^T} b \cdot \nabla_v f  \phi \, dz
    = \int_{\mathbb{H}_{-}^T} \bigg[\underbrace{(M^{-1} b)}_{= B} \cdot \nabla_w (\widehat f J)\bigg]  \widehat\phi  \, d\widehat z,
\end{equation}
where  $\bH_{-}^T$ is defined in \eqref{eq1.2.0}.

\textit{Transport term.}
First, changing variables gives
$$
    \int_{\Sigma^T}  (v \cdot \nabla_x \phi) f \, dz
    =
    \int_{\mathbb{H}_{-}^T} (\widehat f J) \,  (M w)^T  \big(\nabla_x \phi\big) (t, x (y), v (y, w))  \, d\widehat z.
$$
	Next, by the chain rule and \eqref{eqA.5},
\begin{align*}
      &  \big(\nabla_x \phi\big) (t, x (y), v (y, w))\\
 &   = 	  \bigg(\big(\frac{\partial x}{ \partial y}\big)^T\bigg)^{-1}   \big[\nabla_y  	\widehat \phi (t, y, w)
	  -  \big(\frac{\partial v}{\partial y}\big)^T (\nabla_v \phi) (t, x (y), v (y, w))\big]\\
&
     = (M^{-1})^T  \nabla_y  	\widehat \phi (t, y, w)
	  - (M^{-1})^T \big(\frac{\partial v}{\partial y}\big)^T (M^{-1})^T \nabla_w \widehat \phi  (t, y, w),
 \end{align*}
and then,
\begin{align*}
        &v^T (y, w)  \big(\nabla_x \phi\big) (t, x (y), v (y, w)) \\
   & =  w^T  \nabla_y  	\widehat \phi (t, y, w)
    -    w^T \big(\frac{\partial v}{\partial y}\big)^T (M^{-1})^T \nabla_w \widehat \phi (t, y, w)\\
  &  = w \cdot   \nabla_y  	\widehat \phi (t, y, w)  - X \cdot \nabla_w  \widehat \phi (t, y, w),
 \end{align*}
 where
 $$
        X =  (X_1, X_2, X_3)^T =   M^{-1} \big(\frac{\partial v}{\partial y}\big) w
        =  M^{-1} \frac{\partial (M w)} {\partial y} w.
 $$
Thus, by the above computations,
\begin{equation}
                \label{eq.A1}
         \int_{\Sigma^T}  (Y \phi) f
    =  \int_{  \mathbb{H}_{-}^T }
   \big( Y (\widehat \phi)  - X \cdot \nabla_w \widehat \phi\big)   (\widehat f  J)  dz.
\end{equation}

\textit{Diffusion term.}
Applying  formula  \eqref{eqA.5} gives
\begin{equation}
			\label{eqA.4}
	\int_{\Sigma^T} (a \nabla_v \phi)^T \nabla_v f \, dz = \int_{  \mathbb{H}_{-}^T }  (A \nabla_w  \widehat \phi)^T  \nabla_w (\widehat f J)  \, dz,
\end{equation}
where
$$
	A  = M^{-1} \widehat a (M^{-1})^T.
$$

Finally, combining \eqref{eqA.5} - \eqref{eqA.4} and recalling $\widetilde f = \widehat f J$, we arrive to \eqref{7.3.2}.

\section{}
            \label{Appendix B}

The following lemma can be derived from the calculations in Appendix \ref{Appendix A}.
\begin{lemma}
            \label{lemma B.2}
Let $p > 1$, be a number, $T \in (-\infty, \infty]$, and $\Phi$ be the diffeomorphism
defined in Subsection \ref{subsection 1.4.1} (see \eqref{eq1.4.2.4} - \eqref{eq1.4.2.8}).
Then, for any function $u$ vanishing outside $(-\infty, T) \times B_{r_0} (x_0) \times \bR^3$, one has
$$
  \|u\|_{ S_p ((-\infty, T) \times \bR^6) } \le N (\Omega, p)  \|\widehat u\|_{ S_p ((-\infty, T) \times \bR^6) },
$$
where $\widehat u$ is defined in \eqref{eq1.4.2.1}.
\end{lemma}

\begin{lemma}
            \label{lemma B.3}
Let $\theta \ge 0$ be a number and $u$ be a function on $\Sigma^T$ such that $u, Y u \in L_{2, \theta} (\Sigma^T)$,
$u (T, \cdot), u (0, \cdot) \in L_{2, \theta} (\Omega \times \bR^3)$,
 $u_{\pm} \in L_{\infty} (\Sigma^T_{\pm}, |v \cdot n_x|)$,
 and the specular reflection boundary condition  is satisfied.
Then, the following variant of the energy identity holds:
$$
    \int_{ \Omega \times \bR^3 }  \big(u^2 (T, x, v) -   u^2 (0, x, v)\big) \, \jb^{\theta} dxdv = 2 \int_{\Sigma^T} u (Y u) \, \jb^{\theta} dz.
$$
\end{lemma}

\begin{proof}
For $\varepsilon > 0$, denote
$$
    \mu_{\varepsilon} (v) = e^{-|v|^2/\varepsilon}.
$$
Note that $u_{\varepsilon}: = u \mu_{\varepsilon} \in E_{2, \theta} (\Sigma^T)$ (see Definition \ref{definition 2.2.1} and \eqref{7.2}) and it satisfies the specular reflection boundary condition.
Then, by the energy identity (see \eqref{eq7.1.1} in Lemma \ref{lemma 7.1}),
$$
    \int_{\Omega \times \bR^3}  [u^2_{\varepsilon}  (T, x, v)  -   u^2_{\varepsilon} (0, x, v)]  \, \jb^{\theta} dxdv = 2 \int_{\Sigma^T} u_{\varepsilon} (Y u_{\varepsilon})  \, \jb^{\theta} dz.
$$
Passing to the limit in the above equality and using the dominated convergence theorem, we prove the lemma.
\end{proof}

For $T \in \bR$, let $\langle\cdot, \cdot\rangle_T$ be the duality pairing between $\bH^{-1}_2 (\bR^7_T)$ and $\bH^1_2 (\bR^7_T)$  given by
$$
	\langle f, g \rangle_{T} = \int_{-\infty}^T \int_{\bR^3} [f (t, x, \cdot), g (t, x, \cdot)] \, dx dt,
$$
where
$$
	[f, g] = \int ((1- \Delta_v)^{-1/2} f) \,  ((1- \Delta_v)^{1/2} g) \, dv.
$$

The proof of the  following variant of the energy identity can be found in \cite{DY_21b}.
\begin{lemma}
			\label{lemma B.5}
Let   $T \in \bR$ be a  number and   $u \in \bH^1_2 ((-\infty, T) \times \bR^6)$, and $Y u \in \bH^{-1}_2 ((-\infty, T) \times \bR^6)$.
Then, for a.e. $s \in  (-\infty, T]$,
$$
	\langle Y u,   u \rangle_s =  (1/2)\|u\|^2_{  L_2 (\bR^6)} (s).
$$
\end{lemma}

\begin{lemma}
		\label{lemma B.4}
Let $T > 0$ be a number  and $u \in E_2 (\Sigma^T)$.
Then, $u \in C ([0, T], L_2 (\Sigma^T))$.
\end{lemma}

\begin{proof}
We only prove the continuity at $t = 0$ because the argument for other  points is similar.
First, note that by the energy identity,
$$
	\lim_{t \to 0+}   \|u (t, \cdot)\|_{  L_2 (\Omega \times \bR^3) } = \|u (0, \cdot)\|_{  L_2 (\Omega \times \bR^3) }.
$$
Hence, we only need to show that $u (t, \dot)$ is  weakly continuous in $L_2 (\Omega \times \bR^3)$   at $t  = 0$.

We fix an arbitrary test function
$$
	\phi \in	C^1_0 \big(([0, T] \times \overline{\Omega} \times \bR^3) \setminus ( (0, T) \times \gamma_0 \cup \{0\} \times \partial \Omega \times \bR^3   \cup \{T\} \times \partial \Omega \times \bR^3)\big),
$$
which belongs to the set $\mathsf{\Phi}$ (see Definition \ref{definition 2.2.3}) by Lemma 2.1 of \cite{G_93}.
Note that by the Green's identity (see \eqref{7.2}), we have
$$
	\lim_{t \to 0+ } \int_{\Omega \times \bR^3} u (t, x, v) \phi (t, x, v) \, dxdv   =	\int_{\Omega \times \bR^3} u (0, x, v) \phi (0, x, v) \, dxdv.
$$
We claim that the above convergence also holds if we replace   $\phi (t, x, v)$ with $\phi (0, x, v)$. To prove this, we first note that by Remark \ref{remark 2.2.2} and the energy identity  \eqref{eq7.1.1} in Lemma \ref{lemma 7.1}, we have
\begin{equation}
		\label{B.4.1}
	u \in L_{\infty} ((0, T), L_2 (\Omega \times \bR^3)).
\end{equation}
By this and the Cauchy-Schwartz inequality,
\begin{align*}
	&\bigg|\int_{\Omega \times \bR^3} u (t, x, v) (\phi (0, x, v) -   \phi (t, x, v))  \, dxdv\bigg| \\
	&\le \| u\|_{  L_{\infty} ((0, T), L_2 (\Omega \times \bR^3))  } \|\phi (0, \cdot) - \phi (t, \cdot)\|_{    L_2 (\Omega \times \bR^3) }.
\end{align*}
The right-hand side of the above inequality converges to $0$ as $t \to 0+$ due our choice of $\phi$, and this  proves the above claim.
Hence, for any continuously differentiable function $\xi$ with compact support in  $\Omega \times \bR^3$,
$$
	 \lim_{t \to 0+} \int_{\Omega \times \bR^3} u (t, x, v) \xi (x, v) \, dxdv = \int_{\Omega \times \bR^3} u (0, x, v) \xi (x, v) \, dxdv.
$$
By using a standard approximation argument combined with \eqref{B.4.1}, we conclude that $u (t, \cdot) \to u (0, \cdot)$ weakly in $L_2 (\Omega \times \bR^3)$
as $t \to 0+$. The lemma is proved.
\end{proof}

\section{}
            \label{Appendix C}
\begin{lemma}
			\label{lemma C.1}
Let $\varkappa \in (0, 1]$, $g \in L_{\infty} ((0, T),  C^{\varkappa/3, \varkappa}_{x, v} ( \overline{\Omega} ))$, and   $\sigma_G$ be a function defined in \eqref{eq11.5}.
Then, we have
$$
	 \|\sigma_G\|_{ L_{\infty} ((0, T),  C^{\varkappa/3, \varkappa}_{x, v} ( \overline{\Omega} ))  } \le N (1+ \|g\|_{ L_{\infty} ((0, T),  C^{\varkappa/3, \varkappa}_{x, v} (  \overline{\Omega}))  }),
$$
where
$N  = N (\varkappa) > 0$.

\end{lemma}
\begin{proof}

Fix  almost any $t \in (0, T)$ and arbitrary $x_1, x_2 \in \overline{\Omega}$ recall that by Lemma 3 of \cite{G_02}, for any $v \in \bR^3$,
\begin{equation}
		\label{eqC.1.1}
	|\sigma_G (t, x_1, v)| + |\nabla_v \sigma_G (t, x_1, v)| \le N (1+\|g (t, x, \cdot)\|_{ L_{\infty} (\Sigma^T) }),
\end{equation}
where $N$ is independent of $\varkappa$.
Next, observe that
$$
	\sigma_G (t, x_1, v) - \sigma_G (t, x_2, v) = \big(\Phi \ast [\mu^{1/2} (g (t, x_1, \cdot) - g (t, x_2, \cdot))]\big) (v).
$$
Replacing $g (t, x_1, v)$ with $g (t, x_1, v) - g (t, x_2, v)$ and using   \eqref{eqC.1.1}, we get
$$
	|\sigma_G (t, x_1, v) - \sigma_G (t, x_2, v) | 
\le N |x_1-x_2|^{\varkappa/3}\sup_{ t \in (0, T), v \in \bR^3  } \|g (t, \cdot, v)\|_{  C^{\varkappa/3} (\overline{\Omega})   },
$$
where $C^{\varkappa/3} (\overline{\Omega})$ is the usual  H\"older space on $\overline{\Omega}$.
Combining the above inequalities and using the interpolation inequality for the H\"older norms, we prove the lemma.
\end{proof}

\begin{lemma}
			\label{lemma C.2}
Let $p > 3/2$ be a number  and $u \in L_p (\Sigma^T)$ be a function such that $\nabla_v u \in L_p (\Sigma^T)$.
For any $m =  \{0, 1, 2, \ldots\}$, we denote
$$
	\{v \sim m\} = \begin{cases}
				|v| < 1, \, \, m = 0, \\
			        c m < |v| < c^{-1} m, \, \, m \ge 1,
				\end{cases}
$$
where $c \in (0, 1)$ is some number.
If the condition \eqref{con11.1} holds, then for any $\theta \ge 0$,
$$
	\| \overline{K}_g u\|_{L_p ((0, T) \times \Omega \times \{v \sim m\})  }
	\le N  m^{-\theta} \||u| + |\nabla_v u|\|_{L_{p, \theta} (\Sigma^T)},
$$
where $\overline{K}_g$ is defined in \eqref{eq11.2} and $N = N (\varkappa, K, \theta, p, c) > 0$.
\end{lemma}

\begin{proof}
Recall that $\overline{K}_g = \sfK + J_g$, where $\sfK$ and $J_g$ are defined in \eqref{eq11.1.2} and \eqref{eq11.2.1'}.
Furthermore, by the definition of the collision kernel $\cQ$ (see \eqref{eq11.10}),
\begin{align*}
	\sfK u &= -\mu^{-1/2}  \partial_{v_i} \bigg(\mu \big(\Phi^{i j} \ast (\mu^{1/2}[\partial_{v_j} u + v_j u])\big)\bigg)\\
	& =  2 v_i \mu^{1/2} \bigg(\Phi^{i j} \ast (\mu^{1/2}[\partial_{v_j} u + v_j u])\bigg)\\
	&- \mu^{1/2}  \bigg( \Phi^{i j} \ast \partial_{v_i} (\mu^{1/2} v_j u)\bigg)
	- \mu^{1/2}  \bigg(\partial_{v_i} \Phi^{i j} \ast (\mu^{1/2} \partial_{v_j} u)\bigg)
 := \sfK_1 u + \sfK_2 u + \sfK_3 u.
\end{align*}

\textit{Estimate of $J_g$.}
By using Lemmas 2 and 3 of \cite{G_02} and the condition \eqref{con11.1}, one can show that
there exist a constant $N = N (K) > 0$ such that
$$
	\||\sigma| + |\nabla_v \sigma| + | \partial_{v_i}\Phi^{i j}  \ast (\mu^{1/2} \partial_{v_j} g)|\\
	 +   |\Phi^{i j}  \ast (\mu^{1/2} v_{   i  } \partial_{v_j} g)|\|_{ L_{\infty} (\bR^7_T) }   \le N (K).
$$
See  the details in   Lemma 3.6 of \cite{DGO}.
Furthermore, by Lemma 3 of \cite{G_02},
$$
	|\sigma^{i j} v_i v_j| \le N \jb^{-1}.
$$
By the above inequalities,
$$
	\| J_g u\|_{ L_p ((0, T) \times \bR^3 \times \{v \sim m\})  } \le N (K) \|u\|_{   L_p ((0, T) \times \bR^3 \times \{v \sim m\})  }.
$$

\textit{Estimate of $\sfK_1$ and $\sfK_2$.}
Note that for any $\theta \ge 0$,
\begin{equation}
			\label{eqC.2.2}
\begin{aligned}
	&\|\sfK_1 u\|_{L_p ((0, T) \times \bR^3 \times \{v \sim m\}) }\\
	& \le N (\theta, c) m^{-\theta} \| \mu^{1/4}  \big(|v|^{-1} \ast [\mu^{1/4} (|u| + |\nabla_v u|)])\big\|_{L_p (\bR^7_T)  }.
\end{aligned}
\end{equation}
Furthermore,  using the H\"older's inequality with $p$ and $q = p/(p-1) \in (1, 3)$,  for any $(t, x, v) \in \bR^7_T$, we have
$$
	 \int |v-v'| ^{-1} \mu^{1/4} (v') |u| (t, x, v') \, dv' \le \bigg(\int |v-v'|^{-q} \mu^{q/4} (v') \, dv'\bigg)^{1/q} \|u (t, x, \cdot) \|_{L_p (\bR^3)  }.
$$
By using  the fact that for $q \in (1, 3)$ and $v \in \bR^3$,
$$
	\int_{\bR^3}  |v-v'|^{-q} \mu^{q/4} (v') \, dv' \le N (q)
$$
combined with \eqref{eqC.2.2}, we obtain
\begin{equation}
		\label{eqC.2.3}
	\|\sfK_1 u\|_{L_p ((0, T) \times \bR^3 \times \{v \sim m\})  } \le N (\theta, c) m^{-\theta} \||u| + |\nabla_v u|\|_{L_p (\bR^7_T)  }.
\end{equation}
By the same argument, we show that
$\sfK_2$ is bounded by the right-hand side of \eqref{eqC.2.3}.

\textit{Estimate of $\sfK_3$.}
By direct calculations,
$$
	\partial_{v_i v_j} \Phi^{i j}  =  - 8 \pi \delta (x),
$$
and since $u \in L_p (\Sigma^T)$, we get
$$
	\partial_{v_i v_j} \Phi^{i j} \ast u = - 8 \pi u.
$$
Hence, integrating by parts in $\sfK_3 u$ and using the above identity, we prove that
$$
	\sfK_3 u = 8 \pi \mu u-\mu^{1/2} \Phi^{i j} \ast \partial_{v_i} (v_j \mu^{1/2} u).
$$
Then, repeating the above argument we used to estimate $\sfK_1 u$, we get
$$
		\|\sfK_3 u\|_{L_p ((0, T) \times \bR^3 \times \{v \sim m\})  } \le N (\theta, c) m^{-\theta} \||u| +|\nabla_v u|\|_{L_p (\bR^7_T)  }.
$$
The assertion of the lemma follows from the above estimates.
\end{proof}

\section{\texorpdfstring{$S_p$}{} regularity theory for kinetic Fokker-Planck equations with   rough coefficients}
									\label{Appendix D}
In this section, we present the main results of \cite{DY_21a}.

Denote
$$
	Q_r (z_0): = \{z: t_0 - r^2 < t < t_0, |x  - x_0 - (t-t_0) v_0|^{1/3} < r, |v-v_0| < r      \},	
$$
which we call a kinetic cylinder.

\begin{assumption}$(\gamma_{\star})$
                  \label{assumption 3.1}
There exists $R_0 > 0$ such that for any $z_0$ and $R \in (0, R_0]$,
\begin{equation}
			\label{eq3.1.0}
     \text{osc}_{x, v} (a, Q_r (z_0)) \le \gamma_{  \star  },
\end{equation}
where
\begin{align}
	\label{eq3.1}
  & \text{osc}_{x, v} (a, Q_r (z_0))\\
&= r^{-(4d+2)} \int_{t_0 - r^2}^{t_0} \int_{  D_r (z_0, t) \times D_r (z_0, t)}
   |a (t, x_1, v_1) - a (t, x_2, v_2)| \, dx_1dv_1 dx_2dv_2 \, dt\notag,
\end{align}
  and
  $$
    D_r (z_0, t) =  \{(x, v): |x  - x_0 - (t-t_0) v_0|^{1/3} < r, |v-v_0| < r      \}.
  $$
\end{assumption}

\begin{remark}
			\label{remark 3.1}
Note that the following assumption  is stronger than Assumption \ref{assumption 3.1}, but somewhat easier to verify in practice.

\begin{assumption}
				\label{assumption 3.1'}
There exists an increasing function $\omega: [0, \infty) \to [0, \infty)$
such that $\omega (0+) = 0$ and
$$
	\sup_{t, x, v} r^{-8d} \int_{  x_1, x_2 \in B_{r^3} (x) } \int_{v_1, v_2 \in B_r (v) } |a (t, x_1, v_1) - a (t, x_2, v_2)| \, dx_1dx_2\, dv_1dv_2 \le \omega (r).
$$
\end{assumption}

Furthermore, note that if $a \in  L_{\infty} ((0, T), C^{\alpha/3, \alpha}_{x, v}  (\bR^6)), \alpha \in (0, 1]$, then,
  Assumption \ref{assumption 3.1'} holds
with $\omega (r)  = N r^{\alpha}$
for some constant $N > 0$.
\end{remark}

\begin{assumption}
                  \label{assumption 3.2}
Let  $b = (b^1, b^2, b^3)^T$ and $c$ be  functions such that
\begin{equation}
			\label{eq3.2.0}
	\||b| +|c|\|_{ L_{\infty} ((-\infty, T) \times \bR^6) } \le K.
\end{equation}

\end{assumption}

The following theorem is a simplified version of Theorem 2.4 of \cite{DY_21a}.
\begin{theorem}
            \label{theorem D.1}
Let $p > 1$, $K > 0$, be numbers,
$T \in (-\infty, \infty]$. Let
Assumptions   \ref{assumption 2.1} (see \eqref{eq2.1.0}),  \ref{assumption 3.2}  hold.
There exists a constant
$$
	\gamma_{ \star  }   =  \gamma_{   \star  } (\delta, p)   > 0
$$
such that if Assumption \ref{assumption 3.1} $(\gamma_{  \star })$ holds,
then, the following assertions are valid.

$(i)$ There exists a constant
$$
	\lambda_0 =    \lambda_0 (p, \delta,  K, R_0)  \ge 0
$$
such that
for any $\lambda \ge \lambda_0$
and any  $u \in S_p ((-\infty, T) \times \bR^6)$,
one has
\begin{align}
			\label{2.1.2}
 	&  \|\lambda |u| + \lambda^{1/2}  |\nabla_v u| + |D^2_v u| +  |(-\Delta_x)^{1/3} u|
	+|D_v (-\Delta_x)^{1/6} u|
	+ |Y u|\|_{ L_{p} ((-\infty, T) \times  \bR^{6})  } \notag \\
&    \le N  \|Y u - a^{i j} \partial_{v_i v_j} u +   b \cdot \nabla_{v} u + c u + \lambda u\|_{ L_{p} ((-\infty, T) \times  \bR^{6})  },
\end{align}
where $R_0 \in (0, 1)$ is the constant in  Assumption \ref{assumption 3.1} $(\gamma_{   \star  })$, and
 $
	N = N (p, \delta,  K).
$
In addition, for any $f \in L_{p} ((-\infty, T) \times  \bR^{6})$,
the equation
\begin{equation}
                \label{2.1.1}
        Y u  - a^{i j} \partial_{v_i v_j} u +  b \cdot \nabla_v u + c u + \lambda u = f
\end{equation}
has a unique solution $u \in S_{p} ((-\infty, T) \times  \bR^{6}).$

$(ii)$  For any numbers $-\infty<S < T<\infty$, $\lambda \ge 0$,
and
$f \in L_p ((S, T) \times \bR^{6})$,
the equation
$$
	   Y u  - a^{i j} \partial_{v_i v_j} u +  b \cdot \nabla_v u + c u + \lambda u = f, \quad u (0, \cdot) = 0
$$
 has a unique solution
$u  \in S_{p} ((S, T) \times \bR^{6})$.
In addition,
\begin{align*}
	  \|u\| + \|D_v u\| + \|D^2_v u\| +  \|(-\Delta_x)^{1/3} u\|
	+\|D_v (-\Delta_x)^{1/6} u\|
	+ \|Y u\|    \le N \|f\|,
\end{align*}
where $\|\,\cdot\,\|=\|\,\cdot\,\|_{L_p ((S, T) \times \bR^{6})  }$ and $N = N (\delta, p, K,  T-S)$.
\end{theorem}

\begin{theorem}[Corollary 2.6 of \cite{DY_21a}]
                                                \label{corollary 3.4}
Under Assumptions \ref{assumption 2.1} (see \eqref{eq2.1.0}),  \ref{assumption 3.2} (see \eqref{eq3.2.0}),
there exist constants
$$
\kappa = \kappa (p) > 0, \quad
\beta = \beta (p) > 0,
\quad	\gamma_{  \star  }  =   \delta^{\beta}  \widetilde \gamma_{  \star } (p)  > 0,
$$
such that if \eqref{eq3.1.0} in Assumption \ref{assumption 3.1} $(\gamma_{ \star })$ holds,
then for any $u \in S_{p} ((-\infty, T) \times \bR^6)$ and $\lambda \ge 0$,
\begin{equation}
				\label{2.2.1}
\begin{aligned}
	\|u\|_{S_{p} ((-\infty, T) \times \bR^6) }
 &  \le N \delta^{-\beta} ( \|Y u  - a^{i j}\partial_{v_i v_j} u + b \cdot \nabla_v u + c u + \lambda u \|_{ L_{p} ((-\infty, T) \times \bR^6) }\\
&\quad+ R_0^{-2}\|u\|_{ L_{p} ((-\infty, T) \times \bR^6) }),
\end{aligned}
\end{equation}
where      $N  = N (p,  K)$, and $R_0 \in (0, 1)$ is the constant in Assumption \ref{assumption 3.1} $(\gamma_{    \star  })$.

\end{theorem}

\begin{remark}	
			\label{remark 3.4}
The above theorem holds with $(-\infty, T)$ replaced with $(0, T)$ provided that $u (0, \cdot) \equiv 0$.
To prove this we use the estimate \eqref{2.2.1} with $u$ replaced with $u 1_{t \ge 0}$, which belongs to $S_p ((-\infty, T) \times \bR^6)$  because $u (0, \cdot) \equiv 0$.
\end{remark}

\begin{lemma}
			\label{lemma B.6}
Let $\lambda \ge 0$,  $1 <  q < p$ be numbers,   functions $a, b, c$ satisfy Assumptions \ref{assumption 2.1} (see \eqref{eq2.1.0}),  \ref{assumption 3.1'} and \ref{assumption 3.2},
and $u \in S_q (\bR^7_T)$
be a function such that $u (0, \cdot) \equiv 0$, and
$$
	h: = Y u - a^{i j} \partial_{v_i v_j} u + b \cdot \nabla_v u + (c + \lambda) u  \in L_p  (\bR^7_T).
$$
Then, $u \in S_p  (\bR^7_T)$.
\end{lemma}

\begin{proof}
We follow the proof of Theorem 4.3.12 $(i)$ of \cite{Kr_08}.
By Theorem \ref{theorem D.1} $(ii)$, the equation
$$
	Y U - a^{i j} \partial_{v_i v_j} U + b \cdot \nabla_v U + (c + \lambda) U = h, \quad U (0, \cdot) \equiv 0
$$
has a unique solution $U \in S_p  (\bR^7_T)$. We will prove that $u = U$ a.e.
By using an induction argument, we may assume that
\begin{equation}
			\label{B.6.1}
	\frac{1}{q} - \frac{1}{p} < \frac{1}{4d}.
\end{equation}
Let  $\phi, \xi \in C^{\infty}_0 (\bR^6)$ be functions such that $\phi (0) = 1$, $\xi (0) = 1$,   and $\text{supp} \, \phi \subset \text{supp} \, \xi$.
For $n \ge 1$, denote
$$
	\phi_n (x, v) = \phi (x/n^3, v/n), \quad \xi_n (x, v) = \xi (x/n^3, v/n).
$$
Then, $U_n := U \phi_n \in S_q (\bR^7_T)$ satisfies the equation
\begin{align*}
	&Y U_n - a^{i j} \partial_{v_i v_j} U_n + b \cdot \nabla_v U_n + (c + \lambda) U_n \\
&= h \phi_n + U (v \cdot \nabla_x \phi_n - a^{i j} \partial_{v_i v_j} \phi_n + b \nabla_v \phi_n)\\
&  - 2 (a \nabla_v \phi_n) \cdot \nabla_v U , \quad U_n (0, \cdot) \equiv 0,
\end{align*}
 and, in addition, by Theorem \ref{theorem D.1} $(ii)$,
$$
	 \|U_n\|_{ S_q  (\bR^7_T)  } \le N \|h \phi_n\|_{ L_q  (\bR^7_T)  } + N n^{-1} \|(|U| + |\nabla_v U|) |\xi_n|\|_{ L_q  (\bR^7_T)  },
$$
where $N = N (q, \delta, K, T)$.
By  using the H\"older's inequality with $p/q$ and $p/(p-q)$ and changing variables, we get
\begin{equation}
			\label{eq8.6.2}
	n^{-1} \|(|U| + |\nabla_v U|) |\xi_n|\|_{ L_q  (\bR^7_T)  } \le  N (p, q) n^{ 4d (p-q)/(pq) - 1 } \||U| + |\nabla_v U|\|_{ L_p  (\bR^7_T)  }.
\end{equation}
Note that the $4d (p-q)/(pq) - 1 < 0$ due to \eqref{B.6.1},
and, hence, passing to the limit in \eqref{eq8.6.2} as $n \to \infty$, we prove that $U \in S_q (\bR^7_T)$.
Then, by the uniqueness part of Theorem \ref{theorem D.1} $(ii)$, we  conclude that $U \equiv u$.
The lemma is proved.
\end{proof}

\section{H\"older continuity of the extended leading  coefficients \texorpdfstring{$\mathbb{A}$}{} on the whole space}
																\label{Appendix E}

Invoke the notation of Section \ref{subsubsection 1.4.3}.

\subsection{Kinetic Fokker-Planck equation}
Here we show that   the coefficients $\mathbb{A}$
defined by \eqref{eq1.4.3.2} - \eqref{eq1.20}, with $a = I_3$, are of class $C^{1/3, 1}_{x, v} (\bR^6)$.
Since $\mathbb{A} \in C^{1/3, 1}_{x, v} (\overline{\bH_{\pm}})$, we only need to check the continuity
on the  set $\{y_3  = 0\} \times \bR^3$.

Note that for $a = I_3$, by \eqref{eq1.4.3.2}
$$
	A  =  \bigg(\frac{\partial y}{\partial x}\bigg) \bigg(\frac{\partial y}{\partial x}\bigg)^T.
$$
By \eqref{eq1.21}, \eqref{eq1.23} and the fact that $A$ is independent of $w$,  it suffices to show that
\begin{equation}
			\label{eqE.0}
	A^{i 3} (t, y_1, y_2, 0) = 0, \, \,  i = 1, 2,
\end{equation}
By using \eqref{eq1.4.2.8}, we compute
$$
	\frac{\partial x}{\partial y} = \begin{pmatrix}
						 1-y_3 \rho_{1 1}  &  - y_3 \rho_{12} & - \rho_1\\
						-y_3 \rho_{1 2}  & 1-y_3 \rho_{22} & - \rho_2\\
						\rho_1	&		\rho_2		& 1
						\end{pmatrix},
$$
where $\rho_{i j}$ is the second-order partial derivatives with respect to $y_i y_j$  variables.
Hence,
\begin{align}
			\label{eqE.1}
	& A^{-1} (y_1, y_2, 0) = (\frac{\partial x}{\partial y}|_{y_3 = 0})^T \frac{\partial x}{\partial y}|_{y_3 = 0}\\
		&				= \begin{pmatrix}
						 1  &  0&  \rho_1\\
						0  &  1 &  \rho_2\\
						-\rho_1	&		-\rho_2		& 1\\
						\end{pmatrix}
						 \begin{pmatrix}
						 1  &  0& - \rho_1\\
						0  &  1 & - \rho_2\\
						\rho_1	&		\rho_2		& 1\\
						\end{pmatrix}
						= \begin{pmatrix}
						 1   + \rho_1^2&  \rho_1 \rho_2 &  0\\
						\rho_1 \rho_2  & 1+ \rho^2_{2} &  0\\
						0	&		0		&1+\rho_1^2+\rho_2^2
						\end{pmatrix}\notag.
\end{align}
It follows from the  formula for the matrix inverse that the condition \eqref{eqE.0} is satisfied.
Thus, the desired regularity holds.

\subsection{Linearized Landau equation}
Let $A, \mathcal{A}, \mathbb{A}$ be the functions defined by \eqref{eq1.4.3.2}, \eqref{eq1.4.3.3}, \eqref{eq1.20} with $a = \sigma_G$.

In this subsection, will  show that, under Assumption \ref{assumption 11.1} (see \eqref{con11.1} - \eqref{con11.1'}),
\begin{itemize}
 \item[--] $\mathbb{A} \in L_{\infty} ((0, T) \times C^{\varkappa/3, \varkappa}_{x, v} (\bR^6))$.

\item[--] for any  radially symmetric $\xi \in C^{\infty}_0 (\bR^3)$,
$(t, y, w) \to \xi (\frac{\partial x}{\partial y}w) \mathbb{A} (t, y, w) \in L_{\infty} ((0, T), C^{\varkappa/3, \varkappa}_{x, v} (\bR^6))$.
\end{itemize}
Due to Lemma \ref{lemma C.1}, $\sigma_G \in L_{\infty} ((0, T) \times C^{\varkappa/3, \varkappa}_{x, v} (\bR^6))$,
and, hence,
$$\mathbb{A} \in L_{\infty} ((0, T), C^{\varkappa/3, \varkappa}_{x, v} (\overline{\mathbb{H}_{\pm}})).$$
Then, we only need to show that, for any $t \in (0, T)$,  the functions  $(y, w) \to \mathbb{A} (t, y, w)$ and $(y, w) \to \xi (\frac{\partial x}{\partial y}w)  \mathbb{A} (t, y, w)$
are continuous on the set $\{y_3 = 0\} \times \bR^3$.
These assertions will follow directly from the next lemma and corollary.

\begin{lemma}
			\label{lemma E.1}
Let $u$ be a function on $\Sigma^T$ satisfying the specular reflection boundary condition
and such that the convolution $U$ defined below makes sense. Denote
\begin{align*}
	& U^{i j} (z) = \Phi^{i j}  \ast u (z),\\
&
\mathfrak{U} (t, y, w)
=      \bigg(\frac{\partial y}{\partial x}\bigg) \widehat U (t, y, w) \bigg(\frac{\partial y}{\partial x}\bigg)^T,
\end{align*}
 where $\widehat U$ is defined by \eqref{eq1.4.2.1} with $U$ in place of $u$.
Then,
\begin{equation}
			\label{eqE.3}
	\mathfrak{U}^{i 3} (t, y_1, y_2, 0, w) =  - \mathfrak{U}^{i 3} (t, y_1, y_2, 0, R w), i = 1, 2.
\end{equation}
\end{lemma}

\begin{proof}
Denote
$$
\Xi (y, w) = \bigg(\frac{\partial y}{\partial x}\bigg) \widehat \Phi (y, w) \bigg(\frac{\partial y}{\partial x}\bigg)^T,
$$
where $\Phi$ is the Landau kernel (see \eqref{eq11.10}).
Then, by the change of variables $v' = \frac{\partial y}{\partial x} w'$,
$$
	\mathfrak{U} (t, y, w) = \bigg|\text{det} \frac{\partial y}{\partial x}\bigg|  \int_{\bR^3} \Xi (y, w') \widehat u (t, y,  w-w') \, dw',
$$
and then, by changing variables $w' \to Rw'$, we obtain
$$
	\mathfrak{U} (t, y, R w) = \bigg|\text{det} \frac{\partial y}{\partial x}\bigg| \int_{\bR^3}  \Xi (y, Rw') \widehat u (t, y, R (w-w')) \, dw'.
$$
Since $u$ satisfies the specular reflection boundary condition, we have, for any $t \in (0, T)$, $y_1, y_2 \in \bR$, $w \in \bR^3$,
$$
	\widehat u (t, y_1, y_2, 0,  w) = \widehat u (t, y_1, y_2, 0, R w).
$$
Thus, it suffices to show that \eqref{eqE.3}  holds with $\mathfrak{U}$ replaced with $\Xi$.

 By direct computations,
$$
	\widehat \Phi (y, w) =  \bigg|\frac{\partial x}{\partial y} w\bigg|^{-1} I_3  - \bigg|\frac{\partial x}{\partial y} w\bigg|^{-3}   \frac{\partial x}{\partial y} w w^T \bigg(\frac{\partial x}{\partial y}\bigg)^Tg,
$$
and, therefore,
$$
	 \Xi (y, w) =   \bigg|\frac{\partial x}{\partial y} w\bigg|^{-1} \frac{\partial y}{\partial x}\bigg(\frac{\partial y}{\partial x}\bigg)^T  -   \bigg|\frac{\partial x}{\partial y} w\bigg|^{-3}  w w^T.
$$
Observe  that \eqref{eqE.3} trivially holds with $\mathfrak{U}$ replaced with $w w^T$, and the matrix $ \frac{\partial y}{\partial x}(\frac{\partial y}{\partial x})^T$ satisfies the condition \eqref{eqE.0}.
Hence, we only need to show that the function
$$
	V (y, w) :=  \bigg|\frac{\partial x}{\partial y} w\bigg|
$$
satisfies the specular reflection boundary condition
\begin{equation}
			\label{eqE.5}
	V (y_1, y_2, 0,  w)  = V (y_1, y_2, 0, R w).
\end{equation}
Note that
$$
	|V (y, R w)|^2 = C^{i j} (y) (R w)_i  (R w)_j, \quad C (y) =  \bigg(\frac{\partial x}{\partial y}\bigg)^T \frac{\partial x}{\partial y}.
$$
Note that when $y_3 = 0$, the matrix $C (y)$ is given by the right-hand side of \eqref{eqE.1}.
Hence, since  \eqref{eqE.0} holds with $A$ replaced with $C$, the condition \eqref{eqE.5} is valid.
The lemma is proved.
\end{proof}

The above lemma holds if we cutoff $U$ by a radially symmetric function depending only on $v$.
\begin{corollary}
Let $\xi  = \xi (v) \in C^{\infty}_0 (\bR^3)$ be a radially symmetric function.
Then, the condition \eqref{eqE.3} holds with $	\mathfrak{U}$ is replaced with the matrix-valued function
$$
(t, y, w) \to \xi (\frac{\partial x}{\partial y}w) \mathfrak{U} (t, y, w).
$$
\end{corollary}

\begin{proof}
The assertion follows from Lemma \ref{lemma E.1} and the fact that
due to radial symmetry and \eqref{eqE.5},   the function $(y, w) \to \xi (\frac{\partial x}{\partial y}w)$ satisfies the specular reflection boundary condition on $\bH^T_{-}$.
\end{proof}

\textbf{Acknowledgments.} We thank Dr. Zhimeng Ouyang for her work on the setup of the problem, attending weekly group meeting and participating in discussion,
and her helpful comments to improve the presentation of the paper.

\end{document}